\documentclass[11pt]{article}
\usepackage{amsmath}
\usepackage{amsthm}
\usepackage{amssymb}
\usepackage{graphicx}
\usepackage{mathtools}
\usepackage[margin=2cm]{geometry}
\usepackage{setspace}
\usepackage{epstopdf}
\usepackage{subfig}
\usepackage{empheq}
\usepackage{algorithm}
\usepackage{algorithmicx}
\usepackage{algpseudocode}
\usepackage{multirow}
\usepackage{url}
\usepackage[title]{appendix}
\usepackage{enumitem}
\usepackage[svgnames]{xcolor}

\makeatletter

\theoremstyle{plain}
\newtheorem{thm}{\protect\theoremname}
  \theoremstyle{definition}
  
  \newtheorem{cond}[thm]{\protect\condname}
  \theoremstyle{plain}
  \newtheorem{lem}[thm]{\protect\lemmaname}
  \newtheorem{prop}[thm]{\protect\propname}
  
  \newtheorem{prob}{\protect\probname}
  
  \providecommand{\propname}{Proposition}
  \providecommand{\condname}{Condition}
  \providecommand{\probname}{Problem}

\makeatother

  \providecommand{\definitionname}{Definition}
  \providecommand{\lemmaname}{Lemma}
  \providecommand{\corollaryname}{Corollary}
\providecommand{\theoremname}{Theorem}

\usepackage{endnotes}
\usepackage{color}
\usepackage{xparse}

\newtheorem{remark}{Remark}

\begin{document}

\title{Multi-Reference Factor Analysis: low-rank covariance estimation under unknown translations}

\author{Boris Landa and Yoel Shkolnisky}









\author{ Boris Landa${^{1,*}}$~~~~Yoel Shkolnisky${^{2}}$\\
\small{${^1}$Program in Applied Mathematics, Yale University}\\
\small{${^2}$Department of Applied Mathematics, School of Mathematical Sciences, Tel-Aviv University}\\
\small{${^*}$Corresponding author. Email: boris.landa@yale.edu}
}

\maketitle

\begin{abstract}
{We consider the problem of estimating the covariance matrix of a random signal observed through unknown translations (modeled by cyclic shifts) and corrupted by noise. Solving this problem allows to discover low-rank structures masked by the existence of translations (which act as nuisance parameters), with direct application to Principal Components Analysis (PCA). We assume that the underlying signal is of length $L$ and follows a standard factor model with mean zero and $r$ normally-distributed factors. To recover the covariance matrix in this case, we propose to employ the second- and fourth-order shift-invariant moments of the signal known as the \textit{power spectrum} and the \textit{trispectrum}. We prove that they are sufficient for recovering the covariance matrix (under a certain technical condition) when $r<\sqrt{L}$. Correspondingly, we provide a polynomial-time procedure for estimating the covariance matrix from many (translated and noisy) observations, where no explicit knowledge of $r$ is required, and prove the procedure's statistical consistency. 
While our results establish that covariance estimation is possible from the {power spectrum} and the {trispectrum} for low-rank covariance matrices, we prove that this is not the case for full-rank covariance matrices.
We conduct numerical experiments that corroborate our theoretical findings, and demonstrate the favorable performance of our algorithms in various settings, including in high levels of noise.}{}
\end{abstract}

\section{Introduction} \label{section:introduction}
Principal Components Analysis (PCA) is a ubiquitous technique in science and engineering, which is used extensively for processing and analyzing large datasets. A standard approach for PCA is to estimate the covariance matrix of the dataset, compute its eigen-decomposition, and then project the data points onto the first several leading eigenvectors (i.e. with the largest corresponding eigenvalues). In various scientific applications, PCA is applied to collections of one-dimensional signals, where the underlying assumption is that these signals are low-rank, in the sense that they reside on (or near) a low-dimensional linear subspace. However, it is often the case that real-world signal measurements are prone to certain group-action deformations, where a common example is that of translations. When different translations are applied to low-rank signals, the resulting covariance matrix loses its low-rank structure (a claim which is made more precise shortly), thus rendering PCA ineffective without first aligning the signals. 
This scenario, where signals are acquired through unknown translations, is encountered for example in radar target classification~\cite{kim2002efficient,zwart2003fast,zyweck1996radar}, chromotographic fingerprinting~\cite{daszykowski2010automated,malmquist1994alignment,tomasi2004correlation}, machine fault diagnosis~\cite{garcia2012fault,liu2011adaptive}, and ECG signal classification~\cite{irvine2008eigenpulse,israel2005ecg}. In these applications, a typical scenario includes collecting a large dataset of signals for analysis and classification, followed by PCA for denoising and dimensionality reduction. For PCA to be effective, the dataset's covariance matrix should be approximately low-rank. Hence, to account for the different translations it is customary to first align the signals in the dataset. Numerous methods exist for the task of signal alignment, where standard approaches include pair-wise registration, and matching with a predefined template. Yet, it is important to stress that if the signals admit significant heterogeneity (i.e. inherent variability not associated with noise) then alignment is not well-defined, as the concept of aligning two very different patterns is meaningless. Moreover, signal alignment -- even between identical copies -- cannot be achieved in high levels of noise~\cite{bendory2017bispectrum,weiss1983fundamental}. Motivated by the above-mentioned limitations of signal alignment, we consider the problem of accurately estimating the covariance matrix of low-rank signals from their translated and noisy observations. 

\subsection{The setting}
We consider the following model for an observed signal $y\in \mathbb{C}^L$:
\begin{equation}
y = R_s\left\{x\right\} + \eta,  \label{eq:MRFA model def}
\end{equation}
where $x\in\mathbb{C}^L$ is the underlying signal (to be described shortly), $\eta$ is a noise vector with either $\eta\sim\mathcal{N}(0,\sigma^2 I_L)$ or $\eta\sim\mathbb{C}\mathcal{N}(0,\sigma^2 I_L)$ ($\mathbb{C}\mathcal{N}$ stands for the circularly-symmetric complex normal distribution, $I_L$ is an $L\times L$ identity matrix), and $R_s\left\{\cdot\right\}$ is a discrete cyclic shift by $s$, i.e.
\begin{equation} \label{eq:cyclic shift def}
R_s\left\{x\right\} [\ell] = x\left[\operatorname{mod}{(\ell-s,L)}\right],
\end{equation} 
with $s$ drawn from some unknown probability distribution over $\mathbb{Z}_L$. In what follows, we consider all vectors as cyclic, and drop the modulus by $L$ from our notation in all index assignments. 
The underlying signal $x$ is assumed to follow a standard zero-mean factor model 
\begin{equation}
x = \sum_{i=1}^r a_i v_i,	\label{eq:x model}
\end{equation}
where $\{ v_i \}_{i=1}^r\in \mathbb{C}^L$ are orthonormal, and $\{a_i\}_{i=1}^r$ are i.i.d with either $a_i \sim \mathcal{N}(0,\lambda_i)$ (if $\eta\sim\mathcal{N}(0,\sigma^2 I_L)$) or $a_i \sim \mathbb{C}\mathcal{N}(0,\lambda_i)$ (if $\eta\sim\mathbb{C}\mathcal{N}(0,\sigma^2 I_L)$). We mention that while the theoretical analysis in this work focuses on the complex-valued case (where $a_i \sim \mathbb{C}\mathcal{N}(0,\lambda_i)$ and $\eta\sim \mathbb{C}\mathcal{N}(0,\sigma^2 I_L)$), for practical purposes we also consider the real-valued case ($v_i\in \mathbb{R}^L$, $a_i \sim \mathcal{N}(0,\lambda_i)$, and $\eta \sim \mathcal{N}(0,\sigma^2 I_L)$), providing appropriate modifications to our algorithms (see Section~\ref{section:step 1}).
Now, given~\eqref{eq:x model}, the covariance matrix of $x$ is 
\begin{equation}
\Sigma_x := \mathbb{E}\left[ x x^* \right] = \sum_{i=1}^r \lambda_i v_i v_i^*,
\end{equation}
where $(\cdot)^*$ stands for complex-conjugate and transpose. 
While the rank of $\Sigma_x$ is $r$ and can be considerably smaller than $L$, the covariance matrix of $y$, given by $\Sigma_y := \mathbb{E}\left[ y y^* \right]$, typically admits a rank much larger than $r$, even if no noise is present ($\sigma=0$). This is because the rank of $\Sigma_y$ is dominated by the dimension of the set of vectors $\left\{{R}_{s}\{v_i\}\right\}_{i,s}$ for $i=1,\ldots,r$ and $s \in S$ (where $S\subset\mathbb{Z}_L$ is a set of allowed shifts), which can exceed $r$ significantly. In particular, if the probability distribution of $s$ is non-vanishing (i.e. all shifts are allowed; $S = \mathbb{Z}_L$) and $v_i$ is aperiodic for some $i$ (see~\cite{abbe2017multireference}), then $\Sigma_y$ is full-rank. 

Considering the setting of~\eqref{eq:MRFA model def}--\eqref{eq:x model}, a fundamental problem of interest is the following one.
\begin{prob}[Multi-Reference Factor Analysis] \label{prob:MRFA}
Given $N$ i.i.d measurements $y_1,\ldots,y_N$ following the model~\eqref{eq:MRFA model def}--\eqref{eq:x model}, estimate $\lambda_1,\ldots,\lambda_r$ and $v_1,\ldots,v_r$.
\end{prob}
Problem~\ref{prob:MRFA}, termed \textit{Multi-Reference Factor Analysis} (MRFA), can be viewed as a generalization of standard factor analysis, to the setting of unknown translations of the underlying signal vectors. 
In this work, instead of estimating $\lambda_1,\ldots,\lambda_r$ and $v_1,\ldots,v_r$ directly, we consider the closely-related problem of estimating the covariance matrix $\Sigma_x$, whose eigenvalues and eigenvectors are $\{\lambda_i\}_{i=1}^r$ and $\{v_i\}_{i=1}^r$, respectively. 
We exemplify our setting in Figure~\ref{fig:setting example} for $r=3$, $L=50$, and $\sigma^2 = 0.01/L$. 

{
\begin{figure}[p]
  \centering
  	\subfloat[Eigenvectors of $\Sigma_x$: $v_1$, $v_2$, $v_3$] 
  	{
    \includegraphics[width=0.33\textwidth]{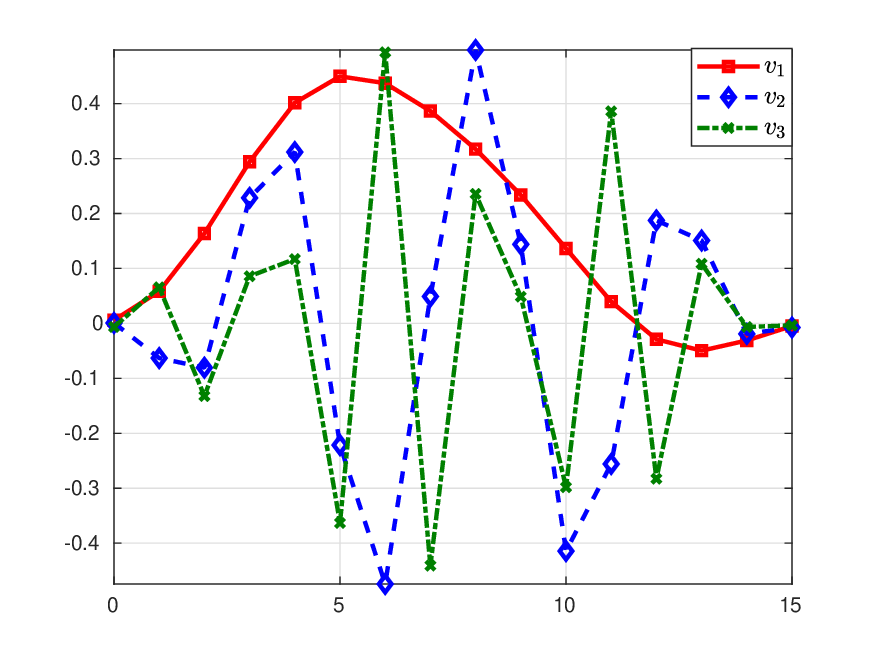} \label{fig:intro_example_eigenvecs}
    }
    \subfloat[Exemplar signal measurements]  
    { 
    \includegraphics[width=0.33\textwidth]{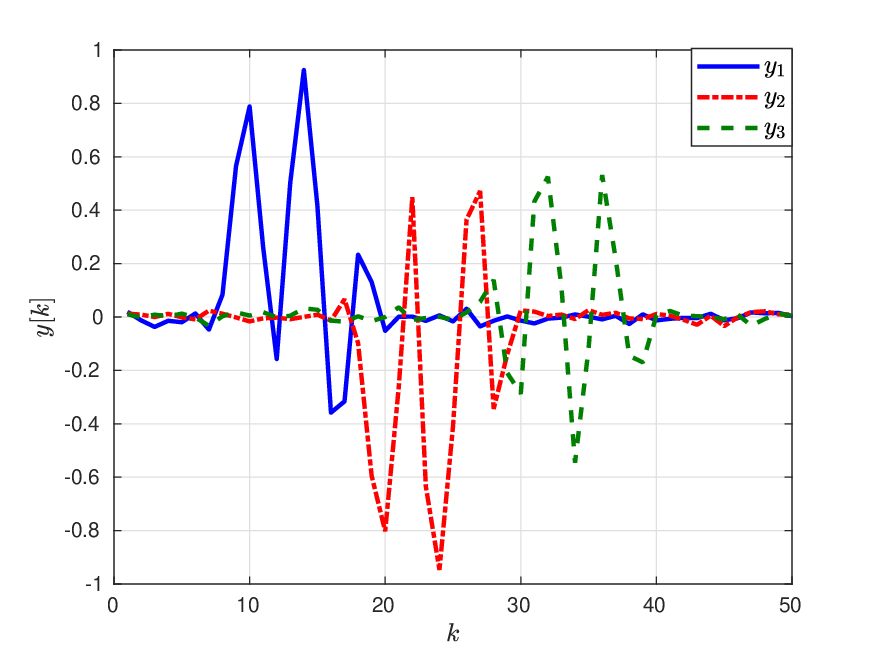} \label{fig:intro_example_signals}
    }
    \subfloat[Eigenvalues of $\Sigma_y$]
    {
    \includegraphics[width=0.33\textwidth]{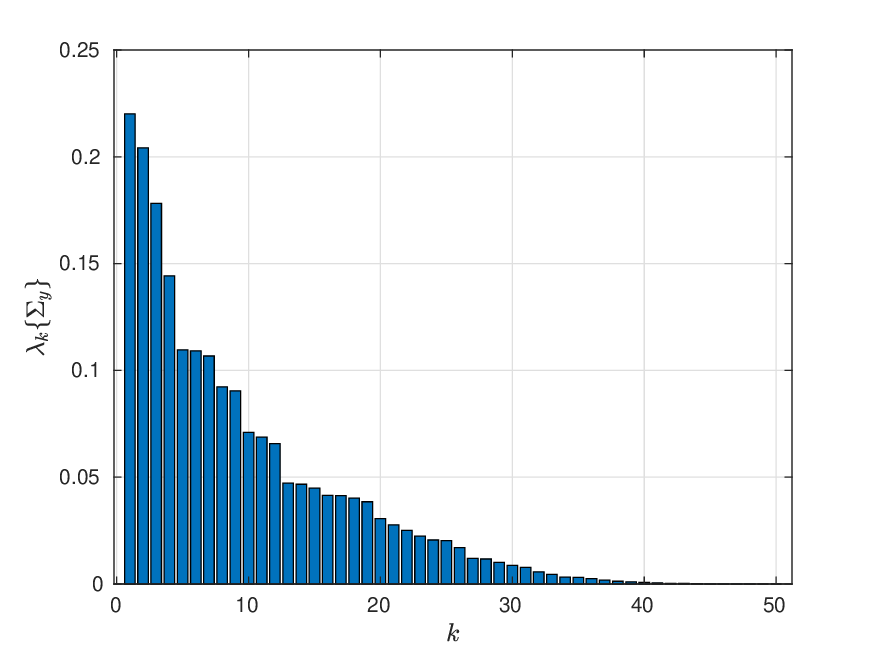} \label{fig:intro_example_eigenvals}
    }
	\caption[setting example] 
	{A prototypical example of our setting for $r=3$, $L=50$, and $\sigma^2=0.01/L$. We set the eigenvalues to $\lambda_1 = 1$, $\lambda_2 = 0.7$, $\lambda_3 = 0.3$, and the eigenvectors $v_1$, $v_2$, $v_3$ (leftmost figure) to the discrete local-cosine basis functions (see~\cite{aharoni1993local} equation (24)) of orders $0$, $3$, $6$, respectively, supported on $16$ samples. We fixed the distribution for the cyclic shift $s$ (see~\eqref{eq:cyclic shift def}) to be uniform over $\{0,\ldots,L-16\}$, and zero otherwise. Typical observations of $y$ can be seen in the central figure, and the eigenvalues of the resulting covariance matrix of $y$ are shown in the rightmost figure. It is evident that different observations of $y$ are dissimilar, precluding the possibility of a straightforward alignment between them. Furthermore, while the rank of $\Sigma_x$ is $3$, $\Sigma_y$ is full-rank and exhibits a slow decay of its eigenvalues.}  \label{fig:setting example}
\end{figure}

\begin{figure}
  \centering
  	\subfloat[Estimated eigenvalues $\widetilde{\lambda}_1,\ldots,\widetilde{\lambda}_L$] 
  	{
    \includegraphics[width=0.33\textwidth]{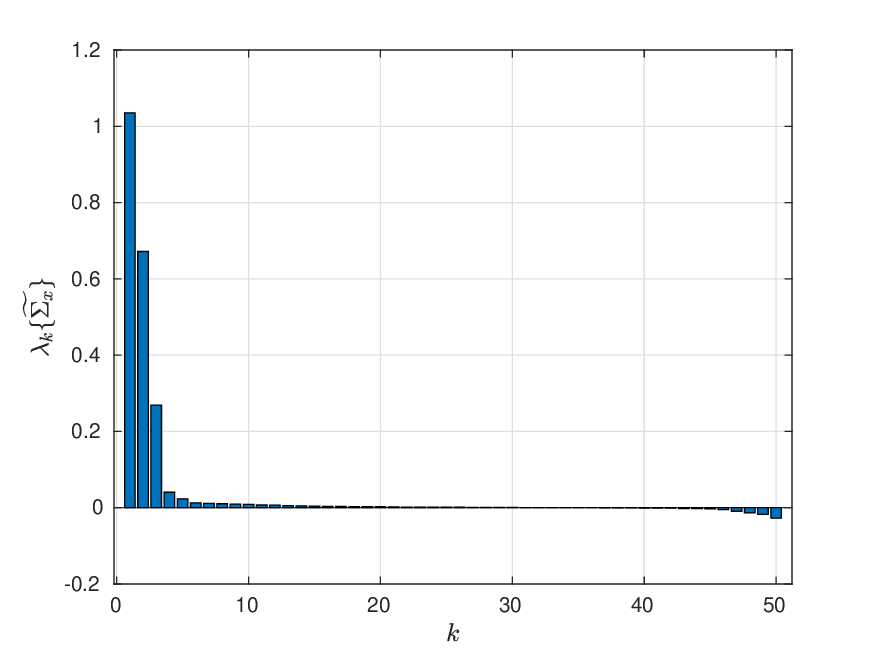} \label{fig:intro_example_eigenvecs}
    } 
    \subfloat[$v_1$ versus $\widetilde{v_1}$]  
    { 
    \includegraphics[width=0.33\textwidth]{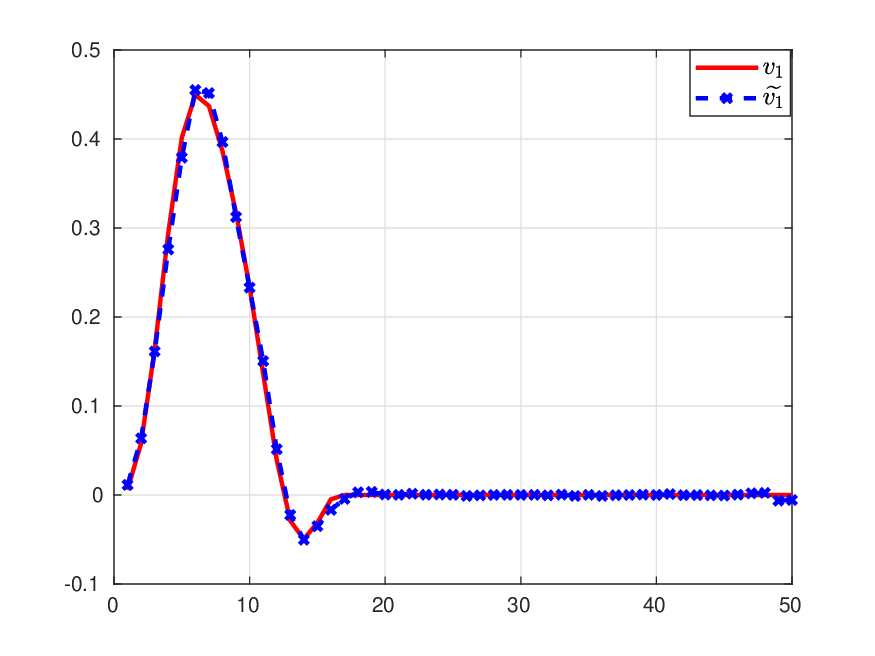} \label{fig:intro_example_signals}
    } \\
    \subfloat[$v_2$ versus $\widetilde{v_2}$]
    {
    \includegraphics[width=0.33\textwidth]{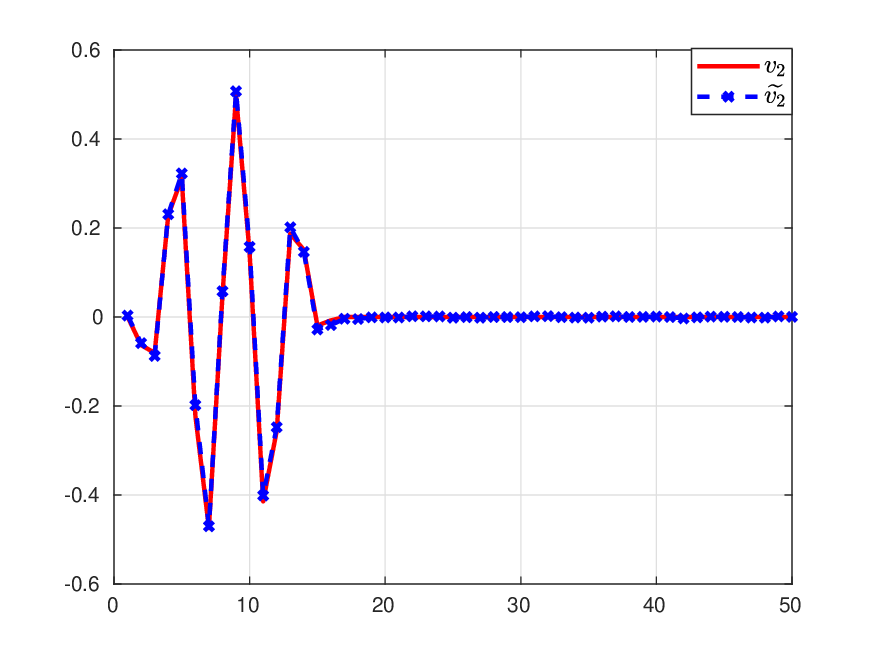} \label{fig:intro_example_eigenvals}
    }
    \subfloat[$v_3$ versus $\widetilde{v_3}$]
    {
    \includegraphics[width=0.33\textwidth]{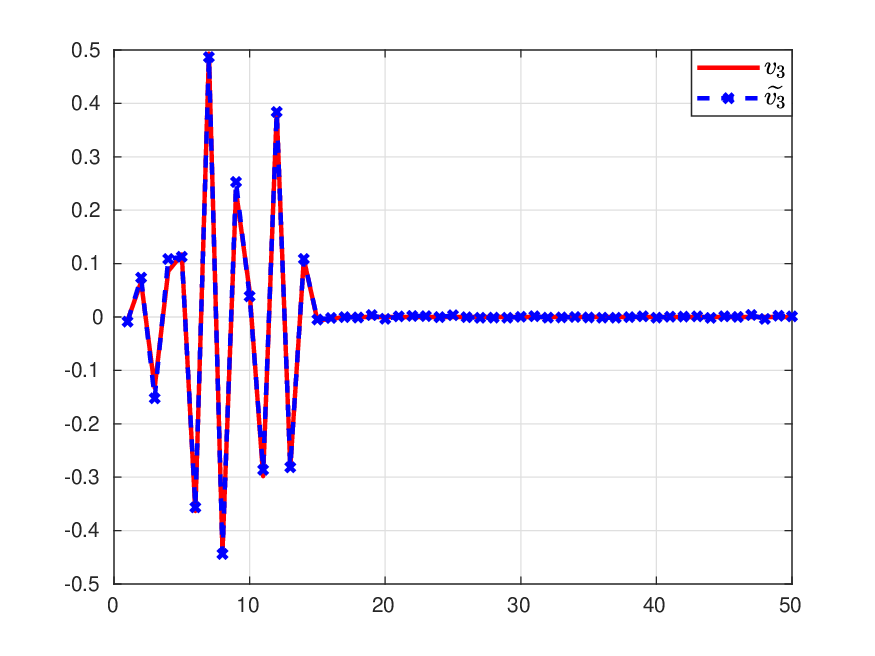} \label{fig:intro_example_eigenvals}
    }
	\caption[setting example estimate] 
	{Results of applying our algorithms (Algorithm~\ref{alg:step 1 alg} followed by Algorithm~\ref{alg:step 2 direct}) to $N=1000$ measurements of $y$, with the setting of Figure~\ref{fig:setting example}. The figure in the top-left depicts the eigenvalues obtained from the estimated covariance matrix $\widetilde{\Sigma}_x$ (an estimate of $\Sigma_x$ computed by Algorithm~\ref{alg:step 2 direct}, see Section~\ref{section: step 2}), and the figures in the top-right, bottom-left, and bottom-right, compare between the first three eigenvectors of $\widetilde{\Sigma}_x$ to $v_1$, $v_2$, $v_3$, respectively. It is evident that the first three eigenvalues of $\widetilde{\Sigma}_x$ are indeed close to $\lambda_1=1$, $\lambda_2=0.7$, $\lambda_3=0.3$, respectively, while the corresponding eigenvectors of $\widetilde{\Sigma}_x$ match $v_1$, $v_2$, and $v_3$, almost exactly. }   \label{fig:setting example estimate}
\end{figure}
}

\subsection{Related work}
As far as we know, the MRFA problem (Problem~\ref{prob:MRFA}) as presented here has not been treated in the literature. However, it is worthwhile to review some closely related problems, and the approaches undertaken to solve them. When the signal $x$ in~\eqref{eq:MRFA model def} is deterministic and fixed, our setting becomes that of Multi-Reference Alignment (MRA)~\cite{abbe2017sample, bandeira2014multireference, bendory2017bispectrum, perry2017sample,chen2018spectral,bandeira2017optimal}, where an unknown vector is to be recovered from its translated and noisy observations. Since $x$ is fixed in MRA, its translated copies can be accurately aligned when the Signal-to-Noise Ratio (SNR) is sufficiently high. Therefore, the line of work in MRA mostly focuses on the low-SNR regime, where alignment is impossible.
Signal estimation under group actions other than translations has recently been considered in~\cite{bandeira2017estimation,abbe2018estimation}. 

In the majority of the above-mentioned works on MRA, signal estimation is carried out through an instance of the Method of Moments~\cite{van2000asymptotic} (see also the generalized method of moments~\cite{hansen1982large}). Specifically, certain shift-invariant moments of the underlying signal are identified and estimated, and then employed for recovering the underlying signal from the emerging moment equations. Shift-invariant moments which are extensively used in this context are the signal's mean, the signal's power spectrum, and the signal's bispectrum~\cite{collis1998higher}. The power spectrum and the bispectrum of a signal essentially correspond to certain double and triple correlations, respectively, in the Fourier domain of the signal, which are invariant to cyclic shifts of the signal. It is important to point out that the power spectrum and the bispectrum are omnipresent in classical statistical signal processing, see for example~\cite{nikias1987bispectrum}, with an endless list of applications. 
In the context of MRA, knowing the bispectrum is sufficient to recover the signal up to fundamental ambiguities (arbitrary cyclic shifts of the underlying signal)~\cite{bendory2017bispectrum}. The algorithmic use of the bispectrum (often coupled with other shift-invariant moments for improved performance) for estimating the signal in MRA was investigated in~\cite{bendory2017bispectrum,perry2017sample}, where different algorithms were presented for this task. Particularly, much focus was put on the required sample complexity, namely the number of samples $N$ required to achieve a prescribed estimation error for a given SNR value. 
We also mention~\cite{abbe2017multireference}, where it was shown that the second moment of the measurements $y$ in MRA is sufficient to recover the signal $x$ if the distribution of the shifts is aperiodic, resulting in an improved sample complexity rate. 
Aside from approaches leveraging shift-invariant moments, another widespread approach for solving estimation problems akin to MRA (or more generally - estimation problems involving nuisance parameters) is Expectation Maximization (EM)~\cite{dempster1977maximum}. While EM is popular and intuitive, its crucial drawback is lack of theoretical convergence guarantees. On the other hand, methods based on shift-invariant moments are amenable to rigorous analysis, allowing for algorithms with provable theoretical guarantees. Furthermore, methods based on invariant moments lead to single-pass algorithms, in the sense that every measurement $y_i$ is only considered once, with subsequent processing involving only the estimated moments.

We also mention that several recent works extend the standard model of MRA (where $x$ is fixed) to a scenario where $x$ assumes discrete heterogeneity~\cite{boumal2017heterogeneous,ma2018heterogeneous,perry2017sample}, namely that $x$ is chosen from a finite set of templates $x_1,\ldots,x_k$.

Another closely related line of work is that of Generalized Principal Components Analysis (GPCA) (also known as \textit{subspace clustering})~\cite{vidal2005generalized,vidal2011subspace,liu2010robust}, where data points are assumed to reside on a union of low-dimensional linear subspaces. The MRFA problem could be considered as a special case of GPCA, where each subspace is described by a different translation of the subspace spanned by $\{v_i\}_{i=1}^r$. However, in GPCA no relation between the different subspaces is assumed, and consequently, the theoretical recovery guarantees for subspace recovery are typically prohibitive if the noise variance and the number of subspaces are large.

Last, we mention~\cite{aizenbud2019rank}, where the authors considered the problem of MRFA in the restricted case of $r=1$. A key observation in~\cite{aizenbud2019rank} is that if the distribution of the shifts is uniform, then the bispectrum vanishes entirely. Therefore, it was proposed to employ the fourth-order shift-invariant moment known as the trispectrum~\cite{collis1998higher} to estimate the signal parameters ($\lambda_1$ and $v_1$), and algorithms were presented for consistently estimating $\lambda_1$ and $v_1$ as $N\rightarrow\infty$. The trispectrum is analogous to the power spectrum and the bispectrum, in the sense that it consists of certain quadruple correlations in the Fourier domain of the signal, which are invariant to cyclic shifts. As in MRA, much focus was put in~\cite{aizenbud2019rank} on the sample complexity of the algorithms in the low-SNR regime, since in the high-SNR regime the observations can be accurately aligned using correlations (a fact which is only true for the case $r=1$).

\subsection{Our contributions and main results}
We consider the problem of estimating the covariance matrix $\Sigma_x$ from the measurements $y_1,\ldots,y_N$ (generated according to the model~\eqref{eq:MRFA model def}--\eqref{eq:x model}) using shift-invariant moments, where we assume no explicit knowledge of the rank $r$. We investigate certain theoretical aspects of uniqueness and identifiability, derive practical algorithms for consistently estimating $\Sigma_x$ from $y_1,\ldots,y_N$, and conclude with extensive simulations. 

We next describe our contributions in detail. Unless otherwise stated, we refer the complex-valued case of our setting, where $a_i \sim \mathbb{C}\mathcal{N}(0,\lambda_i)$ and $\eta\sim \mathbb{C}\mathcal{N}(0,\sigma^2 I_L)$.

\subsubsection{Characterization of uniqueness and identifiability}
We begin by investigating the moments equations arising from the power spectrum and the trispectrum, and consider the question of whether these equations are sufficient to recover $\Sigma_x$, and under which conditions. The results of this investigation are reported in Section~\ref{section:invariant moments}, where we show the following. By posing an equivalent formulation to our problem in the Fourier domain, replacing $\Sigma_x$ with its analogue in the Fourier domain $\hat{\Sigma}_x$, we show that the algebraic structure of the moments equations (arising from the power spectrum and the trispectrum) determines $\hat{\Sigma}_x$ completely up to multiplication (element-wise) with an unknown circulant matrix of phases (namely a circulant matrix with unit-magnitude elements). Essentially, this multiplication with a circulant matrix of phases corresponds to phase uncertainties on the diagonals of $\hat{\Sigma}_x$. We then consider the problem of resolving these uncertainties, a problem which we term \textit{circulant phase retrieval} (see Problem~\ref{problem:Sigma_x diagonal phase retrieval problem}). We show that by leveraging the Hermitian and positive semidefinite (PSD) structure of $\hat{\Sigma}_x$, it is possible to solve the circulant phase retrieval problem and to recover $\hat{\Sigma}_x$ (up to certain fundamental ambiguities, see Proposition~\ref{prop:fundamental ambiguities}), whenever $r<\sqrt{L}$ and a certain technical condition holds (Condition~\ref{cond:low-rank recovery} in Section~\ref{section: step 2}). We note that this technical condition was observed to hold in all conducted numerical experiments, suggesting that it is non-restrictive in practice. While our main result asserts that recovery of a low-rank $\Sigma_x$ is possible from the power spectrum and the trispectrum alone (see Theorem~\ref{thm:low-rank recovery}), we show that this not the case for a full-rank $\Sigma_x$ (see Proposition~\ref{prop:P_y and T_y ambiguity full rank}).

\subsubsection{Practical estimation procedures with theoretical guarantees}
In Section~\ref{section:recovering the covariance}, we describe a statistically-consistent procedure for estimating $\Sigma_x$ from finite-sample estimates of the power spectrum and the trispectrum. This is essentially a two step procedure, where the first step is to estimate $\hat{\Sigma}_x$ up to the aforementioned ``diagonal phase ambiguities'', and the second step is to resolve them while exploiting the fact that $\hat{\Sigma}_x$ is Hermitian and PSD. 
The first step is derived in Section~\ref{section:step 1}, see Algorithm~\ref{alg:step 1 alg}, and consists of solving a convex optimization problem, followed by computing several rank-one decompositions. We prove that regardless of the rank $r$, this procedure consistently estimates $\hat{\Sigma}_x$ as $N\rightarrow\infty$ up to an element-wise multiplication with a circulant phase matrix, see Theorem~\ref{thm:consistency of step 1}. We also provide an appropriate modification of the above-mentioned procedure to handle the real-valued case ($v_i\in \mathbb{R}^L$, $a_i \sim \mathcal{N}(0,\lambda_i)$, and $\eta \sim \mathcal{N}(0,\sigma^2 I_L)$), where the difference lies only in the objective function of the convex optimization problem. We remark that only the first step of our two-step procedure needs to be modified to handle the real-valued case. The second step of the recovery process, i.e. resolving the diagonal ambiguities, is derived in Section~\ref{section: step 2}. Our main contribution in this context, is a polynomial-time procedure for solving the circulant phase retrieval problem (Problem~\ref{problem:Sigma_x diagonal phase retrieval problem} in Section~\ref{section:invariant moments}) when $r<\sqrt{L}$ and Condition~\ref{cond:low-rank recovery} holds, see Algorithm~\ref{alg:step 2 direct} in Section~\ref{section: step 2}. Fundamentally, this procedure begins by constructing a certain matrix from the output of step one, and proceeds with evaluating its right singular vector corresponding to its smallest singular value. When combined with Algorithm~\ref{alg:step 1 alg}, Algorithm~\ref{alg:step 2 direct} allows for a consistent estimate of $\Sigma_x$ as $N\rightarrow\infty$, see Theorem~\ref{thm:consistency of step 2, direct approach}. In Figure~\ref{fig:setting example estimate} we demonstrate the results of applying Algorithms~\ref{alg:step 1 alg} and~\ref{alg:step 2 direct} to estimate $\{\lambda_i\}$ and $\{v_i\}$ of Figure~\ref{fig:setting example}, for $N=1000$. 
Last, in Section~\ref{section:numerical examples} we conduct extensive numerical experiments, corroborating the statistical consistency of our estimators and moreover, demonstrating their favorable properties, such as a $1/N$ squared-error convergence rate and robustness to high levels of noise. 

\section{Invariant moments and identifiability of $\Sigma_x$} \label{section:invariant moments}

We start by introducing the shift-invariant moments used to recover $\Sigma_x$, and provide certain necessary and sufficient conditions for recovery.
Instead of working with the model~\eqref{eq:MRFA model def} directly, we consider a more convenient and equivalent formulation in the Fourier domain, where cyclic shifts are replaced by modulations. Let $F\in\mathbb{C}^{L\times L}$ be the unitary Discrete Fourier Transform (DFT) matrix, and let $f_\ell\in\mathbb{C}^L$ be the $\ell$'th DFT vector, given by
\begin{equation}
F[\ell,k] = \frac{1}{\sqrt{L}}f_\ell[k], \qquad f_\ell[k] = e^{-\imath 2\pi k \ell/L}, \qquad \ell,k = 0,\ldots,L-1. \label{eq:DFT matrix and vectors def}
\end{equation}
We denote the Fourier transforms of the quantities in~\eqref{eq:MRFA model def} and~\eqref{eq:x model} by
\begin{align}
\begin{aligned}
\hat{y} &= F y, \qquad \hat{x} = F x, \qquad \hat{\eta} = F \eta, \\
\hat{v}_i &= F v_i, \qquad i=1,\ldots,r.
\end{aligned}
\label{eq:Fourier domain quantities def}
\end{align}
Then, a formulation equivalent to~\eqref{eq:MRFA model def}--\eqref{eq:x model} in the Fourier domain is
\begin{align}
\begin{aligned}
\hat{y}[k] &= f_{s}[{k}] \hat{x}[k] + \hat{\eta}[k], \\ \hat{x}[k] &= \sum_{i=1}^r a_i \hat{v}_i[k],
\end{aligned} \label{eq:MRFA_model_def_Fourier}
\end{align}
for $k=0,\ldots,L-1$, where $\hat{\eta}\sim \mathbb{C}\mathcal{N}(0,\sigma^2 I_L)$, $\{ \hat{v}_i \}_{i=1}^r$ are orthonormal (since $F$ is unitary), and $s$ is the parameter of the cyclic shift from~\eqref{eq:MRFA model def}. Correspondingly, the covariance matrix of $\hat{x}$ is given by
\begin{equation}
\hat{\Sigma}_{{x}} := \mathbb{E}\left[ \hat{x} \hat{x}^*\right] = F \Sigma_{{x}} F^* = \sum_{i=1}^r \lambda_i \hat{v}_i \hat{v}_i^*, \label{eq:Sigma_x_hat def}
\end{equation}
which is Hermitian and positive semidefinite (PSD), with rank $r$ and with eigenvalues and eigenvectors $\{\lambda_i\}_{i=1}^r$ and $\{\hat{v}_i\}_{i=1}^r$, respectively. Clearly, knowing $\hat{\Sigma}_x$ is equivalent to knowing $\Sigma_x$, and so from this point onward we focus on the recovery of $\hat{\Sigma}_x$ instead of $\Sigma_x$.
Throughout this section, we assume the noiseless case, i.e. $\sigma = 0$, as the existence of noise simply adds a known bias term to $\Sigma_x$ which is easily removed (see Section~\ref{section:step 1} and Appendix~\ref{appendix:Justification of expression for P_y and T_y with diagonals}), and has no influence on the issue of identifiability of the solution.

Let us consider the second and fourth moments of $\hat{y}$, denoted $M^{(2)}_{\hat{y}}\in\mathbb{C}^{L\times L}$ and $M^{(4)}_{\hat{y}}\in\mathbb{C}^{L\times L\times L\times L}$ respectively, and given by
\begin{equation}
\begin{aligned}
M^{(2)}_{\hat{y}}[k_1,k_2] &= \mathbb{E} \left[ \hat{y}[k_1] \overline{\hat{y}[k_2]} \right], \\
M^{(4)}_{\hat{y}}[k_1,k_2,k_3,k_4] &= \mathbb{E} \left[ \hat{y}[k_1] \overline{\hat{y}[k_2]} \hat{y}[k_3] \overline{\hat{y}[k_4]} \right], 
\end{aligned} \label{eq:M_2 and M_4 def}
\end{equation}
for $k_1,k_2,k_3,k_3 \in \{0,\ldots,L-1\}$, where $\overline{(\cdot)}$ denotes complex-conjugation.
It is important to mention that all odd-ordered moments of $\hat{y}$ vanish, since the $a_i$'s of~\eqref{eq:x model} admit a zero-centered symmetric distribution. This explains why we only consider the second and fourth moments of $\hat{y}$, and not the first and third.
Next, we define the following subsets of $M^{(2)}_{\hat{y}}$ and $M^{(4)}_{\hat{y}}$:
\begin{align}
P_y[k] &:= M^{(2)}_{\hat{y}}[k,k], \label{eq:p_y def}\\
T_y[k_1,k_2,k_3] &:= M^{(4)}_{\hat{y}}[k_1,k_2,k_3,k_1-k_2+k_3]. \label{eq:T_y def}
\end{align}
$P_y$ is known as the \textit{power spectrum} of $y$, and $T_y$ is known as the \textit{trispectrum} of $y$~\cite{collis1998higher}. A fundamental property of $P_y$ and $T_y$ is that their entries are invariant to cyclic shifts of $y$ (or equivalently, to integer modulations of $\hat{y}$), regardless of the distribution of the cyclic shifts (this can be easily verified by substituting $\hat{y} = f_s[k] \hat{x}[k]$ into~\eqref{eq:p_y def} and~\eqref{eq:T_y def}, where $f_s[k]=e^{-\imath 2\pi ks/L}$ from~\eqref{eq:DFT matrix and vectors def}). Moreover, if $y$ admits uniformly distributed cyclic shifts, then most of the entries in $M^{(2)}_{\hat{y}}$ and $M^{(4)}_{\hat{y}}$ vanish, and the only non-zero entries of $M^{(2)}_{\hat{y}}$ and $M^{(4)}_{\hat{y}}$ are given by $P_y$ and $T_y$, respectively.

Since $P_y$ and $T_y$ are invariant to cyclic shifts in $y$, and as ${y} = R_s\{x\}$ (see~\eqref{eq:MRFA model def} with $\sigma=0$), $P_y$ and $T_y$ can be viewed as computed directly from $M^{(2)}_{\hat{x}}$ and $M^{(4)}_{\hat{x}}$ (defined by replacing $\hat{y}$ with $\hat{x}$ in~\eqref{eq:M_2 and M_4 def}) instead of $M^{(2)}_{\hat{y}}$ and $M^{(4)}_{\hat{y}}$, respectively. Furthermore, as $a_i$ is normally-distributed, all moments of $\hat{x}$ can be described in terms of its first and second moments, that is, its mean and covariance. Since the mean of $\hat{x}$ is zero, $P_y$ and $T_y$ can be described solely in terms of $\hat{\Sigma}_x$. In particular, we have the following proposition providing the explicit forms of ${P}_y$ and ${T}_y$.
\begin{prop}[Explicit form of ${P}_y$ and ${T}_y$] \label{prop:f_2,f_4,f_6 explicit form}
Consider the noiseless case (i.e. $\sigma = 0$). Then,
\begin{align}
{P}_y[k_1] &= \hat{\Sigma}_x[k_1,k_1], \label{eq: P_y expression Sigma} \\
{T}_y[k_1,k_2,k_3] &= \hat{\Sigma}_x[k_1,k_2]\cdot \overline{\hat{\Sigma}_x[k_3-k_2+k_1,k_3]} + \hat{\Sigma}_x[k_1,k_3-k_2+k_1] \cdot \overline{\hat{\Sigma}_x[k_2,k_3]}, \label{eq: T_y expression Sigma}
\end{align}
for all $k_1,k_2,k_3 \in \{0,\ldots,L-1\}$.
\end{prop}
The proof is provided in Appendix~\ref{Appendix: Proof of f_2,f_4 explicit form}. 

Assuming we have access to the shift-invariant moments $P_y$ and $T_y$, noting that they can be estimated in a straightforward manner from the observations of $y$ (see~\eqref{eq: P_y estimator} and~\eqref{eq: T_y estimator} in Section~\ref{section:step 1}), we turn to address the question of whether $\hat{\Sigma}_x$ can be identified, and under which conditions, from $P_y$ and $T_y$.
To that end, we consider the set of equations
\begin{align}
P_y = \mathcal{P}(X), \qquad
T_y = \mathcal{T}(X), \label{eq:P_y and T_y equations}
\end{align}
where $X\in\mathbb{C}^{L\times L}$ represents the unknown covariance matrix, and $\mathcal{P}:\mathbb{C}^{L\times L}\rightarrow \mathbb{C}^{L}$, $\mathcal{T}:\mathbb{C}^{L\times L}\rightarrow \mathbb{C}^{L\times L \times L}$ are maps encoding the relation between the underlying covariance $\hat{\Sigma}_x$ and the observed shift-invariant moments $P_y$ and $T_y$. Specifically, and in accordance with~\eqref{eq: P_y expression Sigma} and~\eqref{eq: T_y expression Sigma}, we define $\mathcal{P}$ and $\mathcal{T}$ as
\begin{align}
\mathcal{P}(X)[k_1] &:= X[k_1,k_1], \label{eq: P_mathcal expression} \\
\mathcal{T}(X)[k_1,k_2,k_3] &:= X[k_1,k_2]\cdot \overline{X[k_3-k_2+k_1,k_3]} + X[k_1,k_3-k_2+k_1] \cdot \overline{X[k_2,k_3]}, \label{eq: T_mathcal expression} 
\end{align}
for all $k_1,k_2,k_3 \in \{0,\ldots,L-1\}$. We mention that the domains of $\mathcal{P}$ and $\mathcal{T}$ are arbitrary $\mathbb{C}^{L\times L}$ matrices (instead of only Hermitian and PSD matrices) to simplify the analysis in this section. This simplification is achieved by first characterizing the solutions to~\eqref{eq:P_y and T_y equations} for arbitrary $\mathbb{C}^{L\times L}$ matrices, and then restricting our attention to the subset of solutions which are Hermitian and PSD. 

Given $P_y$ and $T_y$, \eqref{eq:P_y and T_y equations} corresponds to a set of non-linear equations which are to be solved to determine $\hat{\Sigma}_x$, where $\mathcal{P}(X)$ is a linear map in $X$, and $\mathcal{T}(X)$ is a quadratic map in $X$. According to~\eqref{eq: P_y expression Sigma}, $P_y$ is merely the main diagonal of $\hat{\Sigma}_x$, and hence insufficient for recovering $\hat{\Sigma}_x$. However, $T_y$ provides additional $L^3$ equations, which is more than the number of variables in a generic covariance matrix, and hence possibly enough for recovering $\hat{\Sigma}_x$. 
Let us denote by $\operatorname{Circulant}\{ z \}$ a circulant matrix constructed from a vector $z\in\mathbb{C}^L$, namely
\begin{equation}
\operatorname{Circulant}\{ z \} [k_1,k_2] = z[\operatorname{mod}(k_2-k_1,L)],
\end{equation}
for $k_1,k_2=0,\ldots,L-1$.
The following lemma characterizes the set of all $\mathbb{C}^{L\times L}$ matrices satisfying the equations~\eqref{eq:P_y and T_y equations}.
\begin{lem} [Solutions of the moments equations] \label{lem:P_y and T_y consequence}
Suppose that $\hat{\Sigma}_x[i,j] \neq 0$ for all $i,j$. Then, a matrix $X\in \mathbb{C}^{L\times L}$ satisfies the set of equations~\eqref{eq:P_y and T_y equations} if and only if 
\begin{equation}
X = \hat{\Sigma}_x \odot \operatorname{Circulant} \left\{[1, e^{\imath \varphi_1} ,\ldots, e^{\imath \varphi_{L-1}} ] \right\}, \label{eq:P_y and T_y circulant phase ambiguities}
\end{equation}
where $\varphi_1,\ldots,\varphi_{L-1}\in [0,2\pi)$ but are otherwise arbitrary, and $\odot$ is the Hadamard product (entry-wise multiplication).
\end{lem}
The proof of Lemma~\ref{lem:P_y and T_y consequence} is provided in Appendix~\ref{Appendix: Proof of Lemma P_y and T_y consequence}. Essentially, Lemma~\ref{lem:P_y and T_y consequence} asserts that a solution to~\eqref{eq:P_y and T_y equations} is equal to the true covariance $\hat{\Sigma}_x$ up to a circulant matrix of unknown phases. That is, each diagonal of $\hat{\Sigma}_x$ with circulant wrapping (i.e. the entries $\{\hat{\Sigma}_x[\ell,\operatorname{mod}(\ell+k,L)]\}_{k=0}^{L-1}$ for the $\ell$'th diagonal) in~\eqref{eq:P_y and T_y circulant phase ambiguities} is multiplied by an unknown phase factor of the form $e^{\imath \varphi}$. 
In accordance with Lemma~\ref{lem:P_y and T_y consequence}, in Section~\ref{section:step 1} we describe a procedure, based on convex optimization followed by a rank-one decomposition, which solves~\eqref{eq:P_y and T_y equations} and outputs a statistically consistent estimate (as $N\rightarrow \infty$) for a matrix $\hat{\Sigma}_x \odot \operatorname{Circulant} \left\{[1, e^{\imath \varphi_1} ,\ldots, e^{\imath \varphi_{L-1}} ] \right\}$ with unknown angles $\varphi_1,\ldots,\varphi_{L-1}$. See Algorithm~\ref{alg:step 1 alg} for a summary of the procedure, and Theorem~\ref{thm:consistency of step 1} for its consistency guarantee.

At this point, it is important to note that the problem of recovering $\hat{\Sigma}_x$ under the model~\eqref{eq:MRFA_model_def_Fourier} admits an inherent ambiguity. Clearly, cyclically shifting the signal $x$ results in a covariance matrix $\Sigma_x$ whose rows and columns are cyclically shifted, while $P_y$ and $T_y$ remain unchanged. In the Fourier domain, where $x$ is replaced with $\hat{x}$, this ambiguity corresponds to integer modulations of the rows and columns of $\hat{\Sigma}_x$. In particular, let $\Omega(\hat{\Sigma}_x)$ be a set of $L$ matrices given by
\begin{equation}
\Omega(\hat{\Sigma}_x) = \left\{ \hat{\Sigma}_x \; , \; \operatorname{diag}(f_1) \cdot \hat{\Sigma}_x \cdot \operatorname{diag}(f_1^*) \; , \; \ldots \; , \; \operatorname{diag}(f_{L-1}) \cdot \hat{\Sigma}_x \cdot  \operatorname{diag}(f_{L-1}^*) \right\}, \label{eq:Omega def}
\end{equation}
where $f_\ell$ is the $\ell$'th DFT vector defined in~\eqref{eq:DFT matrix and vectors def}.
The set $\Omega(\hat{\Sigma}_x)$ is a set of ambiguities associated with the MRFA problem, since each $X\in\Omega(\hat{\Sigma}_x)$ is the covariance matrix of $F \cdot (R_s\{x\})$ (i.e. the Fourier transform of the signal $x$ cyclically shifted by $s$) for some cyclic shift $s$ (see~\eqref{eq:MRFA model def}--\eqref{eq:x model} and~\eqref{eq:DFT matrix and vectors def}--\eqref{eq:MRFA_model_def_Fourier}).
Consequently, $\Omega(\hat{\Sigma}_x)$ is a set of ambiguities inherent in the recovery of $\hat{\Sigma}_x$ from the shift-invariant moments $P_y$ and $T_y$, as established by the following proposition. 
\begin{prop}[Fundamental ambiguities] \label{prop:fundamental ambiguities}
Every matrix $X\in\Omega(\hat{\Sigma}_x)$ is Hermitian, PSD, has rank~$r$, and satisfies the equations in~\eqref{eq:P_y and T_y equations}.
\end{prop}
\begin{proof}
The matrices $\operatorname{diag}(f_\ell) \cdot \hat{\Sigma}_x \cdot \operatorname{diag}(f_\ell^*)$ (for every $\ell\in\{0,\ldots,L-1\}$) are Hermitian since $\hat{\Sigma}_x$ is Hermitian. They are also PSD with rank $r$ since they are similar to $\hat{\Sigma}_x$ (and hence share their eigenvalues with $\hat{\Sigma}_x$).
Last, observe that
\begin{equation}
\operatorname{diag}(f_\ell) \cdot \hat{\Sigma}_x \cdot \operatorname{diag}(f_\ell^*) = \hat{\Sigma}_x \odot f_\ell f_\ell^* = \hat{\Sigma}_x \odot \operatorname{Circulant} \left\{f_\ell^* \right\}, \label{eq:Omega elements with DFT vectors expression} 
\end{equation}
and hence Lemma~\ref{lem:P_y and T_y consequence} establishes that the matrix $\operatorname{diag}(f_\ell) \cdot \hat{\Sigma}_x \cdot \operatorname{diag}(f_\ell^*)$ satisfies the equations in~\eqref{eq:P_y and T_y equations}.
\end{proof}
Henceforth, whenever we refer to the recovery of $\hat{\Sigma}_x$, we essentially mean the recovery of any (arbitrary) element from the set $\Omega(\hat{\Sigma}_x)$.

Evidently, the set of equations~\eqref{eq:P_y and T_y equations} goes a long way in narrowing down the set of feasible covariance matrices, as solving~\eqref{eq:P_y and T_y equations} leaves us with only $L-1$ unknown parameters $\varphi_1,\ldots,\varphi_{L-1}\in [0,2\pi)$, which are to be determined in order to recover $\hat{\Sigma}_x$. This leads us to consider the following problem.
\begin{prob}[Circulant phase retrieval] \label{problem:Sigma_x diagonal phase retrieval problem}
Given $X = \hat{\Sigma}_x \odot \operatorname{Circulant} \left\{[1, e^{\imath \varphi_1} ,\ldots, e^{\imath \varphi_{L-1}} ] \right\}$ with unknown angles $\varphi_1,\ldots,\varphi_{L-1} \in [0,2\pi)$, determine $\hat{\Sigma}_x$ (or any arbitrary element from $\Omega(\hat{\Sigma}_x)$ of~\eqref{eq:Omega def}). 
\end{prob}
In a way, Problem~\ref{problem:Sigma_x diagonal phase retrieval problem} can be viewed as a certain phase retrieval problem, where the phases multiplying each diagonal of $\hat{\Sigma}_x$ (with circulant wrapping) are to be retrieved, hence the name ``circulant phase retrieval''.
In this regard, note that Lemma~\ref{lem:P_y and T_y consequence} considers a general matrix $X\in\mathbb{C}^{L\times L}$, and ignores the fact that we actually seek a matrix which is Hermitian and PSD, which are properties satisfied by the true covariance matrix $\hat{\Sigma}_x$. Without any further prior knowledge on $\hat{\Sigma}_x$ (not even its rank), a natural way to go about solving Problem~\ref{problem:Sigma_x diagonal phase retrieval problem} is to try to solve the following surrogate problem.
\begin{prob} \label{problem:PSD circulant phase programming}
Given $X = \hat{\Sigma}_x \odot \operatorname{Circulant} \left\{[1, e^{\imath \varphi_1} ,\ldots, e^{\imath \varphi_{L-1}} ] \right\}$ with unknown angles $\varphi_1,\ldots,\varphi_{L-1} \in [0,2\pi)$, find angles $\widetilde{\varphi}_1, \ldots, \widetilde{\varphi}_{L-1}\in[0,2\pi)$ such that $\widetilde{X} := X \odot \operatorname{Circulant} \{[1, e^{-\imath \widetilde{\varphi}_1} ,\ldots, e^{-\imath \widetilde{\varphi}_{L-1}} ] \}$ is Hermitian and PSD. 
\end{prob}
Suppose that we are able to solve Problem~\ref{problem:PSD circulant phase programming}, then a fundamental question is whether any $\widetilde{X}$ solving Problem~\ref{problem:PSD circulant phase programming} is also a solution to Problem~\ref{problem:Sigma_x diagonal phase retrieval problem}, i.e. whether $\widetilde{X}$ is in the set of feasible solutions $\Omega(\hat{\Sigma}_x)$.
It turns out that for certain $\hat{\Sigma}_x$ which are sufficiently low-rank, any $\widetilde{X}$ solving Problem~\ref{problem:PSD circulant phase programming} is indeed also a solution to Problem~\ref{problem:Sigma_x diagonal phase retrieval problem}. In particular, we establish that this is true if $r=1$ under mild conditions on $\hat{\Sigma}_x$, or if $1<r<\sqrt{L}$ and $\hat{\Sigma}_x$ satisfies a certain technical condition (Condition~\ref{cond:low-rank recovery} in Section~\ref{section: step 2}). These results are summarized by the next Lemma.

\begin{lem} \label{lem:low-rank diagonal phase retrieval}
Suppose that $\hat{\Sigma}_x[i,j]\neq 0$ for all $i,j$, and either $r=1$, or $1 < r <\sqrt{L}$ and Condition~\ref{cond:low-rank recovery} holds. 
Then, if $\widetilde{X}$ is a solution to Problem~\ref{problem:PSD circulant phase programming} (i.e. $\widetilde{X} = X \odot \operatorname{Circulant} \{[1, e^{-\imath \widetilde{\varphi}_1} ,\ldots, e^{-\imath \widetilde{\varphi}_{L-1}} ] \}$ is Hermitian and PSD, where $X$ is as described in Problem~\ref{problem:PSD circulant phase programming}) it is also a solution to Problem~\ref{problem:Sigma_x diagonal phase retrieval problem} (i.e. $\widetilde{X} \in \Omega(\hat{\Sigma}_x)$).
\end{lem}
The proof of Lemma~\ref{lem:low-rank diagonal phase retrieval} is provided in Appendix~\ref{Appendix: Proof of lemma low-rank diagonal phase retrieval}.
In Section~\ref{section: step 2} we outline a polynomial-time procedure for solving Problem~\ref{problem:PSD circulant phase programming}, which is guaranteed to succeed if $r<\sqrt{L}$ and Condition~\ref{cond:low-rank recovery} holds, see Algorithm~\ref{alg:step 2 direct} (we note that even though Condition~\ref{cond:low-rank recovery} is not required for the claim of Lemma~\ref{lem:low-rank diagonal phase retrieval} in the case of $r=1$, we do require it for our guarantees on the success of Algorithm~\ref{alg:step 2 direct}). We mention that Condition~\ref{cond:low-rank recovery} arises naturally from the procedure described in Section~\ref{section: step 2}. While this condition is very technical and somewhat opaque, it can be easily tested for any $\Sigma_x$ using the singular values of a certain matrix whose construction is detailed in Section~\ref{section: step 2}. Moreover, Condition~\ref{cond:low-rank recovery} was observed to hold in all numerical experiments conducted in Section~\ref{section:numerical examples}. 

The following theorem is the main result concerning the recovery of low-rank covariance matrices from $P_y$ and $T_y$.
\begin{thm}[Low-rank recovery of $\hat{\Sigma}_x$] \label{thm:low-rank recovery}
Suppose that $\hat{\Sigma}_x[i,j]\neq 0$ for all $i,j$. If $r=1$, or if $1 < r <\sqrt{L}$ and Condition~\ref{cond:low-rank recovery} holds, then $\hat{\Sigma}_x$ (or any arbitrary element from $\Omega(\hat{\Sigma}_x)$) can be recovered from $P_y$ and $T_y$. Specifically, if $X$ is Hermitian, PSD, and satisfies equations~\eqref{eq:P_y and T_y equations}, then $X \in \Omega(\hat{\Sigma}_x)$.
\end{thm}
The proof of Theorem~\ref{thm:low-rank recovery} follows immediately from combining Lemma~\ref{lem:P_y and T_y consequence} with Lemma~\ref{lem:low-rank diagonal phase retrieval}.
Evidently, Theorem~\ref{thm:low-rank recovery} together with Proposition~\ref{prop:fundamental ambiguities} assert that a matrix $X$ is Hermitian, PSD, and satisfies equations~\eqref{eq:P_y and T_y equations} if and only if $X \in \Omega(\hat{\Sigma}_x)$.

Coupling the procedure for estimating $\hat{\Sigma}_x$ up to diagonal phase ambiguities (Algorithm~\ref{alg:step 1 alg}) with the procedure for resolving them (Algorithm~\ref{alg:step 2 direct}), we obtain a statistically consistent procedure for recovering $\hat{\Sigma}_x$ (see Theorem~\ref{thm:consistency of step 2, direct approach}), which has polynomial-time complexity. 

Now, while low-rank covariance matrices can be successfully recovered from $P_y$ and $T_y$, this is not the case for full-rank $\hat{\Sigma}_x$, as shown by the following proposition.
\begin{prop} [Full-rank $\hat{\Sigma}_x$] \label{prop:P_y and T_y ambiguity full rank}
Suppose that $\hat{\Sigma}_x[i,j]\neq 0$ for all $i,j$.
If $r = L$, then $\hat{\Sigma}_x$ cannot be recovered from only $P_y$ and $T_y$. That is, there exists a Hermitian and positive definite matrix $X$, with $X \notin \Omega(\hat{\Sigma}_x)$, which satisfies equations~\eqref{eq:P_y and T_y equations}.
\end{prop}
The proof of Proposition~\ref{prop:P_y and T_y ambiguity full rank} is provided in Appendix~\ref{appendix:proof of positive Sigma identifiability for P_y and T_y}.

\section{Recovering the covariance matrix $\Sigma_x$} \label{section:recovering the covariance}
In this section we describe our algorithms for estimating a low-rank $\hat{\Sigma}_x$ (and consequently the covariance matrix of $x$, i.e. $\Sigma_x$) using $N$ observations $y_1,\ldots,y_{N}$ from the model~\eqref{eq:MRFA model def}--\eqref{eq:x model}, with an arbitrary noise variance $\sigma^2$. We also provide appropriate statistical consistency guarantees for these algorithms.

\subsection{Step 1: Recovering $\hat{\Sigma}_x$ up to diagonal phase ambiguities} \label{section:step 1}
Given $N$ observations $y_1,\ldots,y_N$ drawn from the model~\eqref{eq:MRFA model def}--\eqref{eq:x model}, we first compute their Fourier transforms
\begin{equation}
\hat{y}_i = F y_i, \qquad i=1,\ldots,N,
\end{equation}
where $F$ is the DFT matrix from~\eqref{eq:DFT matrix and vectors def}. Next, we estimate $P_y$ and $T_y$ via
\begin{align}
&\widetilde{P}_y[k_1] = \frac{1}{N} \sum_{i=1}^N \left\vert \hat{y}_i[k_1] \right\vert^2,\label{eq: P_y estimator}  \\
& \widetilde{T}_y[k_1,k_2,k_3] = \frac{1}{N} \sum_{i=1}^N {\hat{y}_i}[k_1] \overline{\hat{y}_i[k_2]} {\hat{y}_i}[k_3] \overline{\hat{y}_i[k_1-k_2+k_3]}, \label{eq: T_y estimator} 
\end{align}
for $k_1,k_2,k_3 = 0,\ldots,L-1$.
Evidently, $\widetilde{P}_y$ and $\widetilde{T}_y$ are unbiased and consistent estimators for $P_y$ and $T_y$, respectively~\cite{wasserman2013all}.
We then proceed by constructing estimators for the diagonals of $\hat{\Sigma}_x$ using $\widetilde{P}_y$ and $\widetilde{T}_y$ (as shown below), and prove that they are statistically consistent as $N\rightarrow \infty$ up to arbitrary phase factors (i.e., multiplicative constants with unit magnitude). 

We now describe our estimation procedure in detail. Let us denote by $d_m\in\mathbb{C}^L$ a column vector given by the $m$'th diagonal (with circulant wrapping) of the matrix $\hat{\Sigma}_x + \sigma^2 I_L$, i.e.
\begin{equation}
d_m[k] = 
\begin{dcases}
\hat{\Sigma}_x[k,\operatorname{mod} ( k+m, L)], & m\neq 0, \\
\hat{\Sigma}_x[k,k] + \sigma^2, & m=0,
\label{eq:d_m def}
\end{dcases}
\end{equation}
for $m,k\in \{ 0,\ldots,L-1\}$.
In Appendix~\ref{appendix:Justification of expression for P_y and T_y with diagonals} we show that $P_y$ and $T_y$ can be expressed in terms of $d_0,\ldots,d_{L-1}$ as
\begin{align}
&P_y[k] = d_0[k], \label{eq:P_y expression diagonals}\\
&T_y[k_1,k_1+m,k_2+m] = d_m[k_1] \overline{d_m[k_2]} + d_{k_2-k_1}[k_1] \overline{d_{k_2-k_1}[k_1+m]},  \label{eq:T_y expression diagonals}
\end{align}
for every $k,k_1,k_2,m\in\{0,\ldots,L-1\}$.
Next, we define the matrices $G_m \in\mathbb{C}^{L\times L}$, $m=0,\ldots,L-1$, by
\begin{equation}
G_m = d_m d_m^*, \label{eq:G_m def}
\end{equation}
noting that $G_m \succeq 0$ with $\operatorname{rank}\{ G_m\}=1$, and rewrite~\eqref{eq:T_y expression diagonals} using~\eqref{eq:G_m def} as
\begin{equation}
T_y[k_1,k_1+m,k_2+m] = G_m[k_1,k_2] + G_{k_2-k_1}[k_1,k_1+m].\label{eq:T_y expression G_m}
\end{equation}
We next estimate the matrices $G_m$, for $m\geq 1$, by solving an optimization problem that fits $\widetilde{T}_y$ from~\eqref{eq: T_y estimator} to the form~\eqref{eq:T_y expression G_m} (replacing $T_y$ with its estimate $\widetilde{T}_y$), while removing the rank constraint on $G_m$. Specifically, we solve
\begin{equation}
\begin{aligned} \label{eq:optim for diagonals}
\{ \widetilde{G}_m \}_{m=1}^{L-1} = &\underset{ \{ G^{'}_m \}_{m=1}^{L-1}}{\operatorname{argmin}} \left\{ \sum_{k_1,k_2,m = 0}^{L-1} \left\vert {\widetilde{T}_y[k_1,k_1+m,k_2+m] - G^{'}_{k_2-k_1}[k_1,k_1+m] - G^{'}_m[k_1,k_2]} \right\vert^2 \right\}, \\
&\text{subject to} \;\; \left\{ G^{'}_m \succeq 0 \right\}_{m=1}^{L-1},\;\; G^{'}_0= \widetilde{P}_y\cdot \widetilde{P}_y^T,
\end{aligned}
\end{equation}
which is a linear least-squares problem with semidefinite constraints, hence a convex optimization problem readily solved by a variety of algorithms~\cite{grant2014cvx,boumal2014manopt}.
Even though we omitted the rank constraint on $G_m$ when solving~\eqref{eq:optim for diagonals}, the resulting estimates $\widetilde{G}_m$ approximate the matrices $G_m$  (as established in the proof of Theorem~\ref{thm:consistency of step 1} below) and are close to being rank-one, as exemplified in Figure~\ref{fig:eigenvalues of G_widetilde_m} for $L=10$, $r=3$, $N=10^4$, and $\sigma^2=0.05$.
\begin{figure}
  \centering  	
    \includegraphics[width=0.55\textwidth]{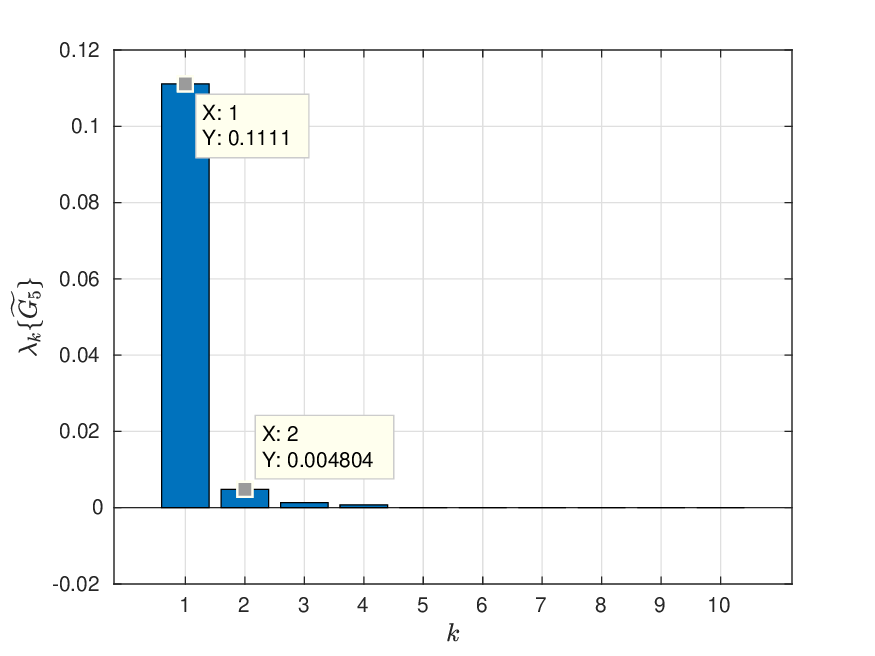}
	\caption[A structure] 
	{Eigenvalues of $\widetilde{G}_5$ for $L=10$, $r=3$, $N=10^4$, and $\sigma^2=0.05$, obtained from solving~\eqref{eq:optim for diagonals} for a simulated $\Sigma_x$ with eigenvalues $[1,0.7,0.5]$ and eigenvectors sampled uniformly from the sphere. It is evident that the matrix $\widetilde{G}_5$, which is an estimate of $G_5$, is close to being rank-one.}  \label{fig:eigenvalues of G_widetilde_m}
\end{figure}
Then, each diagonal $d_m$ for $m\geq 1$ is estimated from the best rank-one approximation to $\widetilde{G}_m$. In particular, if $\widetilde{\mu}_1^{(m)}$ is the largest eigenvalue of $\widetilde{G}_m$ and $\widetilde{u}_1^{(m)}$ is its corresponding eigenvector, then we estimate $d_m$ via
\begin{equation}
\widetilde{d}_m = \sqrt{\widetilde{\mu}_1^{(m)}} \widetilde{u}_1^{(m)}, \label{eq:d_m estimator}
\end{equation} 
noting that $\widetilde{d}_m$ is unique up to a phase factor, i.e. a constant $e^{\imath \varphi_m}$ multiplying $d_m$, where $\varphi_m \in [0,2\pi)$ is an (unknown) angle. Note that according to~\eqref{eq:P_y expression diagonals}, the main diagonal of $\hat{\Sigma}_x$, namely $d_0$, can be estimated directly from $\widetilde{P}_y$ and hence does not suffer from this phase ambiguity.
We therefore proceed by forming the matrix $\widetilde{C}_x\in\mathbb{C}^{L\times L}$, given by
\begin{equation} \label{eq:C_x_tilde def}
\widetilde{C}_x[k_1,k_2] = 
\begin{dcases}
\widetilde{d}_m[k_1], & \operatorname{mod}(k_2-k_1,L) = m, \\
\widetilde{P}_y[k_1] - \sigma^2, & k_2 = k_1,
\end{dcases}
\end{equation}
where the subtraction of $\sigma^2$ from the main diagonal of $\widetilde{C}_x$ corrects for the bias due to noise.
The following theorem establishes that $\widetilde{C}_x$ is a statistically-consistent estimate of $\hat{\Sigma}_x$ as $N\rightarrow \infty$, up to an unknown circulant phase matrix.
\begin{thm}[Consistency of~\eqref{eq:C_x_tilde def}] \label{thm:consistency of step 1}
Suppose that $\hat{\Sigma}_x[i,j]\neq 0$ for all $i,j$. Then,
\begin{align}
\min_{\varphi_1,\ldots,\varphi_{L-1}\in [0,2\pi)}\left\Vert \widetilde{C}_x - \hat{\Sigma}_x \odot \operatorname{Circulant} \left\{ [1,e^{\imath \varphi_1} ,\ldots, e^{\imath \varphi_{L-1}} ] \right\} \right\Vert_F &\underset{N\rightarrow\infty, \; \text{a.s.}}{\longrightarrow} 0. \label{eq:C_x_tilde consistency with ambiguity}
\end{align}
\end{thm}
The proof of Theorem~\ref{thm:consistency of step 1} is provided in Appendix~\ref{appendix: Proof of consistency of step 1}.

Now, from a practical standpoint, it is worthwhile to briefly consider the real-valued case, where $v_i\in \mathbb{R}^L$, $a_i \sim \mathcal{N}(0,\lambda_i)$, and $\eta\sim\mathcal{N}(0,\sigma^2 I_L)$. The only difference between the real-valued case and the complex-valued case lies in the expression for the trispectrum $T_y$, which now admits an additional additive term. In particular, we show in Appendix~\ref{appendix:trispectrum for real-valued} that instead of~\eqref{eq:T_y expression G_m} we have
\begin{equation}
T_y[k_1,k_1+m,k_2+m] = G_m[k_1,k_2] + G_{k_2-k_1}[k_1,k_1+m] + G_{k_1+k_2+m}[-k_2,-k_2-m]. \label{eq:T_y expression G_m real-valued case}
\end{equation}
Therefore, analogously to~\eqref{eq:optim for diagonals}, we propose to solve
\begin{equation}
\label{eq:optim for diagonals real-valued case}
\begin{gathered} 
\begin{aligned}
\{ \widetilde{G}_m \}_{m=1}^{L-1} &= \underset{ \{ G^{'}_m \}_{m=1}^{L-1}}{\operatorname{argmin}} \bigg\{ \sum_{k_1,k_2,m = 0}^{L-1} \bigg\vert \widetilde{T}_y[k_1,k_1+m,k_2+m] - G^{'}_{k_2-k_1}[k_1,k_1+m] \\
&\qquad \qquad\qquad\qquad\qquad\qquad\qquad- G^{'}_m[k_1,k_2] - G^{'}_{k_1+k_2+m}[-k_2,-k_2-m] \bigg\vert^2 \bigg\}, 
\end{aligned}\\
\text{subject to} \;\; \left\{ G^{'}_m \succeq 0 \right\}_{m=1}^{L-1},\;\; G^{'}_0= \widetilde{P}_y\cdot \widetilde{P}_y^T,
\end{gathered}
\end{equation}
and proceed with the estimation of $\widetilde{d}_m$ and the construction of $\widetilde{C}_x$ as in the complex-valued case. Due to the additional term in the expression for the trispectrum in~\eqref{eq:T_y expression G_m real-valued case}, the proof of the analogue of Theorem~\ref{thm:consistency of step 1} in the real-valued case is somewhat more complicated, and is left for a future work.
Nonetheless, we demonstrate by numerical experiments in Section~\ref{section:numerical examples} that the proposed approach for the real-valued case provides results that are very similar to the complex-valued case.

The algorithm for recovering $\hat{\Sigma}_x$ up to unknown diagonal phase ambiguities, for both the complex-valued and the real-valued cases, is described in Algorithm~\ref{alg:step 1 alg}.

\begin{algorithm}
\caption{Recovering $\hat{\Sigma}_x$ up to diagonal phase ambiguities}\label{alg:step 1 alg}
\begin{algorithmic}[1]
\Statex{\textbf{Required:} Measurements $y_1,\ldots,y_N$ from the model~\eqref{eq:MRFA model def}--\eqref{eq:x model}, for either the real-valued case or the complex-valued case}.
\State Compute the Fourier transforms of the measurements $\hat{y}_i = F y_i$ using the Fast Fourier Transform (FFT), where $F$ is the DFT matrix~\eqref{eq:DFT matrix and vectors def}.
\State Compute $\widetilde{P}_y$ and $\widetilde{T}_y$ according to~\eqref{eq: P_y estimator} and~\eqref{eq: T_y estimator}, respectively.
\State Obtain the $L\times L$ matrices $\{ \widetilde{G}_m \}_{m=1}^{L-1}$ by solving problem~\eqref{eq:optim for diagonals} for the complex-valued case, or solving problem~\eqref{eq:optim for diagonals real-valued case} for the real-valued case. \label{alg 1:solve opt prob}
\State Evaluate $\widetilde{d}_m$ from the best rank-one approximation to $\widetilde{G}_m$ via~\eqref{eq:d_m estimator}, for $m=1,\ldots,L-1$.
\State Output the $L\times L$ matrix $\widetilde{C}_x$ from~\eqref{eq:C_x_tilde def}.
\end{algorithmic}
\end{algorithm}

\begin{remark}
It is worthwhile to point out that problem~\eqref{eq:optim for diagonals} is ill-posed without the semidefinite constraints. Removing the semidefinite constraints in~\eqref{eq:optim for diagonals} results in a linear least-squares system with $L^3$ equations, and a smaller number of $L^3 - L^2$ variables (due to the constraint $G^{'}_0= \widetilde{P}_y\cdot \widetilde{P}_y^T$), which is not underdetermined. Yet, we observe that for every triplet of indices $(k_1,k_2,m)$ there exists another triplet $(k_1,k_1+m,k_2-k_1)$ which results in exactly the same equation as for the first triplet (since the terms $G^{'}_{k_2-k_1}[k_1,k_1+m]$ and $G^{'}_m[k_1,k_2]$ in~\eqref{eq:optim for diagonals} interchange). Therefore, the number of independent equations is actually smaller than the number of variables and the problem is ill-posed. Yet, it turns out that the semidefinite constraints resolve this ill-posedness, as established in the proof of Theorem~\ref{thm:consistency of step 1}.
\end{remark}

\begin{remark} \label{remark:trispectrum symmetries}
We mention that the trispectrum $T_y$ admits several symmetries which can be exploited to reduce the computational burden of Algorithm~\ref{alg:step 1 alg}. 
Notice from~\eqref{eq:T_y def} that swapping the first and  third, or second and fourth indices of $M^{(4)}_{\hat{y}}$ does not change the value of $M^{(4)}_{\hat{y}}$. Therefore, it is clear that
\begin{equation}
\widetilde{T}_y[k_1,k_1+m,k_2+m] = \widetilde{T}_y[k_2+m,k_1+m,k_1] = \widetilde{T}_y[k_2+m,k_2,k_1] = \widetilde{T}_y[k_1,k_2,k_2+m],
\end{equation}
hence it is sufficient to estimate only about a quarter of the elements of $T_y$.
\end{remark}

\subsection{Step 2: Resolving the diagonal phase ambiguities} \label{section: step 2}
We consider an estimator for $\hat{\Sigma}_x$ of the form
\begin{equation}
\widetilde{\hat{\Sigma}}_x := \widetilde{C}_x \odot \operatorname{Circulant}\{[1,e^{-\imath \widetilde{\varphi}_1},\ldots,e^{-\imath \widetilde{\varphi}_{L-1}}] \}, \label{eq:covariance estimator with phase}
\end{equation}
where $\widetilde{C}_x$ is from~\eqref{eq:C_x_tilde def}. In this section, we derive a procedure to find the angles $\widetilde{\varphi}_1,\ldots,\widetilde{\varphi}_{L-1} \in [0,2\pi)$ such that $\widetilde{\hat{\Sigma}}_x$ is close to being Hermitian and PSD.
For simplicity of presentation, we derive the procedure in the limiting case of $N\rightarrow \infty$. Specifically, we consider the setting of Problem~\ref{problem:PSD circulant phase programming}, where 
we assume that we have access to the matrix $X\in\mathbb{C}^{L\times L}$, given by
\begin{equation}
X = \hat{\Sigma}_x \odot \operatorname{Circulant} \{ [1, e^{\imath \varphi_1},\ldots, e^{\imath \varphi_{L-1}} ] \}, \label{eq:X def}
\end{equation}
with unknown angles $\varphi_1,\ldots,\varphi_{L-1}$, and seek angles $\widetilde{\varphi}_1,\ldots,\widetilde{\varphi}_{L-1}$ such that the matrix
\begin{equation}
\widetilde{X} := X\odot \operatorname{Circulant} \{ [1, e^{-\imath \widetilde{\varphi}_1},\ldots, e^{-\imath \widetilde{\varphi}_{L-1}} ] \}, \label{eq:X_tilde def}
\end{equation}
is Hermitian and PSD.

Let us define the matrices ${H}_{i,j}\in\mathbb{C}^{L\times L}$, for $i,j=0,\ldots,L-1$, by
\begin{equation}
{H}_{i,j} = X \odot \overline{R_{i,j}\{ X \}},	\label{eq:H_ij_tilde def}
\end{equation}
where $R_{i,j}\{ X\}$ is the operation of cyclically shifting the rows and columns of $X$ by $i$ and $j$, respectively, namely
\begin{equation}
R_{i,j}\{ X\}[k,m] = X[\operatorname{mod} (k-i,L),\operatorname{mod} (m-j,L)]. \label{eq:R_ij def}
\end{equation}
The following lemma summarizes several properties of $H_{i,i}$ required for our derivation.
\begin{lem}\label{lem:H_ii properties}
The matrix $H_{i,i}$ (taking $j=i$ in~\eqref{eq:H_ij_tilde def}) is given explicitly by
\begin{equation}
{H}_{i,i} = \hat{\Sigma}_x \odot \overline{R_{i,i}\{ {\hat{\Sigma}_x} \} }, \label{eq:H_ii expression} 
\end{equation}
and is Hermitian, PSD, and satisfies
\begin{equation}
\operatorname{Rank}\left\{ H_{i,i} \right\} 
\leq r^2, \label{eq:H_ii rank bound}
\end{equation}
for every $i=0,\ldots,L-1$.
\end{lem}
The proof of Lemma~\ref{lem:H_ii properties} is provided in Appendix~\ref{appendix:proof of properties of H_ii}.
Next, using~\eqref{eq:X_tilde def}, we define the matrix ${S}\in\mathbb{C}^{L^2\times L^2}$ via its $L\times L$ blocks $S^{(i,j)}$ as
\begin{equation} \label{eq:Kron product block formula}
S^{(i,j)} := \widetilde{X} \odot \overline{R_{i,j}\{ \widetilde{X} \} } = 
\begin{dcases}
{H}_{i,i}, & i=j, \\
{H}_{i,j} \odot \operatorname{Circulant}\{[\beta_0^{(j-i)},\beta_1^{(j-i)},\ldots,\beta_{L-1}^{(j-i)}] \} , & i\neq j,
\end{dcases}
\end{equation}
where $S^{(i,j)}\in\mathbb{C}^{L\times L}$ denotes the $(i,j)$'th $L\times L$ block of ${S}$, and 
\begin{equation}
\beta_m^{(j-i)} = e^{-\imath ( \widetilde{\varphi}_m - \widetilde{\varphi}_{m-j+i})}.	\label{eq:beta def}
\end{equation}
We have the following lemma regarding the matrix ${S}$ of~\eqref{eq:Kron product block formula}.
\begin{lem} \label{lem:S properties}
If $\widetilde{X}$ of~\eqref{eq:X_tilde def} is Hermitian and PSD, then ${S}$ of~\eqref{eq:Kron product block formula} is also Hermitian and PSD.
\end{lem}
The proof of Lemma~\ref{lem:S properties} is provided in Appendix~\ref{appendix:Proof of Lemma S properties}.
From Lemma~\ref{lem:S properties} it follows that for appropriate angles $\widetilde{\varphi}_1,\ldots,\widetilde{\varphi}_{L-1}$ such that $\widetilde{X}$ is Hermitian and PSD, ${S}$ is also Hermitian and PSD, and we can write
\begin{equation}
{S} = {K} {K}^*, \qquad \qquad S^{(i,j)} = {K}_i {K}_j^*, \label{eq:Kron product PSD block decomposition}
\end{equation}
for some $K\in\mathbb{C}^{L^2 \times L^2}$, where ${K}_{i}$ denotes the $i$'th $L\times L^2$ block of $K$ (the $L$ consecutive rows of ${K}$ starting from row number $(i-1)L +1$).
From \eqref{eq:Kron product PSD block decomposition} and~\eqref{eq:Kron product block formula}, we have that $H_{i,i} = K_i K_i^*$, which implies that the columns of ${K}_i$ are spanned by the eigenvectors of ${H}_{i,i}$. Following~\eqref{eq:H_ii rank bound}, we define ${V}^{(i)}\in\mathbb{C}^{L\times r^2}$ to be the matrix whose columns are the $r^2$ eigenvectors of ${H}_{i,i} $ which correspond to its largest eigenvalues. Then, we can write
\begin{equation}
{K}_i = {V}^{(i)} A_i, \label{eq:K_i_tilde expansion}
\end{equation}
where $A_i\in \mathbb{C}^{r^2 \times L^2}$ is a matrix of unknown coefficients. Now, using~\eqref{eq:K_i_tilde expansion},~\eqref{eq:Kron product PSD block decomposition}, and~\eqref{eq:Kron product block formula} we have that
\begin{equation}
{V}^{(i)} A_i A_j^* ({V}^{(j)})^* = {V}^{(i)} B_{i,j} ({V}^{(j)})^* = {H}_{i,j} \odot \operatorname{Circulant}\{[\beta_0^{(j-i)},\beta_1^{(j-i)},\ldots,\beta_{L-1}^{(j-i)}] \}, \label{eq:kron prod PSD system fixed i j}
\end{equation}
where we defined $A_i A_j^* = B_{i,j} \in \mathbb{C}^{r^2\times r^2}$ to be a matrix of $r^4$ unknown coefficients. Importantly, fixing $i$ and $j$,~\eqref{eq:kron prod PSD system fixed i j} describes a system of linear equations in the $r^4$ variables $\{ B_{i,j}\}_{i,j=1}^{r^2}$ and the $L$ variables $\beta_0^{(j-i)},\beta_1^{(j-i)},\ldots,\beta_{L-1}^{(j-i)} \in \mathbb{C}$, where we relaxed the requirement that $\beta_m^{(j-i)}$ have unit norm (we will see that this relaxation still enables us to obtain the correct angles $\widetilde{\varphi}_1,\ldots,\widetilde{\varphi}_{L-1}$). Recall that the matrices $H_{i,j}$ and $V^{(i)}$ are computed from the matrix $X$, which is provided to us (or estimated from the data, e.g. $\widetilde{C}_x$ from Section~\ref{section:step 1}). Hence, in total,~\eqref{eq:kron prod PSD system fixed i j} describes a linear system with $L^2$ equations in $r^4+L$ variables, among which the $L$ variables $\{\beta_m^{(j-i)}\}_{m=0}^{L-1}$ encode the required correcting angles $\widetilde{\varphi}_m$ from~\eqref{eq:X_tilde def}. Now, even though it is possible to exploit~\eqref{eq:kron prod PSD system fixed i j} directly to solve the problem at hand (identifying the phases $\varphi_1,\ldots,\varphi_{L-1}$), we proceed by forming an augmented linear system with more equations compared to the number of variables, which ultimately allows to recover $\hat{\Sigma}_x$ for larger ranks $r$. To this end, we couple together all systems of equations from~\eqref{eq:kron prod PSD system fixed i j} for all $i,j$ such that $j-i=1$, noting that $\beta_m^{(j-i)}=\beta_m^{(1)}$ are shared by all such systems. We then obtain the set of equations
\begin{equation}
{V}^{(i)} B_{i,i+1} ({V}^{(i+1)})^* = {H}_{i,i+1} \odot \operatorname{Circulant}\{[\beta_0^{(1)},\beta_1^{(1)},\ldots,\beta_{L-1}^{(1)}] \}, \qquad i=0,\ldots,L-1,  \label{eq:kron prod PSD systems j=i+1}
\end{equation}
which is a system of $L^3$ equations in $L (r^4 + 1)$ variables. Continuing, we can write the linear system of~\eqref{eq:kron prod PSD systems j=i+1} in standard matrix notation as
\begin{equation}
{W} {b} = \mathbf{0}, \label{eq:A beta = 0}
\end{equation} 
where $\mathbf{0}$ is a column vector of $L^3$ zeros, ${b}\in\mathbb{C}^{L+r^4L}$ is a column vector of variables formed by stacking $[\beta_0^{(1)},\beta_1^{(1)},\ldots,\beta_{L-1}^{(1)}]^T$ on top of all of the elements in $\{B_{i,i+1}\}_{i=0}^{L-1}$, and the matrix ${W}\in\mathbb{C}^{L^3 \times (L+r^4 L)}$ is constructed from~\eqref{eq:kron prod PSD systems j=i+1} as follows. Let ${Z}^{(i)}\in\mathbb{C}^{L^2\times r^4}$ and ${M}^{(i)}_m \in \mathbb{C}^{L^2}$, for $i,m\in \{0,\ldots,L-1 \}$, be given by
\begin{align} 
{Z}^{(i)} &= \overline{{V}^{(i+1)}} \otimes {V}^{(i)}, \label{eq:Z_tilde def} \\
{M}^{(i)}_m &= \operatorname{vec}\left\{{H}_{i,i+1}\right\} \odot \operatorname{vec}\left\{\operatorname{Circulant}\{\mathbf{e}_m\} \right\},\label{eq:M_tilde def}
\end{align}
where $\otimes$ is the Kronecker product, $\mathbf{e}_m$ is the $m$'th indicator vector (with a single value of $1$ at the $m$'th entry), $\operatorname{vec}\{\cdot\}$ is the operation of vectorizing a matrix by stacking its columns on top of one another (with the leftmost column being at the top of the resulting vector), and recall that $V^{(i)}$ is the $L\times r^2$ matrix whose colums are the first $r^2$ eigenvectors of $H_{i,i}$ (corresponding to its largest eigenvalues).
Then, ${W}$ is given by 
\begin{equation} \label{eq:A_mathcal_tilde def}
{W} = 
\begin{bmatrix}
\begin{pmatrix}
 {M}^{(0)}_0  & \hdots & {M}^{(0)}_{L-1}  \\ 
\vdots & \ddots & \vdots \\
{M}^{(L-1)}_0 & \hdots & {M}^{(L-1)}_{L-1} 
\end{pmatrix} & 
\begin{pmatrix}
-\operatorname{BlockDiag}\left\{ {{Z}}^{(0)},\ldots,{{Z}}^{(L-1)}\right\}
\end{pmatrix}
\end{bmatrix},
\end{equation}
where $\operatorname{BlockDiag}\{{{Z}}^{(0)},\ldots,{{Z}}^{(L-1)} \}$ stands for a block-diagonal matrix constructed from the matrices ${{Z}}^{(0)},\ldots,{{Z}}^{(L-1)}$, namely a matrix of size $L^3 \times r^4 L$ with $L$ non-zero blocks along its main diagonal, each of size $L^2\times r^4$. Figure~\ref{fig:A structure} depicts the structure of a typical matrix ${W}$.

\begin{figure}
  \centering
  	\subfloat[${W}$]  	 
  	{
    \includegraphics[width=0.5\textwidth]{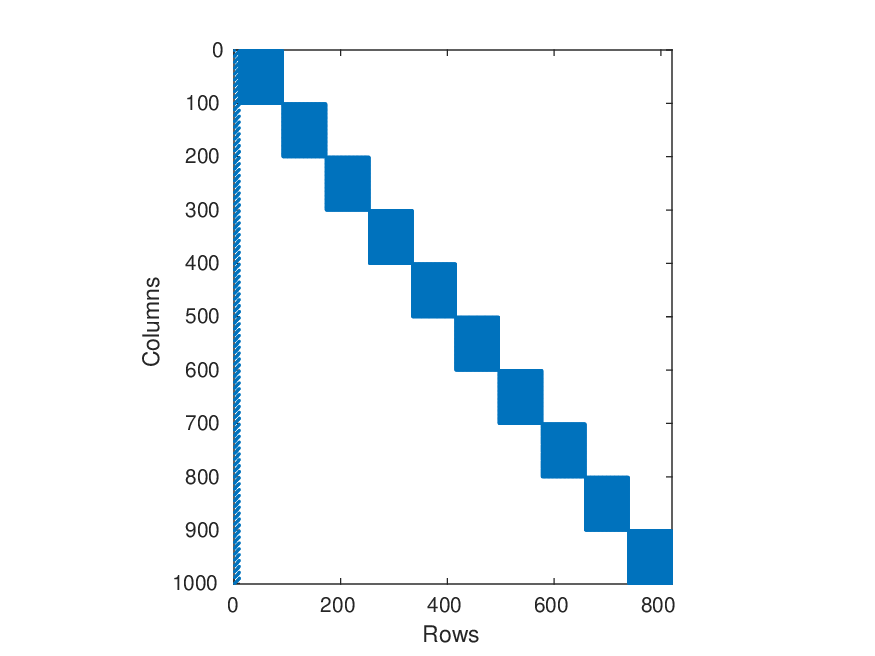} \label{fig:A structure}
    }
    \subfloat[${W}^* {W}$]  
    { 
    \includegraphics[width=0.5\textwidth]{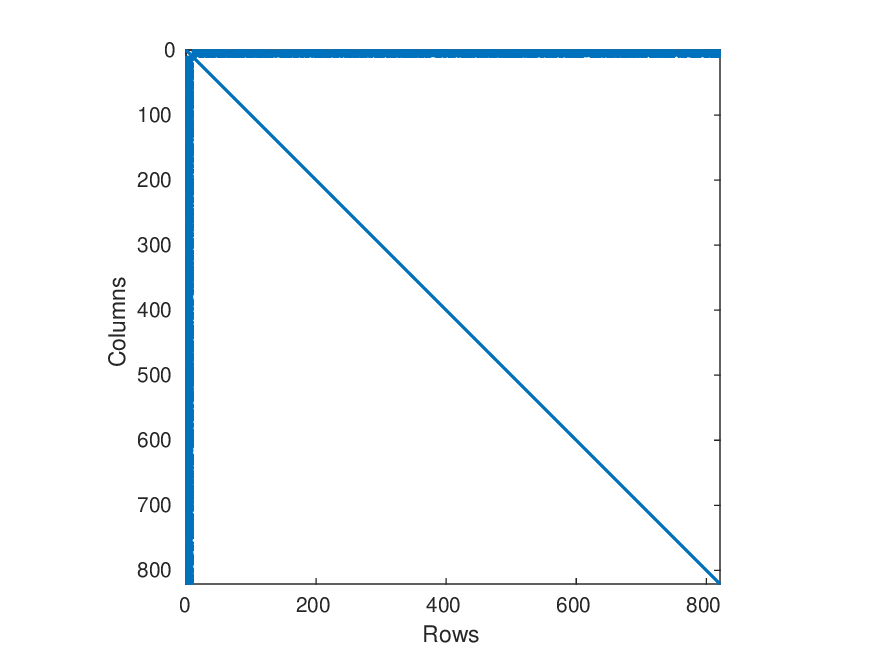} \label{fig:AA structure}
    }
    \caption
    {Nonzero elements of the matrices ${W}$ (left) and ${W}^* {W}$ (right), for $L = 10$ and $r = 3$. The matrix ${W}$ is of size $L^3 \times (L+r^4L)$, with $L$ nonzero columns on the left, followed by a block-diagonal matrix with $L$ blocks, each of size $L^2\times r^4$. The matrix ${W}^* {W}$ is considerably sparser than ${W}$, since all $r^4 L$ rightmost columns of ${W}$ are orthonormal (see Remark~\ref{remark:efficient evaluation of V_mathcal_tilde}). Hence, the nonzero entries of ${W}^* {W}$ include only the top $L$ rows, the $L$ leftmost columns, and the main diagonal. Consequently, ${W}^* {W}$ is much better suited for solving~\eqref{eq:A beta = 0}.}  \label{fig:A_mathcal_tilde structure}
    \end{figure}
    
Next, note that $b = \mathbf{0}$ is a possible solution to~\eqref{eq:A beta = 0}, where $\mathbf{0}$ is the column vector of $L+r^4L$ zeros. However, we know that there must exist at least one additional nonzero solution corresponding to the true phase ambiguities ($\beta_m^{(1)} = e^{\imath (\varphi_m - \varphi_{m-1})}$ is one such solution). Therefore, the linear system of~\eqref{eq:A beta = 0} must admit an infinite number of solutions. This implies that if the system in~\eqref{eq:A beta = 0} is not underdetermined (as we enforce next), then the smallest singular value of ${W}$ must be zero.
Note that the system in~\eqref{eq:A beta = 0} is not underdetermined if we require that $L+r^4L \leq L^3$, which is equivalent to requiring $r<\sqrt{L}$ (since $r$ and $L$ are integers).
In order to proceed, we need the following condition.
\begin{cond} \label{cond:low-rank recovery}
The second-smallest singular value of $W$ is strictly positive.
\end{cond}
Condition~\ref{cond:low-rank recovery} can be easily tested by computing the singular-value decomposition (SVD) of $W$ constructed from $\hat{\Sigma}_x$. Figure~\ref{fig:A singular vals} depicts the six smallest singular values of ${W}$ when constructed from a covariance matrix $\hat{\Sigma}_x$ with $L=10$, eigenvalues $[1,0.7,0.5]$, and eigenvectors randomly sampled from the unit sphere (with uniform distribution).
\begin{figure}
  \centering
  	\subfloat[Smallest singular values of $W$]  	 
  	{
    \includegraphics[width=0.45\textwidth]{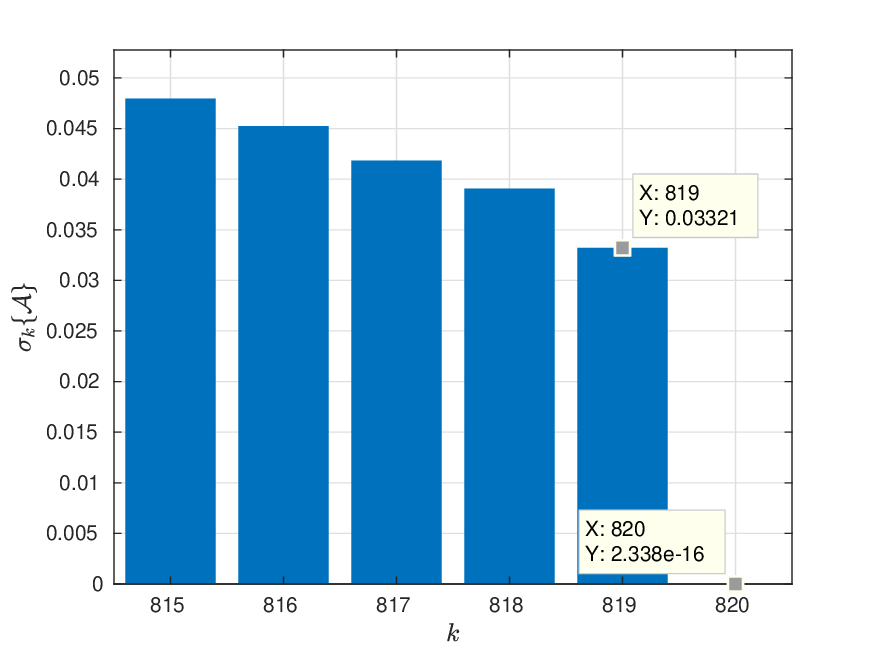} \label{fig:A singular vals}
    } 
    \subfloat[Smallest singular values of $\widetilde{W}$]  
    { 
    \includegraphics[width=0.45\textwidth]{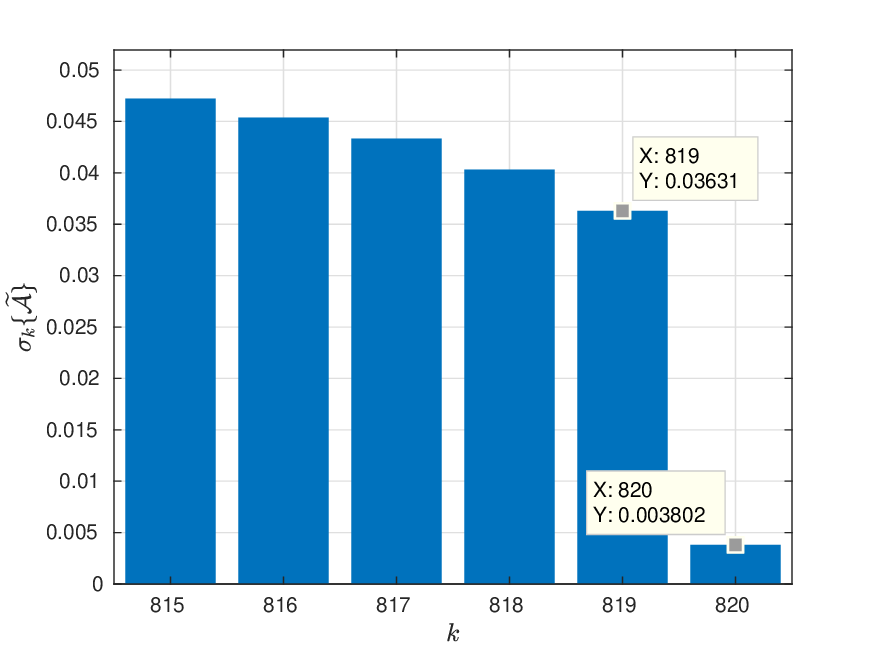} \label{fig:A_tilde singular vals}
    }
    \caption
	{The six smallest singular values of $W$ (left) and $\widetilde{W}$ (right) for $L=10$, $r=3$, $N=10^4$, and $\sigma^2 = 0.05$. The covariance matrix $\hat{\Sigma}_x$ was simulated with eigenvalues $[1,0.7,0.5]$, and eigenvectors sampled uniformly from the (complex) sphere. As expected, the smallest singular value of $W$ is zero, and it is evident that Condition~\ref{cond:low-rank recovery} holds. Moreover, even in the finite sample case where $N=10^4$ and $\sigma^2=0.05$, the smallest singular value of $\widetilde{W}$ is well-separated from the other singular values.} \label{fig:A and A_tilde singular values}
\end{figure}
Now, assuming that $r<\sqrt{L}$ and Condition~\ref{cond:low-rank recovery} holds, then the solution to~\eqref{eq:A beta = 0} is the span of the right singular vector of ${W}$ corresponding to its smallest singular value. 
Denoting this singular vector by ${\mathcal{V}}\in\mathbb{C}^{L+r^4L}$, we have that
\begin{equation}
b = c {\mathcal{V}}, \label{eq:beta solution}
\end{equation}
for any complex constant $c$. At this point, we briefly mention that a naive evaluation of $\mathcal{V}$ can be computationally challenging. In this regard, Remark~\ref{remark:efficient evaluation of V_mathcal_tilde} below outlines an efficient approach, which utilizes $W^* W$ instead of $W$ to evaluate $\mathcal{V}$.
Continuing with our derivation, from~\eqref{eq:beta def},~\eqref{eq:A beta = 0}, and~\eqref{eq:beta solution} it follows that 
\begin{equation}
c {\mathcal{V}}[m] = e^{-\imath (\widetilde{\varphi}_m - \widetilde{\varphi}_{m-1})},  \qquad m=0,\ldots,L-1. \label{eq:V_mathcal_tilde expression}
\end{equation}
Hence, the magnitudes of the first $L$ elements of ${\mathcal{V}}$ should be constant, and their phases should satisfy 
\begin{equation}
\operatorname{arg}\left\{ {\mathcal{V}}[m] \right\} = \operatorname{mod}(-\widetilde{\varphi}_m + \widetilde{\varphi}_{m-1} + \alpha, 2\pi), \qquad m=0,\ldots,L-1, \label{eq:V_mathcal_tilde arg expression}
\end{equation}
where $\operatorname{arg}\{\cdot\}$ is the argument of a complex number ($\varphi = \operatorname{arg}\{e^{\imath \varphi}\}$), and $\alpha\in [0,2\pi)$ is an unknown angle ($\alpha = -\operatorname{arg}\{c\}$). 
Figure~\ref{fig:magnitudes of V_tilde} illustrates the magnitudes of the first $30$ elements of ${\mathcal{V}}$, for the same matrix $\hat{\Sigma}_x$ as used in Figure~\ref{fig:A and A_tilde singular values}, exemplifying the agreement with~\eqref{eq:V_mathcal_tilde expression}. 

\begin{figure}
  \centering  	
  \subfloat[Magnitudes of $\mathcal{V}$]  	 
  	{
    \includegraphics[width=0.45\textwidth]{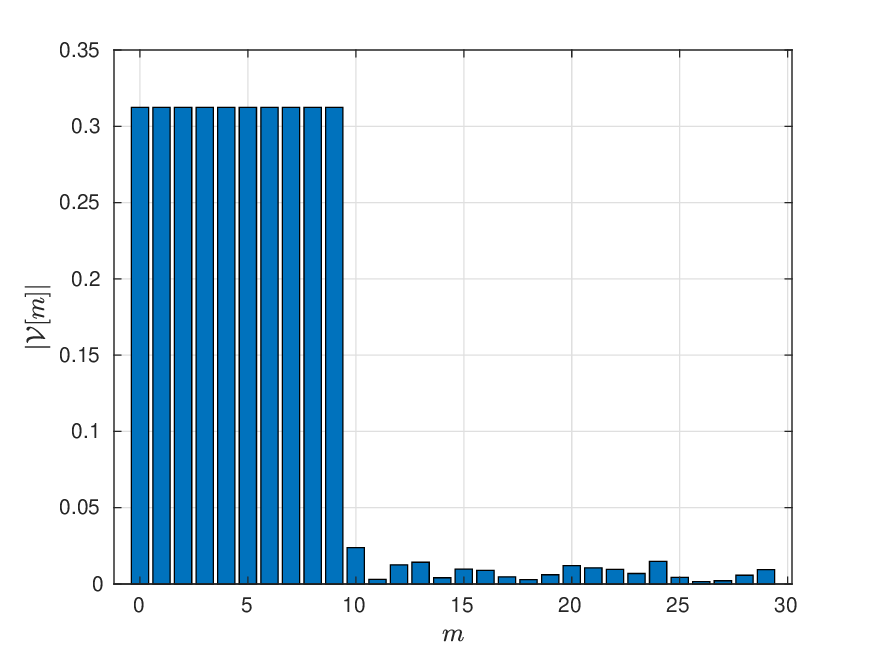} \label{fig:magnitudes of V}
    } 
    \subfloat[Magnitudes of $\widetilde{\mathcal{V}}$]  	 
  	{
    \includegraphics[width=0.45\textwidth]{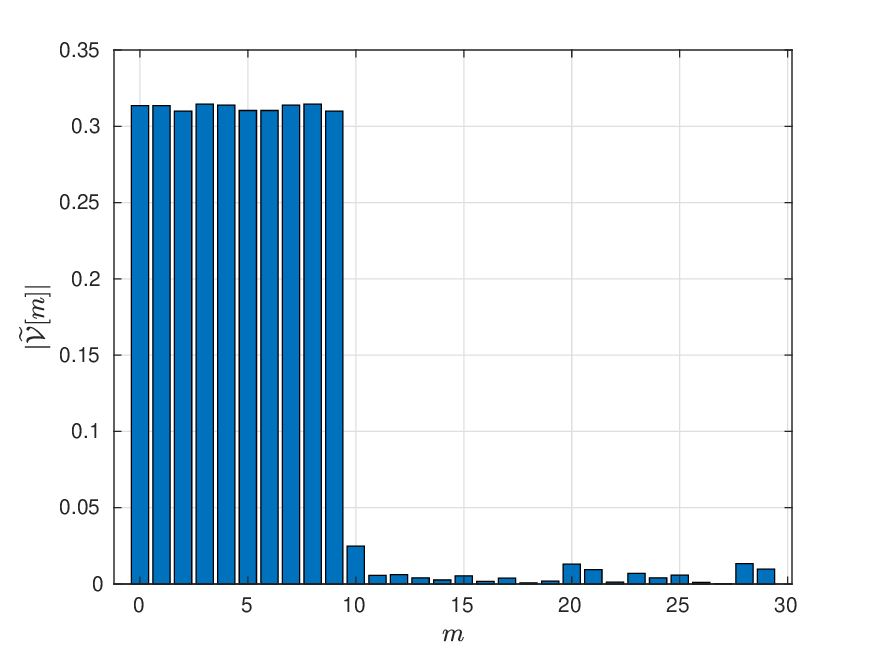} \label{fig:magnitudes of V_tilde}
    }     
	\caption[Magnitudes of $\mathcal{V}$ and $\widetilde{\mathcal{V}}$] 
	{Absolute values of the first $30$ elements of ${\mathcal{V}}$ (left) and $\widetilde{\mathcal{V}}$ (right) using the same $\hat{\Sigma}_x$ as in Figure~\ref{fig:A and A_tilde singular values}, where $r=3$, $L=10$,  $N=10^4$, and $\sigma^2=0.05$. It is evident that the first $L$ elements of ${\mathcal{V}}$ (which encode the phase differences $\widetilde{\varphi}_m - \widetilde{\varphi}_{m-1}$) have constant magnitudes, in agreement with~\eqref{eq:V_mathcal_tilde expression}. Moreover, this also holds approximately for $\widetilde{\mathcal{V}}$, with small fluctuations due to noise and finite sample size.}  \label{fig:magnitudes of V and V_tilde}
\end{figure} 
Next, note that the set of phase differences $\widetilde{\varphi}_m - \widetilde{\varphi}_{m-1}$ for $m=0,\ldots,L-1$ (recalling that $\widetilde{\varphi}_{-1} = \widetilde{\varphi}_{L-1}$), satisfies
\begin{equation}
\sum_{m=0}^{L-1} (\widetilde{\varphi}_m - \widetilde{\varphi}_{m-1}) = 0.
\end{equation}
Therefore, taking the sum over $m=0,\ldots,L-1$ on both sides in~\eqref{eq:V_mathcal_tilde arg expression} yields
\begin{equation}
\sum_{m=0}^{L-1}\operatorname{arg}\left\{ {\mathcal{V}}[m] \right\} = \operatorname{mod}(L \alpha, 2\pi),
\end{equation}
asserting that $\alpha$ must satisfy
\begin{equation}
\alpha = \frac{1}{L}\sum_{\ell=0}^{L-1} \operatorname{arg}\left\{ {\mathcal{V}}[\ell] \right\} + \frac{2\pi k}{L}, \label{eq:alpha expression}
\end{equation}
for some $k\in\{0,\ldots,L-1\}$.
From~\eqref{eq:V_mathcal_tilde arg expression},~\eqref{eq:alpha expression}, and taking $\widetilde{\varphi}_0 = 0$ (in accordance with~\eqref{eq:X_tilde def}), we arrive at
\begin{equation}
\widetilde{\varphi}_m = - \sum_{\ell=1}^m \operatorname{arg}\left\{ {\mathcal{V}}[\ell] \right\} + \frac{m}{L}\sum_{\ell=0}^{L-1} \operatorname{arg}\left\{ {\mathcal{V}}[\ell] \right\} + \frac{2\pi k m}{L}, \label{eq:step 2 angle estimation procedure}
\end{equation}
for $m=1,\ldots,L-1$ and some $k\in\{0,\ldots,L-1\}$ (where $k$ is fixed for all values of $m$), which determines every angle $\widetilde{\varphi}_{m}$ completely up to an additive ambiguity of $2\pi k m/L$ for some $k\in \{ 0,\ldots, L-1\}$ . 
Reviewing the derivation thus far,~\eqref{eq:step 2 angle estimation procedure} is a necessary condition for $\widetilde{X}$ to be Hermitian and PSD (since we derived it from the assumption that $\widetilde{X}$ is Hermitian and PSD). We summarize the derivation up to this point in the following proposition. 
\begin{prop} \label{prop:low-rank procedure necessary condition}
Suppose that $r<\sqrt{L}$ and Condition~\ref{cond:low-rank recovery} holds. If $\widetilde{X}$ of~\eqref{eq:X_tilde def} is Hermitian and PSD, then the angles $\widetilde{\varphi}_1\ldots,\widetilde{\varphi}_{L-1}$ must follow~\eqref{eq:step 2 angle estimation procedure} for some fixed $k\in \{0,\ldots, L-1\}$.
\end{prop}
Although Proposition~\ref{prop:low-rank procedure necessary condition} alone does not guarantee that choosing $\widetilde{\varphi}_1\ldots,\widetilde{\varphi}_{L-1}$ according to~\eqref{eq:step 2 angle estimation procedure} for any particular $k$ leads to $\widetilde{X}$ which is Hermitian and PSD, nor that $\widetilde{X}\in\Omega(\hat{\Sigma}_x)$, those two facts actually follow when combining Proposition~\ref{prop:low-rank procedure necessary condition} with Proposition~\ref{prop:fundamental ambiguities}. We then get the following result. 
\begin{prop} \label{prop:low-rank procedure derivation summary}
Suppose that $r<\sqrt{L}$ and Condition~\ref{cond:low-rank recovery} holds. Then, $\widetilde{X}$ of the form of~\eqref{eq:X_tilde def} is Hermitian and PSD if and only if $\widetilde{X}\in\Omega(\hat{\Sigma}_x)$. Moreover, if the angles $\widetilde{\varphi}_1\ldots,\widetilde{\varphi}_{L-1}$ follow~\eqref{eq:step 2 angle estimation procedure} for any fixed $k\in \{0,\ldots, L-1\}$ then
\begin{equation}
\widetilde{\varphi}_m = \varphi_m + \frac{2\pi k^{'} m}{L}, \qquad m = 1,\ldots,L-1, \label{eq:varphi_tilde feasible solutions}
\end{equation}
for some $k^{'}\in \{0,\ldots, L-1\}$,
where $\varphi_1,\ldots,\varphi_{L-1}$ are from~\eqref{eq:X def}, and consequently $\widetilde{X}\in \Omega(\hat{\Sigma}_x)$.
\end{prop}
\begin{proof}
Proposition~\ref{prop:low-rank procedure necessary condition} asserts that for $\widetilde{X}$ to be Hermitian and PSD, the vector of correcting angles $[\widetilde{\varphi}_1,\ldots,\widetilde{\varphi}_{L-1}]$ must be chosen from the $L$ different options corresponding to different $k\in \{0,\ldots,L-1\}$ in~\eqref{eq:step 2 angle estimation procedure} (the options are different since $\widetilde{\varphi}_1$ is clearly different for every $k\in\{0,\ldots,L-1\}$). On the other hand, taking $[\widetilde{\varphi}_1\ldots,\widetilde{\varphi}_{L-1}]$ according to~\eqref{eq:varphi_tilde feasible solutions} gives
\begin{align}
\widetilde{X} &= X \odot \operatorname{Circulant} \{ [1, e^{-\imath ({\varphi}_1 + 2\pi k^{'}/L)},\ldots,e^{-\imath ({\varphi}_{L-1} + 2\pi k^{'} (L-1)/L)}]\} \nonumber \\ 
&= \hat{\Sigma}_x \odot \operatorname{Circulant} \{ [1, e^{-\imath 2\pi k^{'}/L},\ldots,e^{-\imath 2\pi k^{'} (L-1)/L}]\} = \hat{\Sigma}_x \odot \operatorname{Circulant} \{ f_{k^{'}}\} \nonumber \\
&= \hat{\Sigma}_x \odot f_{k^{'}}^* f_{k^{'}} = \hat{\Sigma}_x \odot f_{L-k^{'}} f_{L-k^{'}}^* =  \operatorname{diag}(f_{L-k^{'}}) \cdot \hat{\Sigma}_x \cdot \operatorname{diag}(f_{L-k^{'}}^*) \in \Omega(\hat{\Sigma}_x),
\end{align}
asserting (via Proposition~\ref{prop:fundamental ambiguities}) that $\widetilde{X}$ is Hermitian and PSD, for every $k^{'}\in\{0,\ldots,L-1\}$. Therefore, there are $L$ different choices for  $[\widetilde{\varphi}_1\ldots,\widetilde{\varphi}_{L-1}]$, given explicitly by~\eqref{eq:varphi_tilde feasible solutions}, that result in $\widetilde{X}$ which is Hermitian and PSD. Consequently, choosing $[\widetilde{\varphi}_1,\ldots,\widetilde{\varphi}_{L-1}]$ according to~\eqref{eq:step 2 angle estimation procedure} must coincide with~\eqref{eq:varphi_tilde feasible solutions}, with some one-to-one mapping between the values of the indices $k$ and $k^{'}$, thus establishing all claims of Proposition~\ref{prop:low-rank procedure derivation summary}.
\end{proof}

Considering the finite-sample case, it is clear that we do not have access to $X$ (of~\eqref{eq:X def}) nor $\widetilde{X}$ (of~\eqref{eq:X_tilde def}). Therefore, we replace $X$ and $\widetilde{X}$ with $\widetilde{C}_x$ (of~\eqref{eq:C_x_tilde def}) and $\widetilde{\hat{\Sigma}}_x$ (of~\eqref{eq:covariance estimator with phase}), respectively, and denote by $\widetilde{(\cdot)}$ the corresponding finite-sample analogues of all quantities defined in this section. Figure~\ref{fig:A_tilde singular vals} depicts the behavior of the smallest singular values of $\widetilde{W}$ (the finite-sample analogue of $W$), and Figure~\ref{fig:magnitudes of V_tilde} illustrates the magnitudes of $\widetilde{\mathcal{V}}$ (the right singular vector of $\widetilde{W}$ corresponding to its smallest singular value).
The resulting procedure for resolving the diagonal phase ambiguities in the finite-sample case is detailed in Algorithm~\ref{alg:step 2 direct}. We mention that if $r$ is unknown, we take it to be the maximal rank allowed by Algorithm~\ref{alg:step 2 direct}, which is the largest $r$ such that $r < \sqrt{L}$ (for the system~\eqref{eq:A beta = 0} not to be underdetermined).

\begin{algorithm}
\caption{Circulant phase retrieval}\label{alg:step 2 direct}
\begin{algorithmic}[1]
\Statex{\textbf{Required:} A matrix $X\in\mathbb{C}^{L\times L}$ following or approximating~\eqref{eq:X def} (e.g. $\widetilde{C}_x$ from Algorithm~\ref{alg:step 1 alg}}). 
\State If the rank $r$ is unknown, set $r = \lceil \sqrt{L}-1 \rceil$. \label{step:rank set}
\State For $i=0,\ldots,L-1$ do
\begin{enumerate} [label=(\alph*)]
\item Form the $L\times L$ matrices ${H}_{i,i}$ and ${H}_{i,i+1}$ according to~\eqref{eq:H_ij_tilde def}.
\item Take the first $r^2$ eigenvectors of ${H}_{i,i}$ (corresponding to its $r^2$ largest eigenvalues in absolute value) and store them in the $L\times r^2$ matrix ${V}^{(i)}$. 
\end{enumerate}
\item Form the $L^3\times (L+r^4 L)$ matrix ${W}$ according to~\eqref{eq:Z_tilde def},~\eqref{eq:M_tilde def}, and~\eqref{eq:A_mathcal_tilde def}.
\item Evaluate ${\mathcal{V}}$, which is the right singular vector of ${W}$ corresponding to its smallest singular value. See Remark~\ref{remark:efficient evaluation of V_mathcal_tilde} for an efficient evaluation procedure.  
\item Take $\widetilde{\varphi}_0=0$ and compute $\widetilde{\varphi}_1,\ldots,\widetilde{\varphi}_{L-1}$  according to~\eqref{eq:step 2 angle estimation procedure}.
\item Output $\widetilde{\Sigma}_x = F^* \widetilde{\hat{\Sigma}}_x F$, where $\widetilde{\hat{\Sigma}}_x$ is given by~\eqref{eq:covariance estimator with phase} and $F$ is the DFT matrix of~\eqref{eq:DFT matrix and vectors def}.
\end{algorithmic}
\end{algorithm}

The following theorem establishes the consistency of the estimator $\widetilde{\Sigma}_x$ computed by Algorithm~\ref{alg:step 2 direct}.
\begin{thm}[Consistency of Algorithm~\ref{alg:step 2 direct}] \label{thm:consistency of step 2, direct approach}
Let $\widetilde{C}_x$ be the input to Algorithm~\ref{alg:step 2 direct} (replacing $X$). Suppose that $r<\sqrt{L}$, Condition~\ref{cond:low-rank recovery} holds, and that $H_{i,i}$ from~\eqref{eq:H_ii expression} has $r^2$ distinct non-zero eigenvalues for every $i\in\{0,\ldots,L-1\}$. Then,
\begin{align}
\min_{\ell\in \{0,\ldots,L-1\}}\left\Vert \widetilde{{\Sigma}}_x - R_{\ell,\ell}\left\{ {\Sigma}_x \right\} \right\Vert_F \underset{N\rightarrow\infty, \; \text{a.s.}}{\longrightarrow} 0, \label{eq:Sigma_x estimation consistency}
\end{align} 
where $\widetilde{{\Sigma}}_x$ is the output of Algorithm~\ref{alg:step 2 direct}.
\end{thm}
The proof of Theorem~\ref{thm:consistency of step 2, direct approach} is provided in Appendix~\ref{appendix:Proof of consistency of step 2}.

\begin{remark}[Efficient evaluation of $\mathcal{V}$ and $\widetilde{\mathcal{V}}$] \label{remark:efficient evaluation of V_mathcal_tilde}
In general, ${W}$ is of size $ L^3 \times \mathcal{O}(L^3)$ (assuming $r$ is unknown), and hence the singular-value decomposition (SVD) of ${W}$ becomes computationally prohibitive even for moderate values of $L$. Therefore, it is essential to compute ${\mathcal{V}}$ without the full SVD of $W$, while exploiting the special structure of ${W}$. In particular, note that since the columns of $V^{(i)}$ (the $r^2$ eigenvectors of $H_{i,i}$) are orthonormal, then also the columns of ${Z}^{(i)} = \overline{V^{(i+1)}} \otimes V^{(i)}$ (see~\eqref{eq:Z_tilde def}) are orthonormal due to the definition of the Kronecker product. Consequently, the $r^4 L$ rightmost columns of ${W}$ are orthonormal due to the block-diagonal structure in~\eqref{eq:A_mathcal_tilde def}. It then follows that the matrix ${W}^* {W}$ is much sparser than ${W}$, see Figure~\ref{fig:AA structure}, as it includes only $\mathcal{O}(L^4)$ nonzero elements (compared to $\mathcal{O}(L^5)$ in ${W}$). As ${\mathcal{V}}$ is also the eigenvector of ${W}^* {W}$ corresponding to its smallest eigenvalue, ${\mathcal{V}}$ can be computed efficiently by the inverse power method using the conjugate-gradients algorithm for inverting ${W}^* {W}$ at each iteration. The above discussion applies equivalently to the evaluation of $\widetilde{\mathcal{V}}$ from $\widetilde{W}^* \widetilde{W}$.
\end{remark}

\section{Numerical experiments} \label{section:numerical examples}
Next, we report our experimental findings on the recovery of $\Sigma_x$ from the measurements $y_1,\ldots,y_N$, using Algorithm~\ref{alg:step 1 alg} followed by Algorithm~\ref{alg:step 2 direct}. 
We use the following measure of discrepancy to evaluate the estimation error:
\begin{equation}
\text{Error} := \frac{\underset{\ell\in\{0,\ldots,L-1\}}{\min} \left\Vert \widetilde{\Sigma}_x - R_{\ell,\ell}\{ \Sigma_x \} \right\Vert_F^2}{\left\Vert  \Sigma_x \right\Vert_F^2}, \label{eq:error def}
\end{equation}
where $\widetilde{\Sigma}_x$ is the estimator of $\Sigma_x$ obtained from the output of Algorithm~\ref{alg:step 2 direct} while treating the rank $r$ as unknown (setting $r=\lceil \sqrt{L} - 1\rceil$ per step~\ref{step:rank set} in Algorithm~\ref{alg:step 2 direct}). In all our experiments, we generate the covariance matrix $\Sigma_x$ randomly as follows. We sample the eigenvalues $\lambda_1,\ldots,\lambda_{r}$ of $\Sigma_x$ uniformly from $[0,1]$, and normalize them such that $\sum_{i=1}^r \lambda_i = 1$, essentially enforcing a fixed signal power of $1$ (i.e. $\mathbb{E}\Vert x\Vert^2 = 1$). The eigenvectors $v_1,\ldots,v_r$ of $\Sigma_x$ are sampled uniformly from the $(L-1)$-sphere. After generating the covariance matrix $\Sigma_x$, the observations $y_1,\ldots,y_N$ are drawn according to~\eqref{eq:MRFA model def} and~\eqref{eq:x model} using a uniform distribution for the cyclic shifts~$s$. 

We begin with several experiments testing the performance of our algorithms in the complex-valued case ($v_i\in \mathbb{C}^L$, $a_i \sim \mathbb{C}\mathcal{N}(0,\lambda_i)$ and $\eta \sim \mathbb{C}\mathcal{N}(0,\sigma^2 I_L)$). We first explore the ability of our algorithms to recover $\Sigma_x$ for different ranks $r$ and signal lengths $L$. For these experiments, we use $N=10^5$, and consider the maximal error among $200$ trials. The results are shown in Figure~\ref{fig:low-rank recovery r vs L}. As supported by Theorem~\ref{thm:consistency of step 2, direct approach}, small estimation errors are always achieved when $r<\sqrt{L}$. However, since the algorithm cannot handle the case of $r\geq \sqrt{L}$ (the linear system in~\eqref{eq:A beta = 0} becomes underdetermined), the worst-case estimation errors rapidly increase with $r$ for $r\geq \sqrt{L}$.

Continuing, we investigate the behavior of the estimation error as a function of the number of observations $N$. In this experiment, the error is averaged over $200$ trials. Figure~\ref{fig:E vs N complex} displays the estimation error of Algorithm~\ref{alg:step 2 direct} as a function of $N$, for $L=26$, $r\in\{2,5\}$ and $\sigma^2 \in \{0,0.01,0.05\}$. As expected from Theorem~\ref{thm:consistency of step 2, direct approach}, the error decreases with $N$ for all fixed values of $r$ and $\sigma^2$. In this regard, the empirical results suggest that the error is proportional to $1/N$. Furthermore, it is evident that the existence of noise simply shifts the error curves to the right, such that more observations (by a constant factor) are required to achieve the same estimation error as without noise. Note that in this example, $\sigma^2 = 0.05$ corresponds to a noise magnitude (given by $\sigma^2 L$) approximately equal to the signal's strength (normalized to be $1$). Even in this challenging regime, an accurate estimation of $\Sigma_x$ is achieved when $N\approx 10^5$, implying that our method can successfully cope with high levels of noise.
Also, as evident from Figure~\ref{fig:E vs N complex}, the estimation error grows with $r$. This is due to the fact that the fourth moment of $\hat{y}$ (and the trispectrum in particular) is harder to estimate for larger ranks, since the variability in the quantities $\hat{y}[k_1]\overline{\hat{y}[k_2]}\hat{y}[k_3]\overline{\hat{y}[k_4]}$ increases with $r$. 

\begin{figure}
  \centering  	
    \includegraphics[width=0.7\textwidth]{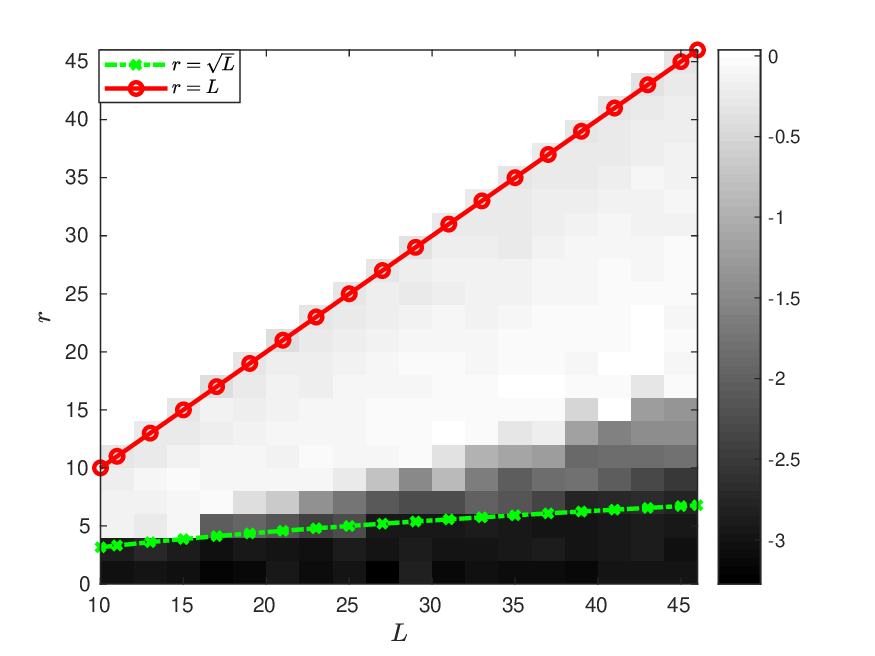}
	\caption[Algorithm~\ref{alg:step 2 direct} r vs L] 
	{Estimation error of our algorithms (Algorithm~\ref{alg:step 1 alg} followed by Algorithm~\ref{alg:step 2 direct}) in the complex-valued case, using $N=10^5$. For every $r$ and $L$ we display the maximal $\log_{10}\left\{\text{Error}\right\}$ over $200$ trials.}  \label{fig:low-rank recovery r vs L}
\end{figure}

\begin{figure}
  \centering
  	{
    \includegraphics[width=0.7\textwidth]{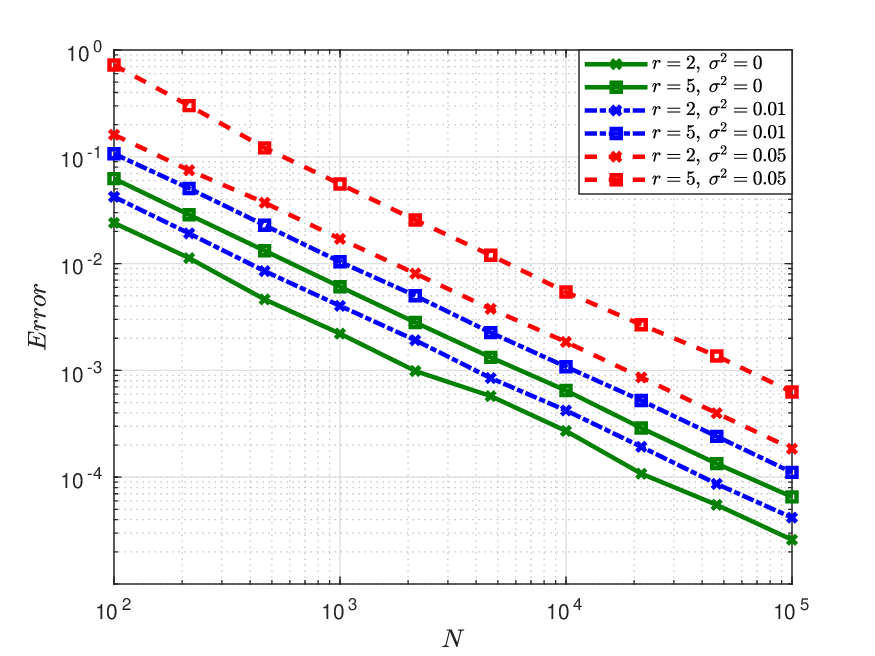} 
    }    
	\caption[L vs r] 
	{Estimation error of our algorithms (Algorithm~\ref{alg:step 1 alg} followed by Algorithm~\ref{alg:step 2 direct}) in the complex-valued case as a function of $N$, for $L=26$, and several ranks and noise levels. The errors were averaged over $200$ trials.	}  \label{fig:E vs N complex}
\end{figure}

Last, we demonstrate the performance of our algorithms for the real-valued case ($v_i\in \mathbb{R}^L$, $a_i \sim \mathcal{N}(0,\lambda_i)$ and $\eta \sim \mathcal{N}(0,\sigma^2 I_L)$), and we note that the difference in our algorithms between the real-valued and the complex-valued cases lies only in step~\ref{alg 1:solve opt prob} of Algorithm~\ref{alg:step 1 alg} (as solving~\eqref{eq:optim for diagonals} in the complex-valued case is replaced by solving~\eqref{eq:optim for diagonals real-valued case} in the real-valued case). The results are shown in Figures~\ref{fig:low-rank recovery r vs L real} and~\ref{fig:E vs N real}, which are analogous to Figures~\ref{fig:low-rank recovery r vs L} and~\ref{fig:E vs N complex} in the complex-valued case, respectively. It is evident that the performance of our algorithms in the real-valued case is very similar to that of the complex-valued case, with almost identical behavior in all aspects, albeit slightly larger estimation errors.

\begin{figure}
  \centering  	
    \includegraphics[width=0.7\textwidth]{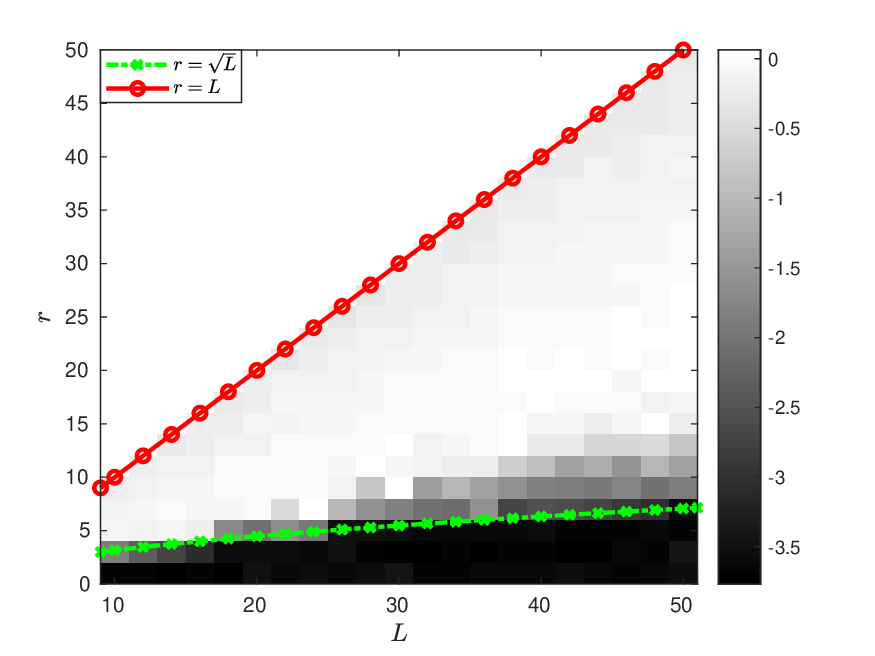}
	\caption[Algorithm~\ref{alg:step 2 direct} r vs L] 
	{Estimation error of our algorithms (Algorithm~\ref{alg:step 1 alg} followed by Algorithm~\ref{alg:step 2 direct}) in the real-valued case, using $N=10^5$. For every $r$ and $L$ we display the maximal $\log_{10}\left\{\text{Error}\right\}$ over $200$ trials.}  \label{fig:low-rank recovery r vs L real}
\end{figure}

\begin{figure}
  \centering
  	{
    \includegraphics[width=0.7\textwidth]{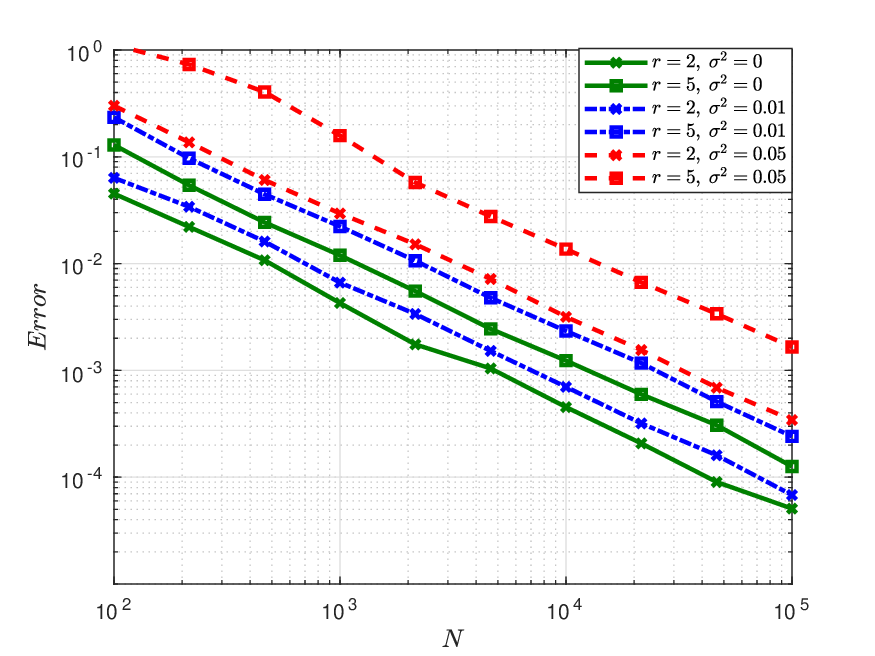} 
    }    
	\caption[L vs r] 
	{Estimation error of our algorithms (Algorithm~\ref{alg:step 1 alg} followed by Algorithm~\ref{alg:step 2 direct}) in the real-valued case as a function of $N$, for $L=26$, and several ranks and noise levels. The errors were averaged over $200$ trials.	}  \label{fig:E vs N real}
\end{figure}

\section{Summary and discussion}
In this work, we considered the problem of recovering the covariance matrix of a random signal $x$ observed through unknown translations and corrupted by noise, where the signal $x$ is low-rank (i.e. it follows the factor model~\eqref{eq:x model} with $r$ much smaller than $L$). We have shown that unique recovery of the covariance matrix is possible (up to a set of fundamental ambiguities, see Proposition~\ref{prop:fundamental ambiguities}) when  $r<\sqrt{L}$ and Condition~\ref{cond:low-rank recovery} holds. We provided statistically-consistent polynomial-time estimation procedures, and concluded with numerical simulations corroborating our theoretical findings. 

There are many open questions emerging from this work, giving rise to several possible future research directions.
First, we discuss future research directions associated with the model~\eqref{eq:MRFA model def}--\eqref{eq:x model}. While we have shown that recovery of the covariance matrix from the power spectrum and the trispectrum is possible when $r<\sqrt{L}$ and impossible when $r=L$ (i.e. the covariance is full-rank), it is of interest to determine tighter upper and lower bounds on the rank $r$ characterizing when the recovery can be attained (both theoretically and using polynomial-time algorithms), possibly determining the exact phase transition, namely the set of ranks above which recovery is no longer possible. Moreover, even when covariance estimation from the power spectrum and the trispectrum alone is impossible, it is of interest to investigate the advantages of adding higher-order moments.
Last, while we established statistical consistency for our estimators, it is favorable to investigate their estimation errors in terms of the quantities governing our model ($N$, $L$, $r$, and $\sigma^2$).

Aside from the above-mentioned research directions associated with the model~\eqref{eq:MRFA model def}--\eqref{eq:x model}, it is worthwhile to consider various extensions of this model. First, one could replace the normal distribution with a broader family of distributions, allowing for more general factor models. Second, the one-dimensional setting (implicitly assumed in~\eqref{eq:MRFA model def}--\eqref{eq:x model}) could be extended to higher dimensions, with other group actions replacing the cyclic shift $R_s$ in~\eqref{eq:cyclic shift def}. For example, one-dimensional signals could be replaced with two-dimensional images, where cyclic shifts are replaced with in-plane rotations. This extension could have important applications in rotation-invariant processing of datasets of two-dimensional images, see for example~\cite{zhao2013fourier,landa2017steerable,landa2018steerable,ma2018heterogeneous}.

\section*{Acknowledgements}
We would like to thank Nicolas Boumal for several enlightening conversations, and to Boaz Nadler for useful comments and suggestions.
This research was supported by the European Research Council (ERC) under the European Union’s Horizon 2020 research and innovation programme (grant agreement 723991 - CRYOMATH).

\begin{appendices}

\section{Proof of Proposition~\ref{prop:f_2,f_4,f_6 explicit form}} \label{Appendix: Proof of f_2,f_4 explicit form}
The expression for $P_y$ follows directly from its definition~\eqref{eq:p_y def}. We now prove~\eqref{eq: T_y expression Sigma}. Towards this end, we use existing results on the moments of the complex normal distribution (see~\cite{sultan1996moments}). However, since $\hat{\Sigma}_x$ is not invertible when $r<L$, we cannot claim that $\hat{x}\sim \mathbb{C}\mathcal{N}(0,\hat{\Sigma}_x)$. Therefore, we first treat the case of $r=L$, and then extend our result to the case of $r<L$ by a continuity argument.

Suppose that $r=L$, i.e. $\lambda_i>0$ for all $i$. Then $\hat{x}\sim\mathbb{C}\mathcal{N}(0,\hat{\Sigma}_x)$, and according to Theorem~5 in~\cite{sultan1996moments} we have
\begin{equation}
\mathbb{E} [ \hat{x}\hat{x}^* \otimes {\hat{x}\hat{x}^*} ] = \left( I_{L^2} + I_{L,L} \right) \left( \hat{\Sigma}_x \otimes \hat{\Sigma}_x \right), \label{eq:M_4 expression ref}
\end{equation}
where $\otimes$ is the Kronecker product, and $I_{L,L}$ is the $L^2\times L^2$ commutation matrix, given by
\begin{equation}
I_{L,L}[k_1+k_3 L,m] = 
\begin{dcases}
1, & m = k_1 L + k_3, \\
0, & \text{otherwise}, \label{eq:commutation matrix def}
\end{dcases}
\end{equation}
for $k_1,k_3\in\{0,\ldots,L-1\}$, and $m\in\{0,\ldots,L^2-1\}$.
Note that we can write $M^{(4)}_{\hat{x}}$ (defined by replacing $\hat{y}$ in~\eqref{eq:M_2 and M_4 def} with $\hat{x}$) as 
\begin{equation}
M^{(4)}_{\hat{x}}[k_1,k_2,k_3,k_4] = \left\{\mathbb{E} [ \hat{x}\hat{x}^* \otimes {\hat{x}\hat{x}^*} ]\right\}[k_1+k_3 L,k_2 + k_4 L],
\end{equation}
and it follows from~\eqref{eq:M_4 expression ref} and~\eqref{eq:commutation matrix def} that
\begin{align}
M^{(4)}_{\hat{x}}[k_1,k_2,k_3,k_4] &= \{ \hat{\Sigma}_x \otimes \hat{\Sigma}_x \} [k_1 + k_3 L, k_2 + k_4 L] + \{ \hat{\Sigma}_x \otimes \hat{\Sigma}_x \} [k_1 L + k_3, k_2 + k_4 L] \nonumber \\
&= \hat{\Sigma}_x[k_1,k_2] \hat{\Sigma}_x[k_3,k_4] +  \hat{\Sigma}_x[k_3,k_2] \hat{\Sigma}_x[k_1,k_4] \nonumber \\
&= \hat{\Sigma}_x[k_1,k_2] \overline{\hat{\Sigma}_x[k_4,k_3]} +  \overline{\hat{\Sigma}_x[k_2,k_3]} {\hat{\Sigma}_x[k_1,k_4]},
\end{align}
where we used the fact that $\hat{\Sigma}_x$ is Hermitian.
Therefore, by substituting $k_4 = k_1 - k_2 + k_3$ we get
\begin{align}
T_y[k_1,k_2,k_3] &= M^{(4)}_{\hat{y}}[k_1,k_2,k_3,k_1-k_2+k_3] = M^{(4)}_{\hat{x}}[k_1,k_2,k_3,k_1-k_2+k_3] \nonumber \\ 
&= \hat{\Sigma}_x[k_1,k_2] \overline{\hat{\Sigma}_x[k_1-k_2+k_3,k_3]} +  {\hat{\Sigma}_x[k_1,k_1-k_2+k_3]} \overline{\hat{\Sigma}_x[k_2,k_3]}. \label{eq:T_y expression for r=L}
\end{align}
Last, we extend the result above to the case of $r<L$. It is easy to verify that $T_y$ is continuous in $\lambda_1,\ldots,\lambda_{L}$, since $T_y$ is a subset of 
\begin{equation}
M_{\hat{x}}^{(4)}[k_1,k_2,k_3,k_4] = \mathbb{E}\left[\hat{x}[k_1] \overline{\hat{x}[k_2]} \hat{x}[k_3] \overline{\hat{x}[k_4]} \right] = \sum_{i,j,\ell,m = 0}^{L-1} \mathbb{E}\left[ a_i \overline{a_j} a_\ell \overline{a_m}\right] v_i[k_1] \overline{v_j[k_2]} v_\ell[k_3] \overline{v_m[k_4]}, 
\end{equation}
which is a polynomial in $\lambda_1,\ldots,\lambda_{L}$ (since $a_i\sim\mathbb{C}\mathcal{N}(0,\lambda_i)$) for all values of $\lambda_1,\ldots,\lambda_L$ (including zero).
Therefore, fixing $r$ and taking $\lambda_i \rightarrow 0$ for $i>r$ on both sides of~\eqref{eq:T_y expression for r=L}, we get due to the continuity of $T_y$ that~\eqref{eq:T_y expression for r=L} also holds for any $r<L$ (where $\lambda_i = 0$ for $i>r$).

\section{Proof of Lemma~\ref{lem:P_y and T_y consequence}} \label{Appendix: Proof of Lemma P_y and T_y consequence}
For $X\in\mathbb{C}^{L\times L}$ it is convenient to define $D_m\in\mathbb{C}^L$ as
\begin{equation}
D_m[k] = X[k,\operatorname{mod} ( k+m, L)], \label{eq:D_m def}
\end{equation}
for $m,k\in \{ 0,\ldots,L-1\}$. In other words, $D_m$ is the $m$'th diagonal of $X$ with circulant wrapping, and is analogous to $d_m$ of~\eqref{eq:d_m def} for $\sigma=0$ (which is the $m$'th diagonal of $\hat{\Sigma}_x$ with circulant wrapping).
By~\eqref{eq:D_m def},~\eqref{eq: P_mathcal expression}, and~\eqref{eq: T_mathcal expression}, the set of equations~\eqref{eq:P_y and T_y equations} can be written as 
\begin{align}
&P_y[k_1] = \mathcal{P}(X)[k_1] = D_0[k_1], \label{eq:d_0 expression proof of Lemma}\\
&T_y[k_1,k_2,k_3] = \mathcal{T}(X)[k_1,k_2,k_3] = D_{k_2-k_1}[k_1] \overline{D_{k_2-k_1}[k_3-k_2+k_1]} + D_{k_3-k_2}[k_1] \overline{D_{k_3-k_2}[k_2]}. \label{eq:T_y proof of lemma perliminary}
\end{align}
By taking (with some abuse of notation) $k_2 = k_1 + m$, $k_3 = k_2 +m$, we can rewrite \eqref{eq:T_y proof of lemma perliminary} more conveniently (and analogously to~\eqref{eq:T_y expression diagonals}) as
\begin{equation}
T_y[k_1,k_1+m,k_2+m] = D_m[k_1] \overline{D_m[k_2]} + D_{k_2-k_1}[k_1] \overline{D_{k_2-k_1}[k_1+m]}, \label{eq:d_m expression proof of Lemma}
\end{equation}
for $k_1,k_2,m\in\{0,\ldots,L-1\}$ (noting that $k_1$ and $k_2$ in~\eqref{eq:d_m expression proof of Lemma} are not equivalent to $k_1$ and $k_2$ in~\eqref{eq:T_y proof of lemma perliminary}).
Now, for the ``if'' part of the ``if and only if'' statement of Lemma~\ref{lem:P_y and T_y consequence}, it is straightforward to verify that taking $X$ according to~\eqref{eq:P_y and T_y circulant phase ambiguities}, namely $D_m[k] = \hat{\Sigma}_x[k,k+m] e^{\imath \varphi_m}$ with $\varphi_0 =0$, satisfies both~\eqref{eq:d_0 expression proof of Lemma} and~\eqref{eq:d_m expression proof of Lemma} 
(substituting~\eqref{eq: P_y expression Sigma} and~\eqref{eq: T_y expression Sigma} into~\eqref{eq:d_0 expression proof of Lemma} and~\eqref{eq:d_m expression proof of Lemma}) since the terms $e^{\imath \varphi_m}$ cancel out. We now consider the other direction, namely the ``only if'' part of the statement. Suppose that the set of equations in~\eqref{eq:P_y and T_y equations} hold. We now prove the required result by the following three steps. First, taking $m=0$ in~\eqref{eq:d_m expression proof of Lemma} and substituting~\eqref{eq:d_0 expression proof of Lemma} gives
\begin{equation}
T_y[k_1,k_1,k_2] = P_y[k_1] P_y[k_2] + | D_{k_2-k_1}[k_1] |^2,
\end{equation}
which establishes (by substituting~\eqref{eq: P_y expression Sigma} and~\eqref{eq: T_y expression Sigma}) that 
\begin{equation}
| D_{k_2-k_1}[k_1] | = | \hat{\Sigma}_x[k_1,k_2] | \label{eq:D magnitude equation}
\end{equation}
for every $k_1,k_2\in\{0,\ldots,L-1\}$. 
Second, taking $m=1$ and $k_2 = k_1 + 1$ in~\eqref{eq:d_m expression proof of Lemma} leads to
\begin{equation}
T_y[k_1,k_1+1,k_1+2] = 2 D_{1}[k_1] \overline{D_{1}[k_1+1]},
\end{equation}
for $k_1 \in \{0,\ldots,L-1\}$. Substituting~\eqref{eq: T_y expression Sigma} in the above equation, we have
\begin{equation}
2 \hat{\Sigma}[k_1,k_1+1] \overline{ \hat{\Sigma}[k_1+1,k_1+2]} = 2 D_{1}[k_1] \overline{D_{1}[k_1+1]}. \label{eq:D_1 iterative expression}
\end{equation}
Now, taking $k_1=0$ and $k_2=1$ in~\eqref{eq:D magnitude equation} establishes that $D_1[0] =\hat{\Sigma}_x[0,1] e^{\imath \varphi_1}$ with some $\varphi_1\in [0,2\pi)$. Then, ~\eqref{eq:D_1 iterative expression} determines $D_1[k]$ completely for all $k$ by an iterative procedure, as $D_1[1]$ is obtained from $D_1[0]$, $D_1[2]$ is obtained from $D_1[1]$, and so on, where each element is obtained by dividing both sides of~\eqref{eq:D_1 iterative expression} by $D_1[k]$ (we never divide by zero from the assumption in Lemma~\ref{lem:P_y and T_y consequence} that $\hat{\Sigma}_x[k_1,k_2]\neq 0$). Consequently, we have that
\begin{equation}
D_1[k] = \hat{\Sigma}[k,k+1] e^{\imath \varphi_1}. \label{eq:D_1 expression}
\end{equation} 
Last, taking $k_2-k_1=1$ in~\eqref{eq:d_m expression proof of Lemma} gives
\begin{equation}
T_y[k_1,k_1+m,k_1+m+1] = D_m[k_1] \overline{D_m[k_1+1]} + D_{1}[k_1] \overline{D_{1}[k_1+m]},
\end{equation}
and substituting~\eqref{eq: T_y expression Sigma} together with~\eqref{eq:D_1 expression} establishes that
\begin{equation}
\hat{\Sigma}[k_1,k_1+m] \overline{ \hat{\Sigma}[k_1+1,k_1+m+1]} = D_m[k_1] \overline{D_m[k_1+1]}. \label{eq:D_m iterative expression}
\end{equation}
Repeating our previous argumentation (for $m=1$) now for every $m\in\{2,\ldots,L-1\}$, we take $k_1=0$ and $k_2=m$ in~\eqref{eq:D magnitude equation}, which establishes that $D_m[0] =\hat{\Sigma}_x[0,m] e^{\imath \varphi_m}$ with some $\varphi_m\in [0,2\pi)$, and then~\eqref{eq:D_m iterative expression} determines $D_m[k]$ for every $k=1,\ldots,L-1$, by the previously mentioned iterative process. Therefore, we have that
\begin{equation}
D_m[k] = \hat{\Sigma}[k,k+m] e^{\imath \varphi_m},
\end{equation}
for all $m=1,\ldots,L-1$ and $k=0,\ldots,L-1$, which concludes the proof. 

\section{Proof of Lemma~\ref{lem:low-rank diagonal phase retrieval}} \label{Appendix: Proof of lemma low-rank diagonal phase retrieval}
We begin with the case of $r=1$. Note that
\begin{equation}
\widetilde{X} = \hat{\Sigma}_x \odot \operatorname{Circulant}\{1,e^{\imath (\varphi_1 - \widetilde{\varphi}_1)},\ldots,e^{\imath (\varphi_{L-1} - \widetilde{\varphi}_{L-1})}\}.
\end{equation}
Let us define
\begin{equation}
\varphi_m^{'} := \varphi_m - \widetilde{\varphi}_m,
\end{equation}
for $m=1,\ldots,L-1$, and it follows that we can write $\widetilde{X}$ as
\begin{equation}
\widetilde{X} = \lambda_1 \hat{v}_1 \hat{v}_1^* \odot \operatorname{Circulant}\{1,e^{\imath \varphi_1^{'}},\ldots,e^{\imath \varphi_{L-1}^{'}}\} = \lambda_1 \operatorname{diag}\{\hat{v}_1\} \cdot \operatorname{Circulant}\{1,e^{\imath \varphi_1^{'}},\ldots,e^{\imath \varphi_{L-1}^{'}}\} \cdot  \operatorname{diag}\{\overline{\hat{v}_1}\}.
\end{equation}
Suppose that $\widetilde{X}$ solves Problem~\ref{problem:PSD circulant phase programming}, namely $\widetilde{X}$ is Hermitian and PSD. Recall that the \textit{inertia} of a Hermitian matrix $A$ is the triplet $\{n_0\{A\},n_+\{A\},n_-\{A\}\}$ describing the number of zero, positive, and negative eigenvalues, respectively, of $A$. Since $\widetilde{X}$ is Hermitian and PSD, all of its eigenvalues are non-negative, hence $n_-\{\widetilde{X}\}=0$, and by Sylvester's law of inertia, the matrix
\begin{equation}
\left(\operatorname{diag}\{\hat{v}_1\}\right)^{-1} \cdot \widetilde{X} \cdot \left(\operatorname{diag}\{\overline{\hat{v}_1}\}\right)^{-1} = \lambda_1 \operatorname{Circulant}\{1,e^{\imath \varphi_1^{'}},\ldots,e^{\imath \varphi_{L-1}^{'}}\}, \label{eq:X rank 1 similarity}
\end{equation}
is also Hermitian and PSD, since it preserves the inertia of $\widetilde{X}$ (where $\left(\operatorname{diag}\{\hat{v}_1\}\right)^{-1}$ is well-defined since $\hat{v}_1[k]\neq 0$ from the assumptions of Lemma~\ref{lem:low-rank diagonal phase retrieval}). Now, it is well-known that a circulant matrix can be diagonalized by the DFT matrix~\eqref{eq:DFT matrix and vectors def}, and in particular, we can write
\begin{equation}
\operatorname{Circulant}\{1,e^{\imath \varphi_1^{'}},\ldots,e^{\imath \varphi_{L-1}^{'}}\} = \sum_{i=1}^L \mu_i \left(\frac{f_i}{\sqrt{L}}\right) \left(\frac{f_i}{\sqrt{L}}\right)^*,
\end{equation}
where $f_i$ is the $i$'th DFT vector defined in~\eqref{eq:DFT matrix and vectors def}, and $\mu_1,\ldots,\mu_L$ are the eigenvalues of $\operatorname{Circulant}\{1,e^{\imath \varphi_1^{'}},\ldots,e^{\imath \varphi_{L-1}^{'}}\}$, which are non-negative as shown in~\eqref{eq:X rank 1 similarity}. We now prove that $\mu_1,\ldots,\mu_L$ are all non-negative only if they are all zero except for one of them. Note that 
\begin{equation}
\operatorname{Trace}\left\{ \operatorname{Circulant}\{1,e^{\imath \varphi_1^{'}},\ldots,e^{\imath \varphi_{L-1}^{'}}\}\right\} = L, \qquad \left\Vert \operatorname{Circulant}\{1,e^{\imath \varphi_1^{'}},\ldots,e^{\imath \varphi_{L-1}^{'}}\} \right\Vert_F^2 = L^2,
\end{equation}
and therefore
\begin{equation}
\sum_{i=1}^L \mu_i = L, \qquad \sum_{i=1}^L \mu_i^2 = L^2. \label{eq:mu_i equations}
\end{equation}
When combining both of the above equations (in particular, squaring both sides of the left equation and subtracting the right equation), we have that
\begin{equation}
\sum_{i\neq j} \mu_i\mu_j = 0, \label{eq:mu products equation}
\end{equation}
with $\mu_i\geq 0$ for all $i$. It then immediately follows that $\mu_\ell>0$ for some single $l\in\{0,\ldots,L-1\}$ while $\mu_k = 0$ for all $k\neq \ell$, since otherwise $\mu_\ell \mu_k >0$ for some $\ell,k$, which is a contradiction to~\eqref{eq:mu products equation}. Consequently, we have that $\mu_\ell = L$ for some $\ell\in\{0,\ldots,L-1\}$ (see the left equation in~\eqref{eq:mu_i equations}), and 
\begin{equation}
\operatorname{Circulant}\{1,e^{\imath \varphi_1^{'}},\ldots,e^{\imath \varphi_{L-1}^{'}}\} = {f_\ell} {f_\ell}^*, 
\end{equation}
which implies that $\widetilde{X} \in \Omega(\hat{\Sigma}_x)$ (see also~\eqref{eq:Omega elements with DFT vectors expression}), and hence $\widetilde{X}$ solves Problem~\ref{problem:Sigma_x diagonal phase retrieval problem}.

Last, for the case of $1<r<\sqrt{L}$ under Condition~\ref{cond:low-rank recovery}, we refer the reader to the derivation in Section~\ref{section: step 2}, which provides a complete proof for this case through the derivation of the procedure for solving Problem~\ref{problem:PSD circulant phase programming}, and whose results are summarized in Proposition~\ref{prop:low-rank procedure derivation summary}.

\section{Proof of Proposition~\ref{prop:P_y and T_y ambiguity full rank}} \label{appendix:proof of positive Sigma identifiability for P_y and T_y}
Let us take $X \in\mathbb{C}^{L\times L}$ as
\begin{equation}
X = \hat{\Sigma}_x \odot \operatorname{Circulant}\{1,e^{\imath \varphi_1},1,\ldots,1,e^{-\imath \varphi_1}\}.
\end{equation}
More specifically, we have
\begin{equation}
X[k_1,k_2] =
\begin{dcases}
\hat{\Sigma}_x[k_1,k_2] \cdot e^{\imath \varphi_1}, & \operatorname{mod}(k_2-k_1,L) = 1, \\
\hat{\Sigma}_x[k_1,k_2] \cdot e^{-\imath \varphi_1}, & \operatorname{mod}(k_2-k_1,L) = -1, \\
\hat{\Sigma}_x[k_1,k_2], & \text{otherwise}.
\end{dcases}
\end{equation}
Clearly, $X$ follows the form of~\eqref{eq:P_y and T_y circulant phase ambiguities} in Lemma~\ref{lem:P_y and T_y consequence}, hence $X$ satisfies the equations
\begin{equation}
P_y = \mathcal{P}\{X\}, \qquad T_y = \mathcal{T}\{X\}.
\end{equation}
Moreover, $X$ is Hermitian, and
\begin{equation}
\left\Vert \hat{\Sigma}_x -X \right\Vert_{F} \underset{\varphi_1\rightarrow 0}{\longrightarrow} 0.
\end{equation}
Since the eigenvalues of a square matrix depend continuously on its elements (theorem 2.4.9.2 in~\cite{horn2012matrix}), we also have that
\begin{equation}
\left\vert \lambda_L - \lambda_{\min} \{ X \} \right\vert \underset{\varphi\rightarrow 0}{\longrightarrow} 0,
\end{equation}
where $\lambda_{\min} \{ X \}$ stands for the smallest eigenvalue of $X$.
Because $\lambda_L > 0$ when $r=L$, there exists a sufficiently small $\epsilon>0$ such that if $0<|\varphi_1| \leq \epsilon$ then $\lambda_{\min} \{ X \} > 0$, and consequently $X$ is PSD. 
However, it is evident that $X \notin \Omega(\hat{\Sigma}_x)$, which concludes the proof.

\section{Justification of~\eqref{eq:P_y expression diagonals} and~\eqref{eq:T_y expression diagonals}} \label{appendix:Justification of expression for P_y and T_y with diagonals}
Since we want to account for an arbitrary noise variance $\sigma^2$, whereas Proposition~\ref{prop:f_2,f_4,f_6 explicit form} considers explicitly the noiseless case $\sigma=0$, we introduce a certain update which places us in the noiseless setting and allows us to use Proposition~\ref{prop:f_2,f_4,f_6 explicit form}. Note that according to the definition of $y$ in~\eqref{eq:MRFA model def}, $y$ admits the same distribution as $R_s\{x + \eta \}$ (since the distribution of the noise $\eta$ is invariant to the operation $R_s$), and consequently, $\hat{y}$ from~\eqref{eq:MRFA_model_def_Fourier} admits the same distribution as $\operatorname{diag}(f_s) (\hat{x} + \hat{\eta})$. Therefore, we can absorb the noise variance $\sigma^2$ into the main diagonal of $\hat{\Sigma}_x$. That is, with some abuse of notation, we update $\hat{\Sigma}_x$ according to
\begin{equation}
\hat{\Sigma}_x \longleftarrow \hat{\Sigma}_x + \sigma^2 I_L, \label{eq:noise absob update}
\end{equation}
and then omit the noise vector $\hat{\eta}$ (from~\eqref{eq:MRFA_model_def_Fourier}) entirely. 
This update places us in the noiseless setting of $\sigma=0$ in~\eqref{eq:MRFA_model_def_Fourier} (after fixing $\hat{\Sigma}_x$ according to~\eqref{eq:noise absob update}) where the power spectrum and the trispectrum are determined solely by $\hat{\Sigma}_x$ according to Proposition~\ref{prop:f_2,f_4,f_6 explicit form}. 
Then, taking $d_m$ according to~\eqref{eq:d_m def} and applying Proposition~\ref{prop:f_2,f_4,f_6 explicit form} gives~\eqref{eq:P_y expression diagonals} and~\eqref{eq:T_y expression diagonals}.

\section{The trispectrum for the real-valued case} \label{appendix:trispectrum for real-valued}
This proof follows very closely with the proof in Appendix~\ref{Appendix: Proof of f_2,f_4 explicit form}. Analogously to the proof in Appendix~\ref{Appendix: Proof of f_2,f_4 explicit form}, we first consider the case of $r=L$ (i.e. $\lambda_i>0$ for all $i$) so we may claim that $x$ is normally-distributed and use standard results on the moments of the normal distribution. We then extend our result to any $r<L$ by a continuity argument. Consider the case of $r=L$. Then, $x\sim\mathcal{N}(0,\Sigma_x)$, and according to Isserlis' formula~\cite{isserlis1918formula} (for computing the moments of the zero-mean multivariate normal distribution) we have that
\begin{align}
\mathbb{E}\left[ \hat{x}[k_1] \hat{x}[k_2] \hat{x}[k_3] \hat{x}[k_4]\right] &= \mathbb{E}\left[ \hat{x}[k_1] \hat{x}[k_2]\right] \mathbb{E}\left[ \hat{x}[k_3] \hat{x}[k_4]\right] + \mathbb{E}\left[ \hat{x}[k_1] \hat{x}[k_3]\right]\mathbb{E}\left[ \hat{x}[k_2] \hat{x}[k_4]\right] + \mathbb{E}\left[ \hat{x}[k_1] \hat{x}[k_4]\right] \mathbb{E}\left[ \hat{x}[k_2] \hat{x}[k_3]\right] \nonumber \\
&= \hat{\Sigma}_x[k_1,-k_2] \hat{\Sigma}_x[k_3,-k_4] + \hat{\Sigma}_x[k_1,-k_3] \hat{\Sigma}_x[k_2,-k_4] + \hat{\Sigma}_x[k_1,-k_4] \hat{\Sigma}_x[k_2,-k_3],
\end{align}
where we used the fact that $x$ is real-valued, hence $\hat{x}[k] = \overline{\hat{x}[-k]}$ and thus $\mathbb{E}\left[ \hat{x}[k_1] \hat{x}[k_2]\right] = \mathbb{E}[ \hat{x}[k_1] \overline{\hat{x}[-k_2]}] = \hat{\Sigma}_x[k_1,-k_2]$.
Therefore, it follows that
\begin{align}
M^{(4)}_{\hat{x}}[k_1,k_2,k_3,k_4] &= \mathbb{E}\left[ \hat{x}[k_1] \hat{x}[-k_2] \hat{x}[k_3] \hat{x}[-k_4]\right] \nonumber \\ 
&= \hat{\Sigma}_x[k_1,k_2] \hat{\Sigma}_x[k_3,k_4] + \hat{\Sigma}_x[k_1,-k_3] \hat{\Sigma}_x[-k_2,k_4] + \hat{\Sigma}_x[k_1,k_4] \hat{\Sigma}_x[-k_2,-k_3] \nonumber \\
&= \hat{\Sigma}_x[k_1,k_2] \overline{\hat{\Sigma}_x[k_4,k_3]} + \hat{\Sigma}_x[k_3,k_2] \overline{\hat{\Sigma}_x[k_4,k_1]} + \hat{\Sigma}_x[k_1,-k_3] \overline{\hat{\Sigma}_x[k_2,-k_4]}, \label{eq:M_4 real expression Sigma_hat}
\end{align}
where we used the observation that $\hat{\Sigma}_x[k_1,k_2] = \overline{\hat{\Sigma}_x[-k_1,-k_2]} = \hat{\Sigma}_x[-k_2,-k_1]$ for any $k_1,k_2$, since $\hat{\Sigma}_x$ is Hermitian. Taking $d_m$ according to~\eqref{eq:d_m def}, and using $G_m = d_m d_m^*$ in~\eqref{eq:M_4 real expression Sigma_hat} gives
\begin{align}
T_y[k_1,k_1+m,k_2+m] &= M^{(4)}_{\hat{y}}[k_1,k_1+m,k_2+m,k_1-(k_1-m)+(k_2-m)] \nonumber \\
&= M^{(4)}_{\hat{x}}[k_1,k_1+m,k_2+m,k_1-(k_1-m)+(k_2-m)] \nonumber \\
&= \hat{\Sigma}_x[k_1,k_1+m] \overline{\hat{\Sigma}_x[k_2,k_2+m]} 
+ \hat{\Sigma}_x[k_2+m,k_1+m] \overline{\hat{\Sigma}_x[k_2,k_1]}  \nonumber \\
&\quad + \hat{\Sigma}_x[k_1,-k_2-m)] \overline{\hat{\Sigma}_x[k_1+m,-k_2)]} \nonumber \\
&= \hat{\Sigma}_x[k_1,k_1+m] \overline{\hat{\Sigma}_x[k_2,k_2+m]} 
+ \hat{\Sigma}_x[k_2,k_1] \overline{\hat{\Sigma}_x[k_1+m,k_2+m]} \nonumber \\
&\quad + \hat{\Sigma}_x[-k_2,k_1+m] \overline{\hat{\Sigma}_x[-k_2-m,k_1] } \nonumber \\
&= G_m[k_1,k_2] + G_{k_2-k_1}[k_1,k_1+m] + G_{k_1+k_2+m}[-k_2,-k_2-m],
\end{align}
where we used the fact that $\hat{\Sigma}_x$ is Hermitian.
Last, a continuity argument (repeating the argument at the end of Appendix~\ref{Appendix: Proof of f_2,f_4 explicit form}) extends the above result to the case of an arbitrary $r<L$.

\section{Proof of Theorem~\ref{thm:consistency of step 1} } \label{appendix: Proof of consistency of step 1}

The following lemma establishes that when $\widetilde{T}_y = T_y$ and $\widetilde{P}_y = P_y$ (i.e. when $N\rightarrow \infty$), then~\eqref{eq:optim for diagonals} admits a unique minimizer, which is equal to~\eqref{eq:G_m def}.
\begin{lem} \label{lem:optim for diagonals clean case unique min}
Suppose that $\widetilde{T}_y = T_y$, $\widetilde{P}_y = P_y$, and assume that $|\hat{\Sigma}_x[i,j]|>0$ for all $i,j$. If $\{G_m^\star\}_{m=1}^{L-1}$ is a minimizer of~\eqref{eq:optim for diagonals}, then $G_m^\star = G_m = d_m\cdot d_m^*$ for $m=1,\ldots,L-1$.
\end{lem}
\begin{proof}
Since~\eqref{eq:optim for diagonals} is convex, all minimizers attain the same objective value, which is zero since $\{G_m\}_{m=1}^{L-1}$ is a minimizer. Therefore, we have that 
\begin{equation}
{T}_y[k_1,k_1+m,k_2+m] = G^\star_{k_2-k_1}[k_1,k_1+m] + G^\star_m[k_1,k_2], \label{eq:T and G_star equations}
\end{equation}
for all indices $m,k_1,k_2$.
Considering the case $m=0$, and observing that $G_0^\star =G_0$ (since~\eqref{eq:optim for diagonals} enforces ${G}^{\star}_0 = \widetilde{P}_y \widetilde{P}_y^T = P_y P_y^T = G_0$), we have
\begin{equation}
G^\star_{k_2-k_1}[k_1,k_1] = {T}_y[k_1,k_1,k_2] - G_0[k_1,k_2] = G_{k_2-k_1}[k_1,k_1]. \label{eq:G_m main diagonal equality}
\end{equation}
Hence, the main diagonal of $G_m^{\star}$ is equal to the main diagonal of $G_m$, for $m=1,\ldots,L-1$.
Next, consider the case $k_2-k_1=m$, for which we get from~\eqref{eq:T and G_star equations}
\begin{equation}
2G_m^\star[k_1,k_1+m] = T_y[k_1,k_1+m,k_1+2m] = 2G_m[k_1,k_1+m]. \label{eq:G_m m'th diagonal equality}
\end{equation}
Therefore, we conclude that the $m$'th diagonal (with circulant wrapping) of $G_m^{\star}$ is equal to the $m$'th diagonal of $G_m$, for $m=1,\ldots,L-1$. Up to this point, we established that $G_m^{\star}$ and $G_m$ agree on two diagonals (their main diagonal and their $m$'th diagonal) for every $m$. Now, we turn to show that if $G_m^\star \succeq 0$, then~\eqref{eq:G_m main diagonal equality} and~\eqref{eq:G_m m'th diagonal equality} imply that $G_m^\star = G_m$ (i.e. $G_m^{\star}$ and $G_m$ agree on all diagonals).
Let us define the matrix
\begin{equation}
\hat{G}_m^\star = \operatorname{diag}\{d_m\}^{-1} G_m^\star \operatorname{diag}\{d_m^*\}^{-1}, \label{eq:G_m_hat def}
\end{equation}
where $d_m$ is from~\eqref{eq:d_m def}, and~\eqref{eq:G_m_hat def} is well defined since $|d_m[k]|>0$ for all $m,k$ from the assumptions of the lemma.
Since $\{G_m^\star\}_{m=1}^{L-1}$ is a minimizer of~\eqref{eq:optim for diagonals} then $G_m^\star \succeq 0$, and by~\eqref{eq:G_m_hat def} also $\hat{G}_m^\star \succeq 0$ (due to Sylvester's Inertia theorem).
Then, since $G_m = d_m d_m^*$, and the fact that $G_m^{\star}$ and $G_m$ have the same values on their main and $m$'th diagonals, it follows that 
\begin{equation}
\hat{G}_m^\star[k,k] = 1, \quad \hat{G}_m^\star[k,k+m] = 1,
\end{equation}
for all indices $m,k$.
Now, since $\hat{G}_m^\star$ is positive semidefinite with unit diagonal, it can take the role of a correlation matrix of a random vector. In particular, let us consider a random vector $z_m\in\mathbb{R}^L$, with $z_m[k]~\sim\mathcal{N}(0,\hat{G}_m^\star[k,k])$ for $k=0,\ldots,L-1$, noting that 
\begin{equation}
\mathbb{E}|z_m[k]|^2=1, \qquad \mathbb{E}\left[ z_m[k]z_m^*[k+m] \right]=1.
\end{equation}
Fixing $m=1$, the above relations imply that $z_1[k]$ and $z_1[k+1]$ are perfectly correlated normal variables with unit variances, and hence linearly dependent (almost surely) with 
\begin{equation}
z_1[k] = z_1[k+1],
\end{equation}
for $k=0,\ldots,L-1$ (since $\mathbb{E}\vert z[k] - z[k+1]\vert^2=0)$. Therefore, it follows that $z_1[0] = z_1[1] = \ldots = z_1[L-1]$ (almost surely), and $\hat{G}_1^\star = \mathbb{E}[z_1 z_1^*]$ must be of rank one with
\begin{equation}
\hat{G}_1^\star = \mathbf{1} \cdot \mathbf{1}^T, \label{eq:G_1_hat rank one}
\end{equation}
where $\mathbf{1}$ denotes an $L\times 1$ vector of ones. 
From~\eqref{eq:G_1_hat rank one} and~\eqref{eq:G_m_hat def} it then follows that
\begin{equation}
{G}_1^\star = d_1 d_1^* = G_1.
\end{equation}
Using the above result for $\hat{G}_1^\star$ together with~\eqref{eq:T and G_star equations} provides us with an additional equation on the diagonals of $G_m^\star$, namely
\begin{equation}
G^\star_{k_2-k_1}[k_1,k_1+1] = {T}_y[k_1,k_1+1,k_2+1] - G_1[k_1,k_2] = G_{k_2-k_1}[k_1,k_1+1], \label{eq:G_m first diagonal equality}
\end{equation}
and hence, using~\eqref{eq:G_m_hat def} again,
\begin{equation}
\hat{G}_m^\star[k,k+1] = 1,
\end{equation}
for $k=0,\ldots,L-1$ and $m=2,\ldots,L-1$.
Therefore, we established that $G_m^{\star}$ and $G_m$ agree on their main and first diagonals for every $m$. We can now repeat our previous arguments of the case of $m=1$ (using $\hat{G}^*_1 = \mathbb{E}[z_1 z_1^*]$) for $m=2,\ldots,L-1$, resulting in
\begin{equation}
z_m[k] = z_m[k+1],
\end{equation}
for $k=0,\ldots,L-1$ and $m=2,\ldots,L-1$, almost surely. Therefore,
\begin{equation}
\hat{G}_m^\star = \mathbf{1} \cdot \mathbf{1}^T,
\end{equation}
and consequently
\begin{equation}
{G}_m^\star = d_m d_m^* = G_m,
\end{equation}
for all $m=2,\ldots,L-1$.
\end{proof}

The next lemma establishes that problem~\eqref{eq:optim for diagonals} is robust to errors in the estimation of $P_y$ and $T_y$.
\begin{lem}[Stability of~\eqref{eq:optim for diagonals}] \label{lem: stability of optim for diagonals}
Suppose that $|\hat{\Sigma}_x[i,j]|>0$ for all $i,j\in\{0,\ldots,L-1\}$. If $\widetilde{T}_y\underset{N\rightarrow\infty, \; \text{a.s.}}{\longrightarrow} T_y$ and $\widetilde{P}_y\underset{N\rightarrow\infty, \; \text{a.s.}}{\longrightarrow} P_y$ (element-wise), then
\begin{align}
\widetilde{G}_m \underset{N\rightarrow\infty, \; \text{a.s.}}{\longrightarrow} G_m,
\end{align}
for $m=1,\ldots,L-1$.
\end{lem}
\begin{proof}
For convenience, we formulate a problem equivalent to~\eqref{eq:optim for diagonals} in matrix-vector notation. Let ${t}\in\mathbb{C}^{L^3}$, $\widetilde{t}\in\mathbb{C}^{L^3}$, $g_0\in\mathbb{C}^{L^2}$, $\widetilde{g}_0\in\mathbb{C}^{L^2}$, $g\in\mathbb{C}^{L^3-L^2}$, $\widetilde{g}\in\mathbb{C}^{L^3-L^2}$ be vectors obtained from vectorizing $T_y$, $\widetilde{T}_y$, $G_0$, $\widetilde{G}_0 := \widetilde{P}_y \widetilde{P}_y^T$, $\{G_m\}_{m=1}^{L-1}$, and $\{\widetilde{G}_m\}_{m=1}^{L-1}$, respectively.
Then, the set of equations~\eqref{eq:T_y expression G_m}
\begin{equation}
{T}_y[k_1,k_1+m,k_2+m] = G_m[k_1,k_2] + G_{k_2-k_1}[k_1,k_1+m],
\end{equation}
for $k_1,k_2,m = 0,\ldots,L-1$, can be written in matrix form as
\begin{equation}
t = A_0 g_0 + A g, \label{eq: trispectrum equations matrix form}
\end{equation}
where $A_0\in\mathbb{R}^{L^3\times L^2}$, $A\in\mathbb{R}^{L^3\times (L^3-L^2)}$ are suitable matrices (whose exact expressions are not important for this proof).
Next, we define the following functions:
\begin{align}
\widetilde{J}(g^{'}) := \left\Vert \widetilde{t} - A_0 \widetilde{g}_0 - A g^{'} \right\Vert, \\
{J}(g^{'}) := \left\Vert {t} - A_0 {g}_0 - A g^{'} \right\Vert,
\end{align}
for $g^{'}\in\mathbb{C}^{L^3-L^2}$ obtained by vectorizing $\{G^{'}_{m}\}_{m=1}^{L-1}$ from~\eqref{eq:optim for diagonals}, hence satisfying the semidefinite constraints associated with $G^{'}_{m} \succeq 0$ for $m=1,\ldots,L-1$.
Recall from~\eqref{eq:optim for diagonals} that $\widetilde{g}$ is a minimizer of $\widetilde{J}$.
Therefore,
\begin{equation}
\widetilde{J}(\widetilde{g}) \leq \widetilde{J}(g) = \left\Vert \widetilde{t} - A_0 \widetilde{g}_0 - A g \right\Vert,
\end{equation}
which together with~\eqref{eq: trispectrum equations matrix form} gives
\begin{equation}
\widetilde{J}(\widetilde{g}) \leq \left\Vert \widetilde{t} - A_0 \widetilde{g}_0 - A g \right\Vert = \left\Vert \widetilde{t}-t - A_0 \left( \widetilde{g}_0 - g_0\right) \right\Vert. \label{eq:J upper bound}
\end{equation}
On the other hand, by the reverse triangle inequality it follows that
\begin{align}
\widetilde{J}(\widetilde{g}) &= \left\Vert \widetilde{t} - A_0 \widetilde{g}_0 - A \widetilde{g} \right\Vert = \left\Vert t - A_0 {g}_0 - A \widetilde{g} +(\widetilde{t}-t)  - A_0 (\widetilde{g}_0 - {g}_0 )  \right\Vert \nonumber \\
&\geq \left\vert \left\Vert t - A_0 {g}_0 - A \widetilde{g} \right\Vert - \left\Vert \widetilde{t}-t - A_0 \left( \widetilde{g}_0 - g_0\right) \right\Vert \right\vert, \label{eq:J lower bound}
\end{align}
and by combining~\eqref{eq:J lower bound} and~\eqref{eq:J upper bound} we have
\begin{equation}
J(\widetilde{g}) = \left\Vert t - A_0 {g}_0 - A \widetilde{g} \right\Vert \leq 2 \left\Vert \widetilde{t}-t - A_0 \left( \widetilde{g}_0 - g_0\right) \right\Vert  \leq 2\left\Vert \widetilde{t}-t \right\Vert + 2\left\Vert A_0 \right\Vert \cdot \left\Vert \widetilde{g}_0 - g_0 \right\Vert.
\end{equation}
Therefore, if $\widetilde{T}_y \underset{N\rightarrow\infty, \; \text{a.s.}}{\longrightarrow} T_y$, $\widetilde{P}_y \underset{N\rightarrow\infty, \; \text{a.s.}}{\longrightarrow} P_y$, we have that $\left\Vert \widetilde{t}-t \right\Vert \underset{N\rightarrow\infty, \; \text{a.s.}}{\longrightarrow} 0$, $\left\Vert \widetilde{g}_0 - g_0 \right\Vert \underset{N\rightarrow\infty, \; \text{a.s.}}{\longrightarrow} 0$, and thus
\begin{equation}
J( \widetilde{g}) \underset{N\rightarrow\infty, \; \text{a.s.}}{\longrightarrow} 0. \label{eq:J converges to zero}
\end{equation}
Last, since $J$ is a non-negative and convex function over a convex domain with a unique minimizer (Lemma~\ref{lem:optim for diagonals clean case unique min}), it follows from~\eqref{eq:J converges to zero} that (see Corollary 27.2.2 in~\cite{rockafellar1970convex})
\begin{equation}
\widetilde{g} \underset{N\rightarrow\infty, \; \text{a.s.}}{\longrightarrow} g,
\end{equation}
or equivalently
\begin{equation}
\{ \widetilde{G}_m \}_{m=1}^{L-1} \underset{N\rightarrow\infty, \; \text{a.s.}}{\longrightarrow} \{ G_m \}_{m=1}^{L-1}.
\end{equation}
\end{proof}

Since $\widetilde{P}_y$ and $\widetilde{T}_y$ (of~\eqref{eq: P_y estimator} and~\eqref{eq: T_y estimator}) are consistent estimators for $P_y$ and $T_y$, respectively, we have that
\begin{equation}
\widetilde{P}_y \underset{N\rightarrow\infty, \; \text{a.s.}}{\longrightarrow} P_y, \qquad \widetilde{T}_y \underset{N\rightarrow\infty, \; \text{a.s.}}{\longrightarrow} T_y.
\end{equation}
Therefore, by Lemma~\ref{lem: stability of optim for diagonals} it follows that
\begin{equation}
\widetilde{G}_m \underset{N\rightarrow\infty, \; \text{a.s.}}{\longrightarrow} G_m = d_m d_m^*, \label{eq:G_m_tilde convergence}
\end{equation}
for $m=1,\ldots,L-1$. 
Note that $G_m$ is of rank one, with leading eigenvector $d_m/\Vert d_m \Vert$ and leading eigenvalue
\begin{equation}
\lambda_{\max} \{ G_m \} = \left\Vert d_m \right\Vert^2 > 0.
\end{equation}
Recall from~\eqref{eq:d_m estimator} that 
\begin{equation}
\widetilde{d}_m = \sqrt{\widetilde{\mu}_1^{(m)}} \widetilde{u}_1^{(m)},
\end{equation}
where $\widetilde{\mu}_1^{(m)}$ is the leading eigenvalue of $\widetilde{G}_m$ and $\widetilde{u}_1^{(m)}$ is its corresponding eigenvector.
Classical results in matrix perturbation theorey establish that $\widetilde{d}_m$ converges to $d_m$ almost surely. In particular, the Davis-Kahan theorem~\cite{yu2014useful,davis1970rotation} asserts that
\begin{equation}
\left\Vert \widetilde{d}_m/\Vert \widetilde{d}_m \Vert - e^{\imath \varphi_m} d_m/\Vert d_m \Vert \right\Vert = \left\Vert \widetilde{u}_1^{(m)} - e^{\imath \varphi} {u}_1^{(m)} \right\Vert \leq \frac{2\left\Vert \widetilde{G}_m - G_m \right\Vert_F}{\left\Vert d_m \right\Vert^2}, \label{eq:Davis-Kahan perturbation bound}
\end{equation}
for some $\varphi_m\in [0,2\pi)$, where ${u}_1^{(m)}$ is the leading eigenvector of $G_{m}$, and by Weyl~\cite{weyl1912asymptotische} 
\begin{equation}
\left\vert \Vert \widetilde{d}_m \Vert^2 - \Vert d_m \Vert^2 \right\vert = \left\vert \widetilde{\mu}_1^{(m)} - {\mu}_1^{(m)} \right\vert \leq \left\Vert \widetilde{G}_m - G_m \right\Vert_F, \label{eq:Weyl eigenvalue perturbation bound}
\end{equation}
where ${\mu}_1^{(m)}$ is the largest eigenvalue of $G_m$ (corresponding to ${u}_1^{(m)}$).
Hence, by~\eqref{eq:G_m_tilde convergence}, \eqref{eq:Davis-Kahan perturbation bound} and~\eqref{eq:Weyl eigenvalue perturbation bound} it follows that
\begin{equation}
\min_{\varphi\in [0,2\pi)}\left\Vert \widetilde{d}_m - e^{\imath \varphi} d_m\right\Vert \underset{N\rightarrow\infty, \; \text{a.s.}}{\longrightarrow} 0, \label{eq:d_m_tilde convergence}
\end{equation}
for $m=1,\ldots,L-1$. Last, by the definition of $\widetilde{C}_x$ in~\eqref{eq:C_x_tilde def}
\begin{align}
&\min_{\varphi_1,\ldots,\varphi_{L-1}\in [0,2\pi)}\left\Vert \widetilde{C}_x - \hat{\Sigma}_x \odot \operatorname{Circulant}\{1,e^{\imath \varphi_1},\ldots,e^{\imath \varphi_{L-1}}\} \right\Vert_F^2 \nonumber \\
&= \left\Vert \widetilde{P}_y - \sigma^2 - (d_0 - \sigma^2) \right\Vert^2 +  \sum_{m=1}^{L-1}\min_{\varphi_m\in [0,2\pi)}\left\Vert \widetilde{d}_m - e^{\imath \varphi_m} d_m \right\Vert^2 \nonumber \\
&= \left\Vert \widetilde{P}_y - P_y \right\Vert^2 +  \sum_{m=1}^{L-1}\min_{\varphi_m\in [0,2\pi)}\left\Vert \widetilde{d}_m - e^{\imath \varphi_m} d_m \right\Vert^2
\underset{N\rightarrow\infty, \; \text{a.s.}}{\longrightarrow} 0,
\end{align}
where we used the fact that $\hat{\Sigma}_x[k,k] = d_0[k] - \sigma^2 = P_y[k] - \sigma^2$ (see~\eqref{eq:d_m def} and~\eqref{eq:P_y expression diagonals}),~\eqref{eq:d_m_tilde convergence}, and the fact that $\widetilde{P}_y \underset{N\rightarrow\infty, \; \text{a.s.}}{\longrightarrow} P_y$, which concludes the proof.

\section{Proof of Lemma~\ref{lem:H_ii properties}} \label{appendix:proof of properties of H_ii}
First,~\eqref{eq:H_ii expression} follows from~\eqref{eq:X def} since a circulant matrix is invariant to $R_{i,i}$ (for any $i$), hence
\begin{equation}
\operatorname{Circulant} \{ [1, e^{\imath \varphi_1},\ldots, e^{\imath \varphi_{L-1}} ] \} \odot \overline{R_{i,i}\{ \operatorname{Circulant} \{ [1, e^{\imath \varphi_1},\ldots, e^{\imath \varphi_{L-1}} ] \} \} } = \mathbf{1}_{L\times L},
\end{equation}
where $\mathbf{1}_{L\times L}$ is a $L\times L$ matrix of ones. Second, the fact that $H_{i,i}$ is Hermitian follows from
\begin{align}
\overline{H_{i,i}[k,m]} &= \overline{\hat{\Sigma}_x[k,m]} \odot {R_{i,i}\{\hat{\Sigma}_x\}[k,m]} = \hat{\Sigma}_x[m,k] \odot \hat{\Sigma}_x[\operatorname{mod} (k-i,L),\operatorname{mod} (m-i,L)] 
\nonumber \\ 
&= \hat{\Sigma}_x[m,k] \odot \overline{\hat{\Sigma}_x[\operatorname{mod} (m-i,L),\operatorname{mod} (k-i,L)]} = \hat{\Sigma}_x[m,k] \odot {R_{i,i}\{\hat{\Sigma}_x\}[m,k]} = H_{i,i}[m,k],
\end{align}
where we used~\eqref{eq:H_ii expression} and~\eqref{eq:R_ij def}. Third, by~\eqref{eq:H_ii expression}, $H_{i,i}$ is PSD since the Hadamard product of two PSD matrices is also PSD due to the Schur product theorem (see Theorem 5.2.1 in~\cite{roger1994topics}). Last,~\eqref{eq:H_ii rank bound} is due to a well-known bound on the rank of the Hadamard product (see Theorem 5.1.7 in~\cite{roger1994topics}).

\section{Proof of Lemma~\ref{lem:S properties}} \label{appendix:Proof of Lemma S properties}
By the definition of the $L\times L$ blocks of $S$ in~\eqref{eq:Kron product block formula}, we have that
\begin{equation}
S^{(i,j)}[k_1,k_2] = S[k_1 +iL,k_2+jL] = \widetilde{X}[k_1,k_2] \overline{\widetilde{X}[\operatorname{mod}(k_1-i,L),\operatorname{mod}(k_2-j,L)]}. \label{eq:S explitic form}
\end{equation}
It is easy to verify from~\eqref{eq:S explitic form} that $S$ is Hermitian if $\widetilde{X}$ is Hermitian (this follows immediately from interchanging $i$ with $j$, and $k_1$ with $k_2$).
Next, a key observation for this proof is that $S$ is similar to the matrix $\overline{\widetilde{X}} \otimes \widetilde{X}$, where $\otimes$ is the Kronecker product. This is due to the fact that
\begin{equation}
\left\{ \overline{\widetilde{X}} \otimes \widetilde{X} \right\}[k_1+i L,k_2 +j L] = \widetilde{X}[k_1,k_2] \cdot \overline{\widetilde{X}[i,j]},
\end{equation}
and hence $S$ can be transformed into $\overline{\widetilde{X}} \otimes \widetilde{X}$ by an appropriate permutation of its rows and columns. Specifically, we can write 
\begin{equation}
S[k_1+i L,k_2 +j L] = \left\{ \overline{\widetilde{X}} \otimes \widetilde{X} \right\}[k_1+L \cdot \operatorname{mod}(k_1-i,L),k_2 +L \cdot \operatorname{mod}(k_2-j,L)],
\end{equation}
and it follows that there exists a permutation matrix $P\in\mathbb{R}^{L^2\times L^2}$ such that
\begin{equation}
P S P^T = \overline{\widetilde{X}} \otimes \widetilde{X}.
\end{equation}
It is well-known that the eigenvalues of $\overline{\widetilde{X}} \otimes \widetilde{X}$ are given by the pair-wise products between the eigenvalues of $\widetilde{X}$ and the eigenvalues of $\overline{\widetilde{X}}$ (see~\cite{roger1994topics}). Hence, if $\widetilde{X}$ is Hermitian and PSD, then $\overline{\widetilde{X}} \otimes \widetilde{X}$ is also Hermitian and PSD, and consequently so is $S$ by its similarity to $\overline{\widetilde{X}} \otimes \widetilde{X}$.

\section{Proof of Theorem~\ref{thm:consistency of step 2, direct approach}} \label{appendix:Proof of consistency of step 2}
By Theorem~\ref{thm:consistency of step 1}, we can write
\begin{equation}
\widetilde{C}_x = X^{(N)} + E^{(N)}, \label{eq:C_tilde_x est err}
\end{equation}
where $X^{(N)}$ is equal to $X$ from~\eqref{eq:X def} but with angles $\varphi_1,\ldots,\varphi_{L-1}$ that may depend on $N$, and $\Vert E^{(N)} \Vert_F \underset{N\rightarrow\infty, \; \text{a.s.}}{\longrightarrow} 0$. For simplicity of presentation, we omit the superscript $(\cdot)^{(N)}$ in $X^{(N)}$ and $E^{(N)}$ from all subsequent derivations. Let $W$ of~\eqref{eq:A_mathcal_tilde def} be the matrix constructed from $X$ as described in Section~\ref{section: step 2}, and let $\widetilde{W}$ be a matrix analogous to $W$ when using $\widetilde{C}_x$ instead of $X$.
We now analyze the different quantities involved in the construction of $\widetilde{W}$. By~\eqref{eq:H_ij_tilde def} we have that
\begin{equation}
\widetilde{H}_{i,j} = \widetilde{C}_x \odot \overline{R_{i,j}\{\widetilde{C}_x\}} = H_{i,j} + E \odot \overline{R_{i,j}\{{X}\}} + X \odot \overline{R_{i,j}\{E\}} + E \odot \overline{R_{i,j}\{E\}}.
\end{equation}
Using the bound $\Vert A\odot B\Vert_F\leq \Vert A\Vert_F \Vert B\Vert_F$ (see~\cite{roger1994topics}), it follows that
\begin{equation}
\left\Vert \widetilde{H}_{i,j} - H_{i,j} \right\Vert_F \leq 2 \left\Vert E \right\Vert_F \left\Vert \hat{\Sigma}_x \right\Vert_F + \left\Vert E \right\Vert_F^2 \underset{N\rightarrow\infty, \; \text{a.s.}}{\longrightarrow} 0, \label{eq:H_ij_tilde convergence}
\end{equation}
where we used the fact that $\left\Vert X \right\Vert_F = \left\Vert \hat{\Sigma}_x \right\Vert_F$.
Since $H_{i,i}$ admits $r^2$ distinct and non-zero eigenvalues (see the assumptions of Theorem~\ref{thm:consistency of step 2, direct approach}), we have from the Davis-Kahan Theorem~\cite{yu2014useful,davis1970rotation} that
\begin{equation}
\min_{\vartheta_1,\ldots,\vartheta_{r^2}\in [0,2\pi)}\left\Vert \widetilde{V}^{(i)} - V^{(i)} \cdot \operatorname{diag}\{e^{\imath \vartheta_1},\ldots,e^{\imath \vartheta_{r^2}}\} \right\Vert_F \underset{N\rightarrow\infty, \; \text{a.s.}}{\longrightarrow} 0, \label{eq:V_tilde convergence}
\end{equation}
where $V^{(i)}$ and $\widetilde{V}^{(i)}$ are $L\times r^2$ matrices whose columns are the first $r^2$ eigenvectors (corresponding to the largest eigenvalues) of $H_{i,i}$ and $\widetilde{H}_{i,i}$, respectively. Next, define the matrices $\widetilde{Z}^{(i)}\in\mathbb{C}^{L^2\times r^4}$ and the vectors $\widetilde{M}^{(i)}_m \in \mathbb{C}^{L^2}$, for $i,m\in \{0,\ldots,L-1 \}$, analogously to ${Z}^{(i)}$ and ${M}^{(i)}_m$ of~\eqref{eq:M_tilde def}, by
\begin{equation}
\begin{aligned} 
\widetilde{Z}^{(i)} &= \overline{\widetilde{V}^{(i+1)}} \otimes \widetilde{V}^{(i)},  \\
\widetilde{M}^{(i)}_m &= \operatorname{vec}\left\{\widetilde{H}_{i,i+1}\right\} \odot \operatorname{vec}\left\{\operatorname{Circulant}\{\mathbf{e}_m\} \right\}, \label{eq:Z_tilde and M_tilde def}
\end{aligned}
\end{equation}
where $\otimes$ is the Kronecker product, $\mathbf{e}_m$ is the $m$'th indicator vector (with a single value of $1$ at the $m$'th entry), and $\operatorname{vec}\{\cdot\}$ is the operation of vectorizing a matrix by stacking its columns on top of each other. From~\eqref{eq:V_tilde convergence}, we can write
\begin{equation}
\widetilde{V}^{(i)} = V^{(i)} \cdot \operatorname{diag}\{e^{\imath \vartheta^{(i)}_1},\ldots,e^{\imath \vartheta^{(i)}_{r^2}}\} + E_V^{(i)},
\end{equation}
where $\Vert E_V^{(i)} \Vert_F \underset{N\rightarrow\infty, \; \text{a.s.}}{\longrightarrow} 0$ (and the angles $\vartheta^{(i)}_1,\ldots,\vartheta^{(i)}_{r^2}$ depend on $N$). Therefore, we can write
\begin{align}
&\widetilde{Z}^{(i)} = \overline{\widetilde{V}^{(i+1)}} \otimes \widetilde{V}^{(i)} = \left( \overline{V^{(i+1)}} \cdot \operatorname{diag}\{e^{-\imath \vartheta^{(i+1)}_1},\ldots,e^{-\imath \vartheta^{(i+1)}_{r^2}}\} + \overline{E_V^{(i+1)}}\right) \otimes \left( V^{(i)} \cdot \operatorname{diag}\{e^{\imath \vartheta^{(i)}_1},\ldots,e^{\imath \vartheta^{(i)}_{r^2}}\} + E_V^{(i)}\right)  \nonumber \\
&= \left( \overline{V^{(i+1)}} \cdot \operatorname{diag}\{e^{-\imath \vartheta^{(i+1)}_1},\ldots,e^{-\imath \vartheta^{(i+1)}_{r^2}}\} \right) \otimes \left( V^{(i)} \cdot \operatorname{diag}\{e^{\imath \vartheta^{(i)}_1},\ldots,e^{\imath \vartheta^{(i)}_{r^2}}\}\right) + \nonumber \\
&  \overline{E_V^{(i+1)}} \otimes \left( V^{(i)} \cdot \operatorname{diag}\{e^{\imath \vartheta^{(i)}_1},\ldots,e^{\imath \vartheta^{(i)}_{r^2}}\}\right)  + \left( \overline{V^{(i+1)}} \cdot \operatorname{diag}\{e^{-\imath \vartheta^{(i+1)}_1},\ldots,e^{-\imath \vartheta^{(i+1)}_{r^2}}\} \right) \otimes E_V^{(i)}  + \overline{E_V^{(i+1)}} \otimes E_V^{(i)}   \nonumber \\
&= \left( \overline{V^{(i+1)}} \otimes V^{(i)}\right) \cdot \left(  \operatorname{diag}\{e^{-\imath \vartheta^{(i+1)}_1},\ldots,e^{-\imath \vartheta^{(i+1)}_{r^2}}\}  \otimes \operatorname{diag}\{e^{\imath \vartheta^{(i)}_1},\ldots,e^{\imath \vartheta^{(i)}_{r^2}}\} \right) + \nonumber \\
&\overline{E_V^{(i+1)}} \otimes \left( V^{(i)} \cdot \operatorname{diag}\{e^{\imath \vartheta^{(i)}_1},\ldots,e^{\imath \vartheta^{(i)}_{r^2}}\}\right)  + \left( \overline{V^{(i+1)}} \cdot \operatorname{diag}\{e^{-\imath \vartheta^{(i+1)}_1},\ldots,e^{-\imath \vartheta^{(i+1)}_{r^2}}\} \right) \otimes E_V^{(i)}  + \overline{E_V^{(i+1)}} \otimes E_V^{(i)}  \nonumber \\
&= Z^{(i)} \cdot  \operatorname{diag}\left\{[e^{-\imath \vartheta^{(i+1)}_1},\ldots,e^{-\imath \vartheta^{(i+1)}_{r^2}}] \otimes [e^{\imath \vartheta^{(i)}_1},\ldots,e^{\imath \vartheta^{(i)}_{r^2}}] \right\}  + \nonumber \\
&\overline{E_V^{(i+1)}} \otimes \left( V^{(i)} \cdot \operatorname{diag}\{e^{\imath \vartheta^{(i)}_1},\ldots,e^{\imath \vartheta^{(i)}_{r^2}}\}\right)  + \left( \overline{V^{(i+1)}} \cdot \operatorname{diag}\{e^{-\imath \vartheta^{(i+1)}_1},\ldots,e^{-\imath \vartheta^{(i+1)}_{r^2}}\} \right) \otimes E_V^{(i)}  + \overline{E_V^{(i+1)}} \otimes E_V^{(i)} ,
\end{align}
where we used the mixed-product property of the Kronecker product (i.e. $(A\cdot B)\otimes (C\cdot D) = (A\otimes C)\cdot (B\otimes D)$, see~\cite{roger1994topics}). By using the bound $\Vert A\otimes B\Vert_F\leq \Vert A\Vert_F \Vert B\Vert_F$ (see~\cite{roger1994topics}) together with $\Vert E_V^{(i)} \Vert_F \underset{N\rightarrow\infty, \; \text{a.s.}}{\longrightarrow} 0$, we have that
\begin{equation}
\min_{\gamma^{(i)}_1,\ldots,\gamma^{(i)}_{r^4}\in [0, 2\pi)}\left\Vert \widetilde{Z}^{(i)} - {Z}^{(i)} \cdot \operatorname{diag}\{e^{\imath \gamma^{(i)}_1},\ldots,e^{\imath \gamma^{(i)}_{r^4}}\}\right\Vert_F \underset{N\rightarrow\infty, \; \text{a.s.}}{\longrightarrow} 0, \label{eq:Z_tilde convergence}
\end{equation}
for $i\in \{0,\ldots,L-1 \}$.
Next, from~\eqref{eq:H_ij_tilde convergence} and~\eqref{eq:Z_tilde and M_tilde def} it immediately follows that
\begin{equation}
\left\Vert \widetilde{M}^{(i)}_m - {{M}^{(i)}_m}\right\Vert \underset{N\rightarrow\infty, \; \text{a.s.}}{\longrightarrow} 0, \label{eq:M_tilde convergence}
\end{equation}
for $i,m\in \{0,\ldots,L-1 \}$.
Recall that $\widetilde{W}$ is formed according to the right-hand side of~\eqref{eq:A_mathcal_tilde def} when replacing ${{Z}^{(i)}}$ and ${{M}^{(i)}_m}$ with $\widetilde{Z}^{(i)}$ and $\widetilde{M}^{(i)}_m$, respectively. Therefore, by~\eqref{eq:Z_tilde convergence} and~\eqref{eq:M_tilde convergence} it follows that
\begin{equation}
\min_{\gamma^{(0)}_1,\ldots,\gamma^{(0)}_{r^4},\ldots,\gamma^{(L-1)}_1,\ldots,\gamma^{(1)}_{r^4}\in [0, 2\pi)}\left\Vert \widetilde{W} - W \cdot \operatorname{diag}\{\mathbf{1}_L^T,e^{\imath \gamma^{(0)}_1},\ldots,e^{\imath \gamma^{(0)}_{r^4}},{\ldots},e^{\imath \gamma^{(L-1)}_1},\ldots,e^{\imath \gamma^{(L-1)}_{r^4}}\}\right\Vert_F \underset{N\rightarrow\infty, \; \text{a.s.}}{\longrightarrow} 0, \label{eq:A_tilde convergence}
\end{equation}
where $\mathbf{1}_L$ is a column vector of $L$ ones.
Recall that $\mathcal{V}$ and $\widetilde{\mathcal{V}}$ are the right singular vectors of $W$ and $\widetilde{W}$ corresponding to their smallest singular values, respectively. Let $\mathcal{V}_L\in\mathbb{C}^L$ and $\widetilde{\mathcal{V}}_L\in\mathbb{C}^L$ be the first $L$ elements of $\mathcal{V}$ and $\widetilde{\mathcal{V}}$, respectively. Note that the matrices $W$ and $W \cdot \operatorname{diag}\{\mathbf{1}_L^T,e^{\imath \gamma^{(0)}_1},\ldots,e^{\imath \gamma^{(0)}_{r^4}},\ldots,e^{\imath \gamma^{(L-1)}_1},\ldots,e^{\imath \gamma^{(L-1)}_{r^4}}\}$ agree in their singular values and in the first $L$ entries of their singular vectors. Therefore, from~\eqref{eq:A_tilde convergence} it follows that
\begin{equation}
\min_{\varphi\in [0, 2\pi)}\left\Vert \widetilde{\mathcal{V}}_L - {\mathcal{V}}_L \cdot e^{\imath \varphi} \right\Vert \underset{N\rightarrow\infty, \; \text{a.s.}}{\longrightarrow} 0, \label{eq:nu_tilde first L entries convergence}
\end{equation}
where we used the Davis-Kahan Theorem~\cite{yu2014useful,davis1970rotation} together with the fact that the smallest singular value of $W$ is zero while its second-smallest singular value is strictly positive (resulting in a spectral gap for the smallest singular value). 
Since the elements of $\mathcal{V}_L$ are bounded away from zero (they have magnitudes of $1$ according to~\eqref{eq:V_mathcal_tilde expression}), from~\eqref{eq:nu_tilde first L entries convergence} it follows that
\begin{equation}
\min_{\varphi\in [0, 2\pi)}\sum_{k=0}^{L-1}\left\vert \operatorname{arg}\{\widetilde{\mathcal{V}}[k]\} - \operatorname{arg}\{{\mathcal{V}}[k]\} - \varphi \right\vert^2 \underset{N\rightarrow\infty, \; \text{a.s.}}{\longrightarrow} 0. \label{eq:nu_tilde arg convergence}
\end{equation}
Let $\widetilde{\widetilde{\varphi}}_m$ be analogous to $\widetilde{\varphi}_m$ from~\eqref{eq:step 2 angle estimation procedure} when replacing ${\mathcal{V}}$ with $\widetilde{\mathcal{V}}$, and fixing $\widetilde{\widetilde{\varphi}}_0 = 0$, i.e.
\begin{equation}
\widetilde{\widetilde{\varphi}}_m = - \sum_{\ell=1}^m \operatorname{arg}\left\{ \widetilde{\mathcal{V}}[\ell] \right\} + \frac{m}{L}\sum_{\ell=0}^{L-1} \operatorname{arg}\left\{ \widetilde{\mathcal{V}}[\ell] \right\} + \frac{2\pi k m}{L}. \label{eq:varphi_tilde_tilde procedure}
\end{equation}
Then, it follows from~\eqref{eq:varphi_tilde_tilde procedure},~\eqref{eq:step 2 angle estimation procedure} and~\eqref{eq:nu_tilde arg convergence} that
\begin{align}
&\left\vert \widetilde{\widetilde{\varphi}}_m - \widetilde{{\varphi}}_m \right\vert = \left\vert \sum_{\ell=1}^m \left( \operatorname{arg}\left\{ \widetilde{\mathcal{V}}[\ell] \right\} - \operatorname{arg}\left\{ {\mathcal{V}}[\ell] \right\}\right) - \frac{m}{L}\sum_{\ell=0}^{L-1} \left( \operatorname{arg}\left\{ \widetilde{\mathcal{V}}[\ell] \right\} - \operatorname{arg}\left\{ {\mathcal{V}}[\ell] \right\} \right) \right\vert \nonumber \\
&= \left\vert \sum_{\ell=1}^m \left( \operatorname{arg}\left\{ \widetilde{\mathcal{V}}[\ell] \right\} - \operatorname{arg}\left\{ {\mathcal{V}}[\ell] \right\} -\varphi \right) + m\varphi - \frac{m}{L}\sum_{\ell=0}^{L-1} \left( \operatorname{arg}\left\{ \widetilde{\mathcal{V}}[\ell] \right\} - \operatorname{arg}\left\{ {\mathcal{V}}[\ell] \right\} -\varphi \right) - m\varphi \right\vert \nonumber \\
&= \min_{\varphi\in [0, 2\pi)} \left\vert \sum_{\ell=1}^m \left( \operatorname{arg}\left\{ \widetilde{\mathcal{V}}[\ell] \right\} - \operatorname{arg}\left\{ {\mathcal{V}}[\ell] \right\} -\varphi \right) - \frac{m}{L}\sum_{\ell=0}^{L-1} \left( \operatorname{arg}\left\{ \widetilde{\mathcal{V}}[\ell] \right\} - \operatorname{arg}\left\{ {\mathcal{V}}[\ell] \right\} -\varphi \right) \right\vert \nonumber \\
&\leq \min_{\varphi\in [0, 2\pi)} \sum_{\ell=1}^m \left\vert \operatorname{arg}\left\{ \widetilde{\mathcal{V}}[\ell] \right\} - \operatorname{arg}\left\{ {\mathcal{V}}[\ell] \right\} -\varphi \right\vert + \min_{\varphi\in [0, 2\pi)} \frac{m}{L}\sum_{\ell=0}^{L-1} \left\vert \operatorname{arg}\left\{ \widetilde{\mathcal{V}}[\ell] \right\} - \operatorname{arg}\left\{ {\mathcal{V}}[\ell] \right\} -\varphi  \right\vert 
\underset{N\rightarrow\infty, \; \text{a.s.}}{\longrightarrow} 0, 
\end{align}
for $m=1,\ldots,L-1$, which together with Proposition~\ref{prop:low-rank procedure derivation summary} implies that
\begin{equation}
\min_{\ell=0,\ldots,L-1} \sum_{m=1}^{L-1}\left\vert \widetilde{\widetilde{\varphi}}_m - \varphi_m - \frac{2\pi \ell m}{L}  \right\vert^2 \underset{N\rightarrow\infty, \; \text{a.s.}}{\longrightarrow} 0. \label{eq:varphi_tilde_tilde convergence}
\end{equation}
With some abuse of notation, let $\widetilde{\hat{\Sigma}}_x$ be as in~\eqref{eq:covariance estimator with phase} with $\widetilde{\widetilde{\varphi}}_m$ replacing $\widetilde{\varphi}_m$. Then, employing~\eqref{eq:C_tilde_x est err} and~\eqref{eq:X def} yields
\begin{align}
\widetilde{\hat{\Sigma}}_x &= \widetilde{C}_x \odot \operatorname{Circulant} \{ [1, e^{-\imath \widetilde{\widetilde{\varphi}}_1},\ldots, e^{-\imath \widetilde{\widetilde{\varphi}}_{L-1}} ] \} \nonumber \\ 
&= X \odot \operatorname{Circulant} \{ [1, e^{-\imath \widetilde{\widetilde{\varphi}}_1},\ldots, e^{-\imath \widetilde{\widetilde{\varphi}}_{L-1}} ] \} + E \odot \operatorname{Circulant} \{ [1, e^{-\imath \widetilde{\widetilde{\varphi}}_1},\ldots, e^{-\imath \widetilde{\widetilde{\varphi}}_{L-1}} ] \} \nonumber \\ 
&=\hat{\Sigma}_x \odot \operatorname{Circulant} \{ [1, e^{\imath (\varphi_1 - \widetilde{\widetilde{\varphi}}_1)},\ldots, e^{\imath (\varphi_{L-1} - \widetilde{\widetilde{\varphi}}_{L-1})} ] \} + E \odot \operatorname{Circulant} \{ [1, e^{-\imath \widetilde{\widetilde{\varphi}}_1},\ldots, e^{-\imath \widetilde{\widetilde{\varphi}}_{L-1}} ] \}.
\end{align}
Hence, by the above equation together with~\eqref{eq:varphi_tilde_tilde convergence},~\eqref{eq:Omega elements with DFT vectors expression},  and the fact that $\Vert E \Vert_F\underset{N\rightarrow\infty, \; \text{a.s.}}{\longrightarrow} 0$, we have
\begin{equation}
\min_{\ell=0,\ldots,L-1}\left\Vert \widetilde{\hat{\Sigma}}_x - \operatorname{diag}(f_\ell)\cdot \hat{\Sigma}_x \cdot \operatorname{diag}(f_\ell^*) \right\Vert_F \underset{N\rightarrow\infty, \; \text{a.s.}}{\longrightarrow} 0,
\end{equation}
and~\eqref{eq:Sigma_x estimation consistency} in Theorem~\ref{thm:consistency of step 2, direct approach} follows in a straightforward manner (using $\widetilde{\Sigma}_x = F^* \widetilde{\hat{\Sigma}}_x F$).

\end{appendices}

\bibliographystyle{plain}
\bibliography{mybib}

\end{document}